\documentclass{amsart}

\usepackage{amsmath, amssymb, amsthm, marginnote, ltxcmds, adjustbox, stmaryrd, graphicx, bbold, extarrows}
\usepackage{esvect}

\usepackage{mathtools, stackengine, changepage, ragged2e}

\usepackage{todonotes}
\usepackage{float}
\usepackage{multirow}
\usepackage{aliascnt}

\definecolor{greencolor}{rgb}{0,0.45,0}
\definecolor{ucphcolor}{rgb}{0.517,0.016,0.016}
\definecolor{OliveGreen}{rgb}{0,0.6,0}
\usepackage[colorlinks,
            citecolor=greencolor,
            linkcolor=ucphcolor,
            urlcolor=greencolor, unicode]{hyperref}
\usepackage[nameinlink,capitalise,noabbrev]{cleveref}

\usepackage{subfig}

\usepackage[backend=biber, style=alphabetic, maxnames=10, maxalphanames=3, minalphanames=3]{biblatex}
\usepackage{tikz}
\usetikzlibrary{cd}
\usetikzlibrary{arrows, arrows.meta, positioning, calc}
\tikzcdset{arrow style=tikz, diagrams={>={Straight Barb[scale=0.8]}}}

\tikzstyle{arrow} = [-{Straight Barb[scale=0.8]}, line width=0.2mm]
\tikzset{
math to/.tip={Glyph[glyph math command=rightarrow]},
loop/.tip={Glyph[glyph math command=looparrowleft, swap]},
}

\usepackage{enumitem}
{\end{enumerate}%
}

\usepackage[
	letterpaper,
	twoside=false,
	textheight=22cm,
	textwidth=14.4cm,
	marginparsep=0.75cm,
	marginparwidth=2.5cm,
	heightrounded,
	centering
]{geometry}

\usepackage{lineno}

\usepackage{contour}
\usepackage{ulem}

\contourlength{0.5pt}

\newcommand{\myuline}[1]{%
  \uline{\phantom{#1}}%
  \llap{\contour{white}{#1}}%
}

\makeatletter
\newcommand*{\saved@myuline}{}
\let\saved@myuline\myuline

\newcommand*{\mathuline}{%
  \mathpalette{\math@myuline\saved@myuline}%
}
\newcommand*{\math@myuline}[3]{%
  \mbox{#1{$#2#3\m@th$}}%
}

\renewcommand*{\myuline}{%
  \relax  
  \ifmmode
    \expandafter\mathuline
  \else
    \expandafter\saved@myuline
  \fi
}
\makeatother

\usepackage{stmaryrd}

\Crefname{prop}{Proposition}{Propositions}
\Crefname{lem}{Lemma}{Lemmas}
\Crefname{cor}{Corollary}{Corollaries}
\Crefname{thm}{Theorem}{Theorems}
\Crefname{alphThm}{Theorem}{Theorems}

\Crefname{defn}{Definition}{Definitions}
\Crefname{notation}{Notation}{Notations}
\Crefname{cons}{Construction}{Constructions}
\Crefname{rmk}{Remark}{Remarks}
\Crefname{obs}{Observation}{Observations}
\Crefname{trick}{Trick}{Tricks}
\Crefname{warning}{Warning}{Warnings}
\Crefname{conj}{Conjecture}{Conjectures}
\Crefname{assump}{Assumption}{Assumptions}
\Crefname{recollect}{Recollection}{Recollections}
\Crefname{terminology}{Terminology}{Terminologies}
\Crefname{conditionsec}{Condition}{Conditions}
\Crefname{fact}{Fact}{Facts}
\Crefname{exercise}{Exercise}{Exercises}
\Crefname{setting}{Setting}{Settings}

\Crefname{question}{Question}{Questions}
\Crefname{example}{Example}{Examples}

\Crefname{figure}{Figure}{Figures}

\crefformat{equation}{(#2#1#3)}
\crefformat{section}{\S#2#1#3}
\crefmultiformat{section}{\S\S#2#1#3}{and~#2#1#3}{, #2#1#3}{, and~#2#1#3}

\newaliascnt{thm}{subsubsection}
\newtheorem{thm}[thm]{Theorem}
\aliascntresetthe{thm}
\newaliascnt{prop}{subsubsection}
\newtheorem{prop}[prop]{Proposition}
\aliascntresetthe{prop}
\newaliascnt{lem}{subsubsection}
\newtheorem{lem}[lem]{Lemma}
\aliascntresetthe{lem}
\newaliascnt{cor}{subsubsection}
\newtheorem{cor}[cor]{Corollary}
\aliascntresetthe{cor}
\newaliascnt{variation}{subsubsection}

\aliascntresetthe{variation}

\newtheorem{alphThm}{Theorem}

\newcommand{\neutralize}[1]{\expandafter\let\csname c@#1\endcsname\count@}
\makeatother

\newtheorem*{thm*}{Theorem}
\newtheorem*{prop*}{Proposition}
\newtheorem*{lem*}{Lemma}
\newtheorem*{cor*}{Corollary}
  
\newtheorem{alphConj}{Conjecture}



\newtheorem{alphCor}[alphThm]{Corollary}



\newtheorem{alphProp}{Proposition}



\newaliascnt{thmsec}{section}

\aliascntresetthe{thmsec}
\newaliascnt{propsec}{thmsec}

\aliascntresetthe{propsec}
\newaliascnt{lemsec}{thmsec}

\aliascntresetthe{lemsec}
\newaliascnt{corsec}{thmsec}

\aliascntresetthe{corsec}

\theoremstyle{definition}
\newtheorem*{defn*}{Definition}
\newaliascnt{defn}{subsubsection}
\newtheorem{defn}[defn]{Definition}
\aliascntresetthe{defn}
\newaliascnt{cons}{subsubsection}
\newtheorem{cons}[cons]{Construction}
\aliascntresetthe{cons}
\newaliascnt{nota}{subsubsection}
\newtheorem{nota}[nota]{Notation}
\aliascntresetthe{nota}
\newaliascnt{ass}{subsubsection}

\aliascntresetthe{ass}
\newaliascnt{recollect}{subsubsection}
\newtheorem{recollect}[recollect]{Recollections}
\aliascntresetthe{recollect}
\newaliascnt{terminology}{subsubsection}
\newtheorem{terminology}[terminology]{Terminology}
\aliascntresetthe{terminology}
\newaliascnt{assump}{subsubsection}

\aliascntresetthe{assump}
\newaliascnt{conj}{subsubsection}

\aliascntresetthe{conj}
\newaliascnt{idea}{subsubsection}

\aliascntresetthe{idea}
\newaliascnt{question}{subsubsection}

\aliascntresetthe{question}
\newaliascnt{assumption}{subsubsection}

\aliascntresetthe{assumption}
\newaliascnt{example}{subsubsection}
\newtheorem{example}[example]{Example}
\aliascntresetthe{example}
\newaliascnt{setting}{subsubsection}
\newtheorem{setting}[setting]{Setting}
\aliascntresetthe{setting}

\newaliascnt{rmk}{subsubsection}
\newtheorem{rmk}[rmk]{Remark}
\aliascntresetthe{rmk}
\newaliascnt{obs}{subsubsection}
\newtheorem{obs}[obs]{Observation}
\aliascntresetthe{obs}
\newaliascnt{fact}{subsubsection}
\newtheorem{fact}[fact]{Fact}
\aliascntresetthe{fact}
\newaliascnt{note}{subsubsection}

\aliascntresetthe{note}
\newaliascnt{warning}{subsubsection}
\newtheorem{warning}[warning]{Warning}
\aliascntresetthe{warning}
\newaliascnt{ideas}{subsubsection}

\aliascntresetthe{ideas}
\newaliascnt{constr}{subsubsection}

\aliascntresetthe{constr}

\theoremstyle{definition}
\newaliascnt{defnsec}{thmsec}

\aliascntresetthe{defnsec}
\newaliascnt{questionsec}{thmsec}

\aliascntresetthe{questionsec}
\newaliascnt{conditionsec}{thmsec}

\aliascntresetthe{conditionsec}
\newaliascnt{terminologysec}{thmsec}

\aliascntresetthe{terminologysec}
\newaliascnt{conssec}{thmsec}

\aliascntresetthe{conssec}

\newcommand{\nrex}{^{n\mathrm{-ex}}}

\newcommand{\cn}{^{\mathrm{cn}}}
\newcommand{\boldTwocat}{\mathbb{C}\mathrm{AT}}
\newcommand{\boldPr}{\mathbb{P}\mathrm{r}}
\newcommand{\boldCat}{\mathbb{C}\mathrm{at}}
\newcommand{\X}{\mathcal{X}}
\newcommand{\cocartCat}{\mathrm{coCart}}
\newcommand{\THH}{\mathrm{THH}}
\newcommand{\cone}{^{\triangleright}}

\newcommand{\poset}{\mathrm{Pos}}
\newcommand{\tcocone}{^{\triangleright}}
\newcommand{\geomFix}[1]{\Phi^{\family}}

\newcommand{\Mot}{\mathrm{Mot}}
\newcommand{\loc}{^{\mathrm{loc}}}
\newcommand{\U}{\mathcal{U}}

\newcommand{\cofree}{\mathrm{cofree}}

\newcommand{\finiteSmall}{\mathrm{fin}}

\newcommand{\baseCat}{\mathcal{T}}

\newcommand{\distributeAll}[1]{^{\Delta_{#1}}}

\newcommand{\family}{\mathcal{F}}

\newcommand{\Tr}{\mathrm{Tr}}

\newcommand{\susps}{\Sigma^{\infty}}
\newcommand{\perfect}{^{\mathrm{perf}}}
\newcommand{\norm}{\mathrm{N}}
\newcommand{\calg}{\mathrm{CAlg}}
\newcommand{\spc}{\mathcal{S}}

\newcommand{\ind}{\mathrm{Ind}}
\DeclareMathOperator{\coind}{Coind}

\DeclareMathOperator{\cmonoid}{CMon}
\newcommand{\canonical}{\mathrm{can}}
\newcommand{\forget}{\mathrm{fgt}}
\newcommand{\borel}{\mathrm{Bor}}
\DeclareMathOperator{\loops}{\Omega^{\infty}}
\newcommand{\inclusion}{\mathrm{incl}}
\newcommand{\eval}{\mathrm{ev}}
\newcommand{\constant}{\operatorname{const}}
\newcommand{\cat}{\mathrm{Cat}}

\newcommand{\commmonoid}{\mathrm{CMon}}
\newcommand{\presentable}{\mathrm{Pr}}

\newcommand{\A}{\mathcal{A}}

\newcommand{\sC}{{\mathcal C}}
\newcommand{\D}{{\mathcal D}}

\newcommand{\B}{\mathcal{B}}
\newcommand{\op}{^{\mathrm{op}}}

\DeclareMathOperator{\mapsp}{map}
\newcommand{\sphere}{\mathbb{S}}

\newcommand{\sU}{{\mathcal U}}

\newcommand{\E}{\mathcal{E}}

\DeclareMathOperator{\res}{Res}
\DeclareMathOperator{\mackey}{Mack}
\DeclareMathOperator{\presheaf}{PSh}
\newcommand{\orbit}{\mathcal{O}}

\newcommand{\exact}{^{\mathrm{ex}}}
\newcommand{\spectra}{\mathrm{Sp}}
\newcommand{\finite}{\mathrm{Fin}}

\newcommand{\id}{\mathrm{id}}
\DeclareMathOperator{\map}{Map}

\DeclareMathOperator{\cofib}{cofib}
\DeclareMathOperator{\func}{Fun}

\newcommand{\module}{\mathrm{Mod}}
\newcommand{\unit}{\mathbb{1}}

\newcommand{\everythingAlgebra}{\mathbb{E}}
\newcommand{\perfectCat}{\mathrm{Perf}}
\newcommand{\algebraCategory}{\mathrm{Alg}}

\newcommand{\stmodSmall}{\mathrm{stmod}}
\newcommand{\proper}{\mathcal{P}}
\newcommand{\CAT}{\mathrm{CAT}}

\newcommand{\udl}[1]{\underline{{#1}}}

\def\colim{\qopname\relax m{colim}}

\DeclareMathOperator{\spancategory}{Span}

\newcommand{\presentableGstable}{\mathrm{Pr}_G^{\mathrm{st}}}

\newcommand{\presentablestable}{\mathrm{Pr}^{\mathrm{st}}}

\DeclareMathOperator{\const}{Const}

\DeclareMathOperator{\Endomorphism}{End}
\newcommand{\lmodule}{\mathrm{LMod}}

\newcommand{\one}{\mathbb{1}}

\newcommand{\induced}{\mathrm{Ind}}

\newcommand{\dual}{\mathrm{dual}}

\newcommand{\abs}{\mathrm{abs}}
\newcommand{\spl}{\mathrm{split}}
\newcommand{\mot}{\mathrm{mot}}
\newcommand{\yo}{\text{\usefont{U}{min}{m}{n}\symbol{'107}}}
\DeclareFontFamily{U}{min}{}
\DeclareFontShape{U}{min}{m}{n}{<-> dmjhira}{}
\newcommand{\yoneda}{\yo}

\newcommand{\arrdisp}{0.33ex}
\newcommand{\arrdisplacementsp}{0.72ex}

\newcommand{\ardis}{\ar@<\arrdisp>}
\newcommand{\ardissp}{\ar@<\arrdisplacementsp>}

\newcommand*\cocolon{%
        \nobreak
        \mskip6mu plus1mu
        \mathpunct{}%
        \nonscript
        \mkern-\thinmuskip
        {:}%
        \mskip2mu
        \relax
}



\title[{Equivariant  motives and  norms on algebraic K-theory}]{{Equivariant localizing motives and multiplicative norms on algebraic K-theory}}
\author{\textsc{Kaif Hilman}  \and \textsc{Maxime Ramzi}}

\date{\today}

\makeatletter
\patchcmd{\@setaddresses}{\indent}{\noindent}{}{}
\patchcmd{\@setaddresses}{\indent}{\noindent}{}{}
\patchcmd{\@setaddresses}{\indent}{\noindent}{}{}
\patchcmd{\@setaddresses}{\indent}{\noindent}{}{}
\makeatother

\address{Kaif Hilman, Mathematik Zentrum der Universität Bonn, Endenicher Allee 60, 53115 Bonn, Germany}
\email{kaif@math.uni-bonn.de}
\urladdr{https://sites.google.com/view/kaif-hilman/}

\address{Maxime Ramzi, FB Mathematik und Informatik, Universität Münster, Einsteinstra\ss e 62, Münster, Germany}
\email{mramzi@uni-muenster.de}
\urladdr{https://sites.google.com/view/maxime-ramzi-en}

\begin{document}

\begin{abstract}
 We construct multiplicative norms on  equivariant nonconnective algebraic $K$-theory for finite groups $G$. We also construct a genuine equivariant version of THH equipped with a Dennis trace map from K-theory  compatible with the multiplicative norms. To do so, we follow the general strategy of Blumberg--Gepner--Tabuada in the nonequivariant case by generalizing their category of localizing motives to the genuine equivariant context, building upon the theory of perfect $G$-stable categories from \cite{kaifNoncommMotives}. Crucially, we proceed using the recent perspective on noncommutative motives from \cite{MotLoc} which allows us to deal with non-exact functors on this category of motives. Together with an isotropy separation argument for equivariant cubes, we prove our main theorem that   norms  of stable categories preserve equivariant motivic equivalences. As an immediate consequence, we obtain  a unique equivariant multiplicative refinement of nonconnective algebraic $K$-theory.

From these constructions and results, we draw several applications, namely: (1) that the endofunctor of (equivariant) tensor powers on ordinary perfect stable categories preserve motivic equivalences; (2) that the multiplicative norms also preserve the additive motivic equivalences, thus yielding a motivic refinement of a result of \cite{elmantoHaugseng,cnossennorms} that connective algebraic K-theory admits multiplicative norms; (3) we construct a genuine equivariant version  of topological Hochschild homology equipped with a Dennis trace map that is compatible with multiplicative norms; and (4) we prove that every genuine $G$-spectrum is the K-theory of a perfect $G$-stable category.

\end{abstract}

\maketitle

\vspace{-9mm}

\tableofcontents

\section{Introduction}

Developments in recent decades have taught us that algebraic K-theory is best thought of as a homological invariant on stable categories, central to many parts of algebra, geometry, and topology.  Just as in the algebraic topology of spaces where cohomological methods gain their full strength in large part from multiplicative structures such as cup products and Steenrod powers, algebraic K-theory finds its deepest results in connection with analogous multiplicative structures.\footnote{For example, see the positive solution of the chromatic redshift conjecture for $\mathbb{E}_{\infty}$--rings by \cite{landMathewMeierTamme,CMMN2,nullstellensatz} or the new proof of B\"okstedt periodicity using polynomial functoriality in \cite{BGMN}.} The fundamental result in service of guaranteeing such structures, which was proved in the seminal work by Blumberg-Gepner-Tabuada  \cite{BGT,BGTMult} on noncommutative motives, is that the algebraic K-theory functor $K\colon \cat\perfect\rightarrow\spectra$ has a lax symmetric monoidal refinement. In particular, this enhances the spectrum $K(\sC)$ to an $\everythingAlgebra_{\infty}$-ring spectrum when $\sC$ is a symmetric monoidal stable category. But even more than that, the motivic perspective of the aforementioned work endows K-theory with a universal property. For instance, this ensures all the desired multiplicative structures in probing K-theory via trace methods.

The goal of this paper is to extend this discussion to the setting of genuine  equivariant stable homotopy theory for finite groups and capture the subtle multiplicative norm structures of Greenlees-May \cite{greenleesMayMU} and especially of the groundbreaking paper \cite{HHR} of Hill-Hopkins-Ravenel (cf. the introduction of \cite{kaifNoncommMotives} for motivations for such structures, for example, in trying to prove completion theorems for K-theoretic objects). As we shall see, our equivariant work will have consequences also to the nonequivariant theory. For the sake of a leaner mathematical narrative, we have deferred explaining the connections to previous work on this subject in the literature to the end of the introduction.

\addtocontents{toc}{\protect\setcounter{tocdepth}{1}} % Temporarily hide subsections in TOC

\subsection{The main results}
Let $G$ be a finite group. There is no shortage in  diversity as to what equivariant algebraic K-theory could mean.\footnote{See for example the end of the introduction in \cite{kaifNoncommMotives} for a general lay of the land together with precise references.} For us, we will follow the interpretation of \cite{barwick2} and mean by this the functor $\udl{K}\colon \mackey_G(\cat\perfect)\rightarrow \spectra_G\coloneqq \mackey_G(\spectra)$ obtained by taking $G$-Mackey functors $\mackey_G(-)$, in the sense of \cite{barwick1}, of the  \textit{nonconnective} algebraic K-theory functor $K\colon \cat\perfect\rightarrow \spectra$.

In fact, in order to phrase the structure of equivariant multiplicative norms purely $\infty$-categorically, we will need to work in the more highly structured setting of categories parametrized over the orbit category $\orbit_G$ of the group $G$ as introduced by \cite{expose1Elements,nardinThesis,nardinShah}. In this setup, the basic objects are so-called \textit{$G$-categories}, which are objects $\udl{\sC} =\{\sC_H\}_{H\leq G}$ in the category $\cat_G\coloneqq \func(\orbit_G\op,\cat)$ of $G$-categories. For the reader familiar with classical equivariant homotopy theory, this should indeed be viewed as a categorification of the notion of genuine $G$-spaces. 

Moreover, as we shall recall in \cref{section:elements}, there are also the equivariant analogues of stability and symmetric monoidal structures, termed unsurprisingly as $G$-stability and $G$-symmetric monoidal structures. For instance, just as a symmetric monoidal category is an object in $\cmonoid(\cat)$, the category of $\everythingAlgebra_{\infty}$-monoids in categories,  a $G$-symmetric monoidal category is an object in $\mackey_G(\cat)$, the $G$-Mackey functors in categories. The structures in the latter category should be thought of as an $\infty$-categorical axiomatization of the sorts of structures that were considered in \cite{HHR}. Intuitively, an object $\udl{\sC}\in \mackey_G(\cat)$ consists of the data of symmetric monoidal categories $\sC_H$ for each subgroup $H\leq G$, and additionally, for each inclusion of subgroups $H\leq K$, the datum of  symmetric monoidal functors $\res^K_H\colon \sC_K\rightarrow\sC_H$ and $\norm^K_H\colon \sC_H\rightarrow\sC_K$ called the restriction and multiplicative norm functors respectively, together with various compatibility coherences. The classical archetypal example to keep in mind here is the  tensor induction in representation theory which takes a $H$-representation $V$ and produces the $K$-representation $\bigotimes_{k\in K/H}kV$.

On the other hand, the notion of $G$-stability starts with ``fiberwise'' stability, namely the stability of each $\sC_H, H\leq G$, and is combined with $G$-semiadditivity. The latter notion axiomatizes the Wirthm\"uller isomorphism and is about admitting both induction and coinduction functors $\sC_H\to\sC_G$ for each subgroup $H$, defined as left and right adjoints to the restriction functor respectively, and a certain canonical map between them being an equivalence. The first-named author proved in \cite[Thm. B]{kaifNoncommMotives} that these $G$-stable $G$-categories form a nonfull subcategory $\cat\perfect_G$ of the category of $G$-Mackey functors of stable categories, where some previous notions of equivariant K-theory have been defined. These are in some sense the more natural/categorical Mackey functors, since their transfer structures arise from adjunctions rather than being abstract transfers. Furthermore, these in fact assemble to a $G$-symmetric monoidal $G$-category $\udl{\cat}\perfect_G$.

Having sketched the general language, we now come to the contents of this article. The key insight in the recent work \cite{MotLoc} of the second-named author with Sosnilo and Winges is that the category of localizing motives  as constructed by Blumberg-Gepner-Tabuada -- which admits a universal property for mapping out of via colimit-preserving functors per construction -- may alternatively be obtained by performing a Dwyer-Kan localization on the category $\cat\perfect$ by inverting an appropriate notion of motivic equivalences. 

One virtue of this observation is that the category of localizing motives  admits a new universal property for mapping out of  via not necessarily colimit-preserving functors. This opens up the possibility for constructing ``power operations'' at the level of localizing motives (which then descends to the level of algebraic K-theory by decategorification), and the main aim of the present work is to capitalise on this to deal with the multiplicative norms in the more refined setting of $G$-stable categories. In slightly more detail, we mimic \cite{MotLoc}  and introduce an analogous notion of finitary $G$-localizing invariants on the $G$-category $\udl{\cat}\perfect_G$, which in turn gives rise to a definition of  $G$-motivic equivalences, cf. \cref{defn:motivic_equivalence}. Using this, we then define the $G$-category $\udl{\Mot}\loc_G$ as a Dwyer-Kan localization of $\udl{\cat}\perfect_G$ against the $G$-motivic equivalences. Our main results are then the following pair of theorems:

\begin{alphThm}[cf. {\cref{cor:Nmweq}}]\label{alphThm:main}
    Let $H\leq K$ be an inclusion of subgroups of a finite group $G$. The multiplicative norm functor $\norm_H^K : \cat\perfect_H\to\cat\perfect_K$ sends $H$-motivic equivalences to $K$-motivic equivalences. 
\end{alphThm}

% As an immediate consequence of this result, we obtain the symmetric monoidal statement in the following collection of results:

\begin{alphThm}[cf. {\cref{thm:GMotGPres,thm:motnorm,prop:corep,cor:consKlaxG}}]\label{alphThm:K_is_G-laxsym}
    The $G$-category $\udl{\Mot}\loc_G$ is $G$-stable and $G$-presentable, and the Dwyer-Kan localization $\udl{\sU}\loc\colon \udl{\cat}\perfect_G\to \udl{\Mot}\loc_G$ is the universal finitary $G$-localizing invariant. Moreover, $\udl{\Mot}\loc_G$ attains a unique $G$-symmetric monoidal structure which refines $\udl{\sU}\loc$ to a $G$-symmetric monoidal functor.  
    
    Finally, equivariant nonconnective K-theory $\udl{K}\colon \udl{\cat}\perfect_G\subset \udl{\mackey}_G(\cat\perfect)\xrightarrow{\udl{K}} \myuline{\spectra}_G$ is corepresented by the tensor unit in $\udl{\Mot}\loc_G$ in that there is a factorization of $\udl{K}$ as \[\udl{\cat}\perfect_G\xrightarrow{\udl{\sU}\loc}\udl{\Mot}\loc_G\xrightarrow{\myuline{\mapsp}(\unit,-)}\myuline{\spectra}_G\] and thus $\udl{K}\colon \udl{\cat}\perfect_G\rightarrow\myuline{\spectra}_G$ refines to a $G$-lax symmetric monoidal functor.  Furthermore, the canonical map $(-)^\simeq \to \Omega^\infty \udl{K}_G$ has a unique lax $G$-symmetric monoidal refinement. 
\end{alphThm}

Therefore, for any equivariant $\infty$-operad $\mathcal O$ and any $\mathcal O\otimes\mathbb E_1$-ring $G$-spectrum $R\in\spectra_G$, the $G$-spectrum $\udl{K}(R)$ naturally attains an $\mathcal{O}$-ring $G$-spectrum structure.\footnote{This provides one articulation of an answer to Problem 1.2 of Blumberg in \cite{AIMproblemList}.}

Among other things, \cref{alphThm:main} allows for a detailed analysis of operations related to motivic tensor powers in the nonequivariant setting, which we highlight below. Now, let us  briefly comment on its proof. Just as in the nonequivariant setting, being a $G$-localizing invariant is defined in terms of sending Karoubi sequences (i.e. bifiber sequences) of $G$-categories to fiber sequences. Proving the theorem then essentially demands a study of how these exact sequences interact with the norm functors $\norm^K_H$. The difficulty here is that these norms are not linear functors, and so applying it to Karoubi sequences produces null sequences in $\udl{\cat}\perfect_G$ which may fail to be either fiber or cofiber sequences -- this is analogous to the module theoretic statement that if $A$ is a sub-$R$-module of $B$, then $B^{\otimes_R n}/A^{\otimes_R n}$ is not $(B/A)^{\otimes_R n}$. 

To make matters worse, the leftmost term in these normed cofiber sequences involve colimits over equivariant cubes, which are significantly trickier to handle than ordinary cubes. For a more verbose discussion of these problems, we refer the reader to the introduction of \cite{kaifNoncommMotives}. 

The key techniques we develop to deal with these two problems, which might be of independent utility, are then certain isotropy separation arguments performed at the level of categories to check fully faithfulness, cf. \cref{lem:fully_faithful_isotropy_separation}, as well as an equivariant cubical filtration, cf. \cref{prop:cubical_inclusions}. These methods permits us to give inductive proofs for the relevant fully faithfulness statements as well as perform an equivariant cubical devissage to handle the cubical terms in the cofiber sequences above.

\subsection{Applications}

We now turn to some consequences of the main results above. Our first application pertains to nonlinear operations on K-theory  which features both in the question of polynomial functoriality of $K$-theory à la \cite{BGMN} and in questions of categorifying commutative ring spectra following \cite{MotLoc}. As an immediate consequence of \cref{alphThm:main}, we obtain the following, which is part of the theory of Borel localizing motives recorded in \cref{subsec:borel_motives}. 

\begin{alphCor}[cf. {\cref{cor:tensor_power_preserve_moteq,prop:BGMN_key}}]\label{alphCor:symmetric_powers}
    Let $G$ be a finite group and $X$ a finite $G$-set. The functor $((-)^{\otimes X})_{hG}: \cat\perfect\to \cat\perfect$ preserves motivic equivalences, and therefore induces an endofunctor of $\Mot\loc$.  Consequently, for any natural number $n$, $((-)^{\otimes n})_{h\Sigma_n}$ preserves motivic equivalences, and therefore, the functor $\Gamma_n$ introduced in \cite[Cons. 4.16]{BGMN} also preserves motivic equivalences, and induces an endofunctor of $\Mot\loc$. 
\end{alphCor}

From the latter part of this statement, one can in fact directly recover the main theorem of \cite{BGMN} which extends the functoriality of the connective algebraic K-theory space functor to polynomial functors between stable categories, refining to the level of spaces many of the classical nonadditive operations on the K-theory groups. 

Before moving on to the next application, let us explain the essential content of this result and its proof. In the body of the paper, we will in fact prove something slightly stronger, i.e. that $(-)^{\otimes X}\colon \cat\perfect\rightarrow \func(BG,\cat\perfect)$ sends motivic equivalences to \textit{Borel} motivic equivalences, i.e. those morphisms which are inverted by all localizing invariants out of $\func(BG,\cat\perfect)$. Importantly, note that there are localizing invariants out of $\func(BG,\cat\perfect)$ which do not factor through the forgetful functor to $\cat\perfect$, for instance, the functor $K((-)_{hG})$. So even though the notion of Karoubi sequences in both $\cat\perfect$ and $\func(BG,\cat\perfect)$ only depend on the underlying sequence in $\cat\perfect$, Borel motivic equivalences in $\func(BG,\cat\perfect)$ are \textit{not} checked by passing to the underlying morphism in $\cat\perfect$.

From the point of view of giving a proof to such results, the reason it is not as straightforward as one might first surmise is the following: for two motivic equivalences $f\colon \A \rightarrow \B$ and $f'\colon\A'\rightarrow \B'$ in $\cat\perfect$, the easy reason that  $f\otimes f'$ is also a motivic equivalence is that we can factor it as $(f\otimes \id_{\B'})\circ(\id_{\A}\otimes f')$. In turn, it is easy to show that both $(f\otimes \id_{\B'}), (\id_{\A}\otimes f')$ are motivic equivalences because the tensor product preserves Karoubi sequences in each variable. However, in the equivariant context, one needs rather to argue that morphisms like $f^{\otimes G}\colon \A^{\otimes G}\rightarrow \B^{\otimes G}$ are (Borel) motivic equivalences. In this situation, the currying trick from before of factorizing the morphism fails because the factorization is not through $G$-equivariant morphisms, i.e. they are not morphisms in $\func(BG,\cat\perfect)$. This is the fundamental reason which necessitates a symmetric way of dealing with the problem via cubes as afforded by \cref{alphThm:main}.

\vspace{2mm}

Coming back to the genuine equivariant situation, notice that the main theorems have been about nonconnective algebraic K-theory, and one can also wonder about similar questions for the connective variant thereof or rather its associated category of motives. In this setting, the appropriate version of excision is the weaker requirement of being a $G$-splitting\footnote{Traditionally, splitting invariants are called additive invariants. We have opted to follow this new terminology from \cite{MotLoc}, see their Remark 6.4 for a justification.} invariant (\cref{defn:splitting_invariants}) rather than a $G$-localizing invariant. For this question, we use that nonconnective K-theory is a $G$-lax symmetric monoidal functor by \cref{alphThm:K_is_G-laxsym} and deduce the following connective analogue:

\begin{alphThm}[cf. {\cref{cor:norms_on_splitting_motives}}]\label{alphThm:splitting_motives}
    There is a  $G$-category $\udl{\Mot}^{\spl,\abs}_G$ which is $G$-additive, equipped with a map $\udl{\sU}^{\spl,\abs}\colon \udl{\cat}\perfect_G\to \udl{\Mot}^{\spl,\abs}_G$ witnessing it as the  universal  $G$-splitting invariant. Moreover, $\udl{\Mot}^{\spl,\abs}_G$ attains a unique $G$-symmetric monoidal structure which refines $\udl{\sU}^{\spl,\abs}$ to a $G$-symmetric monoidal functor.  
    
    Finally, equivariant connective K-theory $\udl{K}\cn\colon \udl{\cat}\perfect_G\subset \udl{\mackey}_G(\cat\perfect)\xrightarrow{\udl{K}\cn} \myuline{\spectra}\cn_G$ is corepresented by the tensor unit in $\udl{\Mot}^{\spl,\abs}_G$ and thus $\udl{K}\cn\colon \udl{\cat}\perfect_G\rightarrow\myuline{\spectra}\cn_G$ refines to a $G$-lax symmetric monoidal functor.
\end{alphThm}
Perhaps somewhat unexpectedly, we obtain \cref{alphThm:splitting_motives} indirectly \textit{from} \cref{alphThm:K_is_G-laxsym} rather than dealing with it directly even though the former statement about splitting motives seems superficially  a simpler situation than the latter about localizing motives. For a possible explanation of this, we refer the reader to \cref{rmk:counterintuitive_splitting}. Incidentally, this result also settles in the affirmative the question as to whether the ``pointwise'' and ``normed'' motives considered in \cite{kaifNoncommMotives} are equivalent.

\vspace{2mm}

Another consequence of this approach to producing norms is that it directly gives $K$-theory a universal property as a functor equipped with multiplicative norms. To illustrate the theoretical usefulness of this, we turn to an equivariant generalization of trace methods. More specifically, we leverage on the work of \cite{HSS} to construct an equivariant version $\udl{\THH}$ of topological Hochschild homology as a normed functor\footnote{This in particular answers Problem 1.6 of Merling in \cite{AIMproblemList}.} which is also a $G$-splitting invariant. By \cref{alphThm:K_is_G-laxsym}, we  therefore obtain: 

\begin{alphThm}[cf. {\cref{cor:dennis_trace}}]\label{alphThm:dennis_trace}
    There is an equivariant Dennis trace transformation $\udl{K}\to \udl{\THH}$ of $G$-lax symmetric monoidal functors $\udl{\cat}\perfect_G\rightarrow \myuline{\spectra}_G$.    In particular for any $G$-$\infty$-operad $\mathcal O$ and any $\mathbb E_1\otimes\mathcal O$-ring $G$-spectrum $R$, the Dennis trace map $K_G(R)\to \THH_G(R)$ can be canonically upgraded to an $\mathcal O$-ring map.  
\end{alphThm}

We expect this to be a useful fundamental piece of structural theory when running a trace method argument in the equivariant setting in the same way that having the structure of a ring map on the Dennis trace is an indispensable part of the nonequivariant story.  As an illustration that the $G$-lax symmetric monoidality of the trace map potentially captures interesting and nontrivial structures, we note in \cref{obs:factoring_trace_map} that for every finite group $G$, the usual trace map $K(\sphere)\rightarrow \THH(\sphere)=\sphere$ admits a factorization $K(\sphere)\rightarrow \Phi^G\udl{K}(\sphere_G)\rightarrow  \THH(\sphere)=\sphere$ where $\Phi^G$ is the geometric fixed point functor. As further examples, we also point out lifts of the proper Tate construction on the Dennis trace along the Nikolaus--Scholze Tate--Frobenius map in two situations in \cref{example:frobenius_lift_THH_borel,example:frobenius_lift_THH_normed_ring}.

As complements to this construction of equivariant THH, we also provide several basic formal properties about it, for example, that it interacts nicely with geometric fixed points and that it is computed as a free loop space when applied to equivariant local systems. We also use it to split off the $G$-suspension spectrum of a $G$-space from a version of equivariant A-theory, which we expect to coincide with the one of Malkiewich--Merling \cite{monaCaryCobordism}.

\vspace{2mm}

Finally, we forget all matters about multiplicative norms and give an application of the mere construction of the category of equivariant localizing motives and the Dennis trace map from \cref{alphThm:dennis_trace}. Namely, with all the machinery set up, we may directly copy the proof of \cite[Thm. 1.6]{MotLoc} to obtain the following genuine equivariant generalization of the main theorem of \cite{RSW0} that every spectrum is the algebraic K-theory of some stable category.

\begin{alphThm}[cf. {\cref{thm:every_spectrum_is_K-theory}}]\label{alphThm:every_spectrum_is_K-theory}
    Let $X\in\spectra_G$. Then there exists $\udl{\sC}\in\cat\perfect_G$ and an equivalence $\udl{K}(\udl{\sC})\simeq X\in\spectra_G$.
\end{alphThm}

One might be tempted to try to \textit{deduce} the statement above from the nonequivariant result as follows: in \cite{RSW0}, it was shown that in fact $K\colon \cat\perfect\rightarrow \spectra$ admits a functorial section $s$, and so one could try to deduce that this induces also a functorial section to $K\colon \mackey_G(\cat\perfect)\rightarrow\mackey_G(\spectra)=\spectra_G$. However, the functor $s$ in fact does \textit{not} preserve finite products, and so it does not induce a functor backwards on Mackey functors. As such, it seems to us necessary to set up the theory of equivariant localizing motives in order to obtain  \cref{alphThm:every_spectrum_is_K-theory}. Moreover, the theorem above  also gives a  stronger conclusion since it asserts that every genuine $G$-spectrum may be realised as the K--theory spectrum of an object in the nonfull subcategory $\cat\perfect_G\subset \mackey_G(\spectra)$.

\subsection{Relation to other work} We only mention here the types of equivariant K-theory that are most directly related to what we consider in this paper. For a survey of other kinds of equivariant K-theory, see for example the introduction to \cite{kaifNoncommMotives}.

As far as we are aware, the first work to construct the multiplicative norms on  algebraic K-theory is Elmanto-Haugseng \cite[Thm. 4.3.7]{elmantoHaugseng}. More specifically, using the formalism of bispans and taking the deep work on polynomial functoriality of K-theory by Barwick-Glasman-Mathew-Nikolaus \cite{BGMN} as the K-theoretic input, they enhanced the equivariant algebraic K-theory \textit{space} with the structure of multiplicative norms. Later on, Cnossen--Haugseng--Lenz--Linskens \cite[Thm. B]{cnossennorms} proved an abstract result about product-preserving presheaves on bispans  of finite $G$-sets which enhances Elmanto-Haugseng's construction from the K-theory space to the connective K-theory \textit{spectrum}. 

In a separate vein, the first-named author \cite{kaifNoncommMotives} enhanced the connective algebraic K-theory spectrum with the structure of multiplicative norms in the special case of finite 2-groups directly without recourse to \cite{BGMN}. The approach there is to construct the $G$-category of equivariant splitting motives and perform an explicit equivariant cubical analysis in the case of $G=C_2$. The present work therefore finds its precursor in  \cite{kaifNoncommMotives} and completes the approach therein. Just as in \textit{loc. cit.}, our approach does \textit{not} depend on \cite{BGMN} and can, in fact, essentially reprove this result. However, unlike \textit{loc. cit.}, our main object of study are equivariant \textit{localizing} motives and \textit{nonconnective} K-theory, which as far as we know, has not been treated before. We then use our knowledge in the nonconnective setting to deduce  the connective version also. Furthermore, by constructing the category of localizing motives, we place the story in the context of universal properties and thus allows us, among other things, to view the equivariant Dennis trace map as a normed map. 

Since we mentioned that we could use our methods to construct an equivariant Dennis trace, we also note that \cite{chan2025trace} already constructed an equivariant Dennis trace at the point-set level. However, their Dennis trace has target Borel equivariant $\THH$ as opposed to the genuine equivariant $\udl{\THH}$ mentioned in \cref{alphThm:dennis_trace}. It would be beneficial to clarify the relationship between our construction and their work\footnote{They write ``...this approach does not refine all the way to a map of genuine G-spectra. Since the equivariant algebraic K-theory of a G-ring spectrum depends only on the category of perfect modules over the underlying ring with G-action'', which is certainly not the case for our version of equivariant algebraic $K$-theory.}.  We also mention that, in forthcoming work, Marc Gotliboym will also give a different construction of the equivariant Dennis trace map and also treat equivariant cyclotomic structures, which we do not do here.

Finally, in the operator algebraic context, we were informed by Julian Kranz that Ulrich Bunke proved a result  analogous to \cref{alphThm:K_is_G-laxsym} in the context of $C^*$-algebras and operator $K$-theory, see \cite[Prop. 2.19]{bunke} (a precursor of this can also be found in Izumi's \cite[Theorem 2.1]{izumiflip}). The results of \cite{MotLoc} are what allow us to do similar kinds of arguments in the categorical context as in the $C^*$-algebraic context, where the KK-category has been known to be a Dwyer--Kan localization of the category of certain $C^*$-algebras for a longer time.

\subsection{Organization of the paper}

We lay out the parametrized language and ingredients needed in \cref{section:elements}. After recalling and  recording some basic elements in \cref{subsec:prelims}, including the equivariant cubes, we recall in \cref{subsec:catperf} the notion of $G$-perfect-stable categories from \cite{kaifNoncommMotives} as well as the categorical isotropy separation from \cite{PD1}. With this latter construction, we provide an isotropy separation principle to check fully faithfulness in \cref{subsec:fully_faithful}. 

Next, we begin our K-theoretic work proper in \cref{section:GMot}. We define the $G$-category of $G$-lozalizing motives in \cref{subsec:motloc} and prove a Blumberg--Gepner--Tabuada-style universal property in \cref{thm:GMotGPres}. We then come to the heart of the article in \cref{subsec:GSymMonofMotLoc} where we prove \cref{alphThm:main} by performing various cubical isotropy separation arguments. Following that, we show in \cref{subsec:corepresented_K-theory} that equivariant nonconnective K-theory is corepresented by the unit in equivariant localizing motives, and thus complete the proof of \cref{alphThm:K_is_G-laxsym}.

In  \cref{section:applications}, we provide four applications of the preceding results. Using \cref{alphThm:main}, we consider a variant called Borel localizing motives in \cref{subsec:borel_motives} and prove \cref{alphCor:symmetric_powers}. We then consider $G$-splitting motives in \cref{subsec:splitting_motives} and prove \cref{alphThm:splitting_motives} pertaining to \textit{connective} equivariant algebraic K-theory. In \cref{subsec:THH}, we construct an equivariant THH and use our main theorem to construct a normed Dennis trace map as in \cref{alphThm:dennis_trace}. Lastly, we prove \cref{alphThm:every_spectrum_is_K-theory} that every genuine $G$-spectrum is the  K-theory of some $G$-stable category in \cref{subsec:every_spectrum}.

We end our work with four technical appendices which might be of independent interest. In \cref{appendix:fixCofreeAdjunction}, we supply an adjunction needed to prove the useful adjunction \cref{lem:fixmackadj} relating equivariant semiadditive categories and nonequivariant semiadditive categories via the Mackey functor construction. In \cref{appendix:modules_SpG}, we prove the general utility \cref{thm:equivmod} which says that in the presentable setting, a $G$-stable category is equivalently given by a $\spectra_G$-module in ordinary  categories. We then show that finite posets with $G$-actions are equivariantly finite in \cref{appendix:posets} (in fact, also for compact Lie groups $G$), which is important in the general equivariant theory to insure that norms preserve equivariant compactness. For example, this is crucial\footnote{See the end of the introduction in \cite{kaifNoncommMotives} for more details on how this  fits into the general  foundations.} in constructing a $G$-symmetric monoidal structure on $\udl{\cat}\perfect_G$. Finally, we construct a 2-categorical enhancement of the $G$-symmetric monoidal structure on presentables in \cref{appendix:2cat} needed for our construction of equivariant THH.

\subsection{Acknowledgements}
We thank Thomas Blom for providing a reference about 2-categories, and Julian Kranz for bringing Bunke's and Izumi's works \cite{bunke,izumiflip} to our attention. KH is supported by the European Research Council (ERC) under Horizon Europe (GeoCats, ID: 101042990 and BorSym, ID: 101163408), and thanks the University of Bonn and the Max Planck Institute for Mathematics (MPIM) in Bonn for their conducive working environments. KH is also grateful to Germany's Excellence Strategy EXC 2044/2 –390685587, Mathematics Münster: Dynamics–Geometry–Structure for funding a visit to the University of M\"unster during which much of the final stages of this work was carried out.
MR is funded by the Deutsche Forschungsgemeinschaft (DFG, German Research Foundation) – Project-ID 427320536 – SFB 1442, as well as under Germany's Excellence Strategy EXC 2044/2 –390685587, Mathematics Münster: Dynamics–Geometry–Structure.

\addtocontents{toc}{\protect\setcounter{tocdepth}{2}} %subsections appear again in TOC

\section{Elements}\label{section:elements}

\subsection{Preliminaries}\label{subsec:prelims}

The basic materials here belong to the general theory of \textit{parametrized category theory} as introduced by Barwick-Dotto-Glasman-Nardin-Shah. As such, we refer the reader to \cite{expose1Elements,shahThesis,shahPaperII,nardinExposeIV,nardinThesis,nardinShah} for the source materials. The reader may also find \cite[\textsection 2.1 and 2.2]{kaifNoncommMotives} a convenient reference for a detailed collection of the  materials we will need in this work together with precise citations to the aforementioned body of work. As such, we will adopt a brisk pace in the coming recollections. The reader familiar with parametrized category theory should skip ahead to \cref{section:GMot} and return here as needed. We mention also that there is a more recent alternative, and in many (but not all) ways more general, treatment of parametrized category theory via topos theory due to Martini-Wolf, cf. \cite{Martini2022Yoneda,Martini2022Cocartesian,MartiniWolf2024,MartiniWolf2022Presentable}. We have opted to use the older theory since it is the one which affords the notions of $G$-stability and $G$-symmetric monoidality, both of which play a central role in the present work.

\subsubsection*{$G$-categories and $G$-symmetric monoidality}
Let $G$  be a finite group throughout the article.

\begin{recollect}[$G$-categories]
As  in the introduction, the category $\cat_G$ of $G$-categories is defined as $\func(\orbit_G\op,\cat)$  where $\orbit_G$ is the orbit category of $G$. We  denote a generic object therein by $\udl{\sC}$, and for $H\leq G$, we write $\sC_H = \udl{\sC}(G/H)\in\cat$ for its evaluation at $G/H$.

In fact, these categories assemble across the subgroups of $G$ to yield a large $G$-category $\udl{\cat}_G\in \widehat{\cat}_G\coloneqq \func(\orbit_G\op,\widehat{\cat})$ of small $G$-categories which sends $G/H$ to $\cat_H$, and for an inclusion $w\colon H\rightarrowtail K$ of subgroups of $G$, the restriction functor $w^*\colon \cat_K\rightarrow \cat_H$ is induced by the forgetful functor $w\colon \orbit_H\simeq (\orbit_K)_{/K/H}\rightarrow \orbit_K$. These functors admit left and right adjoints $w_!$ and $w_*$, respectively. They should be viewed as indexed coproducts and indexed products, i.e. coproducts and products indexed over (finite) $G$-sets instead of just over  (finite) sets; or alternatively as induction and coinduction functors. Moreover, since $\cat_G$ is cartesian closed, there are internal functor categories which we denote by $\udl{\func}(-,-)$. For $\udl{\sC},\udl{\D}\in\cat_G$ and $f\colon G/H\rightarrow G/G$ the unique map, evaluation of $\udl{\func}(\udl{\sC},\udl{\D})$ at $G/H$ yields the category $\func_H(f^*\udl{\sC},f^*\udl{\D})$ of $H$-functors from $f^*\udl{\sC}$ to $f^*\udl{\D}$.

In fact, we can (and it is convenient to be able to)  speak also of $V$-categories for any finite $G$-set $V = \coprod_iG/H_i$ as follows: via the fully faithful inclusion $\orbit_G\subseteq \finite_G$ into all finite $G$-sets, we may consider the slice category $(\orbit_G)_{/V}\simeq \coprod_i(\orbit_G)_{/G/H_i}$ and therefore also $\cat_V\coloneqq \func((\orbit_G)_{/V}\op,\cat)\simeq \prod_i\cat_{H_i}$. Restricting along the forgetful map $(\orbit_G)_{/V}\rightarrow \orbit_G$, we obtain a restriction functor $V^* \colon \cat_G\rightarrow \cat_V$, so that for a $G$-category $\udl{\sC}\in\cat_G$, $V^*\udl{\sC}\in\cat_V$ will denote its associated $V$-category. 
\end{recollect}

\begin{terminology}
    The notion of $G$-categories recalled above is more highly structured than merely being a category with $G$-actions. The latter notion does however embed in the former via the inclusion $b_* \colon \func(BG,\cat)\hookrightarrow \cat_G$ given by right Kan extending along the fully faithful inclusion $b\colon BG\subseteq \orbit_G$ of the full subcategory spanned by $G/e\in\orbit_G$. Following standard parlance, we will sometimes refer to an object in $\cat_G$ as a \textit{genuine} equivariant category and an object in $\func(BG,\cat)$ as a \textit{Borel} equivariant category, and extend the use of these adjectives to other situations where a similar distinction makes sense.
\end{terminology}

\begin{recollect}[Adjunctions and (co)limits]\label{recoll:adjunction_colimits}
    Since $\cat_G = \func(\orbit_G\op,\cat)$ is naturally an $(\infty,2)$-category, we may speak of adjunctions. It turns out that a $G$-functor has a left (resp. right) adjoint if and only if it has one levelwise and all the associated Beck--Chevalley squares commute (and in this case, the right, resp. left, adjoint is given by the levelwise adjoint).
    
    For $G$-categories $\udl{I}$ and $\udl{\sC}$, we say that $\udl{\sC}$ admits $\udl{I}$-indexed colimits (resp. limits) if the restriction $G$-functor $I^* \colon \udl{\sC}\rightarrow \udl{\func}(\udl{I},\udl{\sC})$ admits a left adjoint $I_!$ (resp. right adjoint $I_*$) in $\cat_G$. More generally, we can of course also speak of left/right Kan extensions along a map $\udl{I}\rightarrow \udl{J}$ of $G$-categories. An important special case is when we consider a map $w\colon W \rightarrow V$ in $\finite_G$ as a map $w\colon \udl{W}\rightarrow \udl{V} $ of $G$-categories. We say that $\udl{\sC}$ admits finite indexed coproducts (resp. products) if $w_!$ (resp. $w_*$) exists for all such $w$. In these cases, note also that $\udl{\func}(\udl{W},\udl{\sC})\simeq W_*W^*\udl{\sC}$, and for example when $V=G/G$, the restriction map $w^*\colon \udl{\sC}\rightarrow W_*W^*\udl{\sC}$ should be thought of as the equivariant diagonal map into the equivariant product $W_*$.

    For $I\in\cat$, we say that $\udl{\sC}$ admits (fiberwise) $I$-indexed (co)limits if it admits $\udl{\constant}I$-indexed (co)limits. It turns out that this is equivalent to requiring that $\sC_H$ admits $I$-indexed (co)limits for each $H\leq G$, and that all the restriction functors preserve these.

    We say that a $G$-category is $G$-cocomplete (resp. $G$-complete) if it admits left (resp. right) Kan extensions along $G$-functors $\udl{I}\rightarrow \udl{V}$ where $\udl{V}\in \finite_G$. A practically useful equivalent characterisation of this is that the $G$-category admits fiberwise colimits (resp. limits) and admits finite indexed coproducts (resp. products).

\end{recollect}

\begin{recollect}[Ind-completions and idempotent-completeness]\label{recollect:ind_completions}
    The notion of compact objects in a parametrized category is simply given by being fiberwise compact, cf. \cite[\textsection 5.1]{kaifPresentable}.
    The parametrized Ind-completion functor which freely adds (fiberwise) filtered colimits is simply given by performing fiberwise Ind-completions, cf. \cite{shahPaperII}. Furthermore, we say that a $G$-category is idempotent-complete if it is so fiberwise, cf. \cite[Def. 5.3.1]{kaifPresentable}.
\end{recollect}

\begin{recollect}[$G$-presentability]
    There is a good theory of $G$-presentable categories by \cite{kaifPresentable,MartiniWolf2022Presentable} and it is much the same as in the ordinary world, for instance, there is  the adjoint functor theorem in this context, cf. for example \cite[Thm. 6.2.1]{kaifPresentable}. The $G$-presentable $G$-categories may be characterised as those $G$-categories which are $G$-cocomplete and which are fiberwise presentable, cf. for example \cite[Thm. 6.1.2 (7)]{kaifPresentable}.
\end{recollect}

\begin{recollect}[$G$-symmetric monoidal structures]
    Just as in the nonequivariant situation where a symmetric monoidal category is equivalently described, cf. \cite[Thm. 2.3.9]{nardinShah}, as a finite-product-preserving functor from spans of finite sets to categories, a $G$-symmetric monoidal category is a $G$-Mackey functor valued in categories, i.e. an object $\udl{\sC}^{\otimes}$ in $\mackey_G(\cat)\coloneqq \func^{\times}(\spancategory(\finite_G),\cat)$ where the right hand side denotes the category of finite-product-preserving functors from the span category of $\finite_G$ to $\cat$. Such objects carry a lot of structures, which includes an underlying $G$-category $\udl{\sC}$. For any map $w\colon W \rightarrow V$ in $\finite_G$, the restriction functoriality $w^* \colon \sC_V\rightarrow \sC_W$ is endowed with a symmetric monoidal structure; furthermore, there are also symmetric monoidal functors $w_{\otimes}\colon \sC_W\rightarrow \sC_V$ which generalize the tensor products in the nonequivariant setting, to be thought of as the structure of tensors indexed by finite $G$-sets and which we shall refer to as the \textit{multiplicative norms} from now on. 
\end{recollect}

\begin{example}
    Since the $G$-category $\udl{\cat}_G$ admits indexed products, it may be endowed with the (unique) $G$-cartesian symmetric monoidal structure, cf. \cite[Thm. A']{natalieTensorEquivariant}.
\end{example}

There are infinite coherences governing the interaction between the restriction on the one hand and the indexed products and multiplicative norms on the other hand. The fundamental basechange\footnote{classically, these are also known as double-coset formulas.} relations are the following, which notations we will refer to throughout the article.

\begin{setting}\label{setting:pullback_notations}
The following are the standard pullback diagrams in $\finite_G$
\begin{center}
\begin{tikzcd}
    Z \rar["\overline{w}"] \dar["\overline{u}"']\ar[dr, phantom, "\lrcorner",very near start]\ar[dr, phantom, "{\text{(i)}}"] & U \dar["u"]\\
    W \rar["w"] & V
\end{tikzcd}
\hspace{10mm}
\begin{tikzcd}
    W \sqcup C \rar["\id\sqcup c"] \dar["\id\sqcup \overline{c}"']\ar[dr, phantom, "\lrcorner",very near start]\ar[dr, phantom, "{\text{(ii)}}"] & W \dar["w"]\\
    W \rar["w"] & V
\end{tikzcd}
\end{center}
where $w\colon W \rightarrow V$ is the direction we usually take $w_*$ or $w_{\otimes}$, and $u$ is the direction we usually take $u^*$. In these situations, we have equivalences
\[u^*w_*\simeq \overline{w}_*\overline{u}^*\quad u^*w_{\otimes}\simeq \overline{w}_{\otimes}\overline{u}^*\quad w^*w_*\simeq \id\times {c}_*\overline{c}^*\quad \overline{w}_{\otimes}\overline{u}^*\quad w^*w_\otimes\simeq \id\otimes {c}_\otimes\overline{c}^*.\]
When necessary, we write $Z=\coprod_iZ_i$ for the orbital decomposition of $Z$ and also write $\overline{u}=\sqcup_i\overline{u}_i$ and $\overline{w}=\sqcup_i\overline{w}_i$ for the associated decomposition.
\end{setting}

\subsubsection*{$G$-semiadditivity  and $G$-stability}
Next, we recall Nardin's \cite{nardinExposeIV} definition of $G$-semiadditive and $G$-stable $G$-categories.

\begin{recollect}[$G$-semiadditivity]
    Let $\udl{\sC}\in\cat_G$ be a pointed $G$-category, i.e. by virtue of the notion of fiberwise (co)limits from \cref{recoll:adjunction_colimits}, this means that  $\sC_H$ is a pointed category for each $H\leq G$, and the restriction functors preserve the zero objects. Suppose also that it admits finite indexed (co)products. Then, as in the nonequivariant situation, for each map $w\colon W\rightarrow V$ in $\finite_G$, there is a Beck-Chevalley-type transformation $w_!\Rightarrow w_*$ of functors $\udl{\func}(\udl{W},\udl{\sC})\rightarrow \udl{\func}(\udl{V},\udl{\sC})$  using the pointedness assumption. We say that $\udl{\sC}$ is $G$-semiadditive if all these transformations are equivalences.
\end{recollect}

Using this, we may then define Nardin's notion of $G$-stability.

\begin{defn}
    Let $\udl{\sC}\in\udl{\cat}$. We say that it is $G$-stable if it is fiberwise stable, that is, each $\sC_H$ is stable, and it is $G$-semiadditive. 
\end{defn}

\begin{example}
    The property of $G$-stability  axiomatizes the basic structures in genuine equivariant stable homotopy theory. As such, the basic examples of $G$-stable categories include that of genuine equivariant spectra $\myuline{\spectra}_G\coloneqq \{G/H\mapsto \spectra_H\}_{H\leq G}$ and more generally, for each $R\in \calg(\spectra_G)$, the $G$-categories $\udl{\module}_R(\myuline{\spectra}_G)\coloneqq \{G/H\mapsto \module_{\res^G_H}(\spectra_H)\}_{H\leq G}$. Another important source of examples are the Borel equivariant stable categories, i.e. the $G$-categories $\udl{\borel}_G(\sC)\coloneqq \{G/H\mapsto \sC^{hH}\}_{H\leq G}$ associated to any stable category $\sC$ equipped with a $G$-action.
\end{example}

There is  a general way of relating ordinary functors and $G$-functors in the semiadditive regime which is useful to connect equivariant statements to ordinary, nonequivariant ones.  Since the proof requires some general technical inputs, we have chosen to defer the proof to the end of \cref{appendix:fixCofreeAdjunction}. 

\begin{prop}\label{lem:fixmackadj}
There is an adjunction $$(-)^G: \cat^{\oplus}_G\rightleftarrows \cat^{\oplus} : \udl{\mackey}_G(-)$$ between the category of $G$-semiadditive $G$-categories and the category of semiadditive categories, where $(-)^G$ is the functor which evaluates a $G$-category at level $G/G$, and $\udl{\mackey}_G(-)$ sends an ordinary semiadditive category $\A$ to the $G$-category which at level $G/H$ is given by the $\A$-valued $H$-Mackey functors, i.e. $\mackey_H(\A)\coloneqq \func^{\times}(\spancategory(\finite_H),\A)$, the category of finite-product-preserving from the span category of $\finite_H$ to $\A$. 

The counit is given by the fixed point functor $\mackey_G(\sC)\xrightarrow{(-)^G} \sC$, and for $\udl{\A} \in \cat_G^\oplus$ and $K\leq H \leq G$, the composite $\A_H \xrightarrow{\textnormal{unit}^H} \mackey_H(\A_G) \xrightarrow{(-)^K} \A_G$ is given by $\coind^G_K\res^H_K\simeq \ind_K^G\res_K^H$. 
\end{prop}

\subsubsection*{$G$-distributivity}

Equivariant distributivity is a notion  which generalizes  the idea of tensor products which are bicocontinuous into the equivariant setting. This was first defined by Nardin in his thesis \cite{nardinThesis}, but see also \cite[Def. 3.2.3]{nardinShah}  where the theory was further developed. This concept will play a central role in our main results and we recall it now.
\begin{defn}
    Let $w\colon W \rightarrow V $ be a map in $\finite_G$. A $V$-functor $w_*\udl{\sC}\rightarrow \udl{\D}$ is said to be $w$-\textit{distributive} if  for any map $u\colon U\rightarrow V$ in $\finite_G$ and using the notations as in the pullback square (i) in \cref{setting:pullback_notations}, and a $Z$-colimit diagram $\partial \colon \underline{K}\tcocone\rightarrow \overline{u}^*\underline{\sC}$, the diagram
    \[(\overline{w}_*\underline{K})\tcocone\xlongrightarrow{\canonical} \overline{w}_*(\underline{K}\tcocone) \xlongrightarrow{\overline{w}_*\partial} \overline{w}_*\overline{u}^*\underline{\sC} \simeq u^*{w}_*\underline{\sC} \xlongrightarrow{u^*\partial} {u}^*\underline{\D}\]
    is a $U$--colimit diagram in $\underline{\D}$.
\end{defn}

\noindent This  allows us to define the equivariant analogue of distributive symmetric monoidal structures.

\begin{defn}
    Let $\udl{\sC}$ be a $G$-cocomplete $G$-category. A $G$-symmetric monoidal structure on $\udl{\sC}$ is said to be a \textit{distributive $G$-symmetric monoidal structure} if for all maps $w\colon W \rightarrow V$ in $\finite_G$, the $V$-functor $w_{\otimes}\colon w_*w^*V^*\udl{\sC}\rightarrow V^*\udl{\sC}$ is $w$-distributive. 
\end{defn}

\begin{example}
    Nardin constructed\footnote{It has been noticed that there are some gaps in his proof; these are addressed in \cite{brankoKaifNatalie}.} in \cite{nardinThesis} a $G$-symmetric monoidal structure on $\udl{\presentable}^L$ analogous to the Lurie tensor product, and objects in $\calg_G(\udl{\presentable}^L)$ are $G$-presentable categories equipped with distributive $G$-symmetric monoidal structures. Examples of such objects are the $G$-categories of $G$-spaces $\udl{\spc}$ and genuine $G$-spectra $\myuline{\spectra}$. As in the nonequivariant situation, we will sometimes call these $G$-presentably symmetric monoidal categories.
\end{example}

We will mainly be concerned with the interactions of distributive multiplicative norms with two kinds of colimits, namely filtered colimits and cofiber sequences. We record these  now.

\begin{rmk}\label{rmk:sifted_constant_diagrams}
    Let $I\in\cat$ be a sifted category. Then note that the constant $G$-category $\udl{\constant}I$ satisfies an equivariant siftedness property, i.e. for every map $w\colon W \rightarrow V$ in $\finite_G$, the diagonal $V$-functor $\Delta\colon \udl{\constant}I\rightarrow w_*w^*\udl{\constant}I$ is $G$-colimit cofinal. This is because  $\Delta$ is given fiberwise by the usual diagonal functor of $I$ to a finite product of itself, and equivariant colimit cofinality is checked fiberwise by \cite[Thm. 6.7]{shahThesis}.
\end{rmk}

\begin{obs}\label{rmk:distributivity_over_sifted}
    Let $\udl{\sC}$ be a $G$-cocomplete $G$-category equipped with a distributive $G$–symmetric monoidal structure and $w\colon W \rightarrow V$ a map in $\finite_G$. Then the functor $w_{\otimes}\colon \sC_W \rightarrow \sC_V$ given by the $G$-symmetric monoidal structure preserves sifted colimits. To see this, given any sifted colimit diagram $I\cone\rightarrow \sC_W$, we get an induced $G$-colimit diagram $\udl{\constant}I\cone \rightarrow W^*\udl{\sC}$. Hence, by distributivity, the diagram $(w_*\udl{\constant}I)\cone \rightarrow w_*(\udl{\constant}I\cone)\rightarrow w_*w^*V^*\udl{\sC}\xrightarrow{w_{\otimes}} V^*\udl{\sC}$ is a $G$-colimit diagram. Thus, by \cref{rmk:sifted_constant_diagrams}, so is the restricted diagram $\udl{\constant}I\cone\xrightarrow{\Delta}(w_*\udl{\constant}I)\cone \rightarrow w_*(\udl{\constant}I\cone)\rightarrow w_*w^*V^*\udl{\sC}\xrightarrow{w_{\otimes}} V^*\udl{\sC}$. Evaluating this composite $V$-functor at level $V$ now gives the desired conclusion.
\end{obs}

The case of cofiber sequences is  a bit trickier to deal with. To properly deal with these, we will need to recall the notion of equivariant cubes. Model categorical precursors to our considerations here have appeared for example in \cite[App. A.3]{HHR} and in \cite{dottoExcision}. However, many point-set arguments there do not work directly in the fully $\infty$-categorical context; we provide model-independent and more general arguments for the next two lemmas, building on the $\infty$-categorical treatment in \cite[\textsection 3]{kaifNoncommMotives}.

\begin{recollect}[Equivariant cubes]
    Let $w\colon W \rightarrow V$ be a map in $\finite_G$. We then obtain a $V$-category $w_*[1]\in \cat_V$ by applying the indexed product functor $w_* \colon \cat_W\rightarrow \cat_V$ on the constant $W$-category $[1]$. These are fiberwise given by cubes of various dimensions, cf. \cite[Prop. 3.1.2]{kaifNoncommMotives}. Moreover, they are in fact even equivariantly finite categories by \cref{prop:finite_posets}.

    In the special case of the unique map $w\colon W \rightarrow G/G$ in $\finite_G$, it is easy to describe $w_*[1]\in \cat_G$ concretely: it may equivalently be viewed as the $G$-category of subsets of the finite $G$-set $W$ which at level $K\leq G$ is given by the full subposet of subsets of $W$ which are $K$-invariant under the  action of $K$ on $W$.
\end{recollect}

The following result is the first technical result we need regarding the interaction of norms and cofibers -- it gives us a specific way of understanding the norm of the cofiber of a map $x\to y$ in terms of $x$ and $y$. 

\begin{lem}\label{prop:cubical_cofibs}
    Let $\udl{\sC}$ be a pointed $W$-category and $\udl{\D}$ a be pointed $V$-category, with finite $W$-, resp. $V$-colimits, and let $w\colon W \rightarrow V$ be a map in $\finite_G$.  Let $N\colon w_*\udl{\sC}\to \udl{\D}$ an  $w$-distributive functor and $x\to y\to z$ a cofiber sequence in $\udl{\sC}$. The induced maps fit in a cofiber sequence
    \[\colim_{w_*[1]\setminus 1}N(x\rightarrow y)\longrightarrow N(y)\rightarrow N(z)\] where we have written $N(x\rightarrow y)$ for the functor $ w_*[1]\xrightarrow{w_*(x\rightarrow y)}w_*\udl{\sC}\xrightarrow{N}\udl{\D}$.
\end{lem}
\begin{proof}
    The proof  works exactly as in \cite[Prop. 3.2.8]{kaifNoncommMotives} where the case of $N$ being part of a parametrized symmetric monoidal structure was treated. Here, we just replace the statement that tensoring with zero objects gives the zero object with the fact that distributive multivariable functors evaluate to zero when one of the inputs is the zero object. 
\end{proof}

The following cubical filtration will serve as the combinatorial backbone of our main result.

\begin{nota}\label{nota:filtration_of_cubes}
    Let $w\colon G/H\rightarrow G/G$ be the unique map in $\finite_G$.  Viewing $w_*[1]$ as the $G$-category of subsets of $G/H$, we write $(w_*[1])_{\leq i}\subseteq w_*[1]$ for the full $G$-subcategory of subsets of size at most $i$ where $1\leq i\leq |G/H|$.
\end{nota}

\begin{lem}\label{lem:inducedness_of_cube_cofibers}
    Let $\udl{\D}$ be a be pointed $G$-category with finite $G$-colimits, $w\colon G/H\rightarrow G/G$ be the unique map, and let $F\colon w_*[1]  \rightarrow \udl{\D}$ be a diagram. For every $1\leq i\leq |G/H|$, the map
        
        \[\colim_{{(w_*[1])}_{\leq i-1}\backslash 1}F \longrightarrow \colim_{{(w_*[1])}_{\leq i}\backslash 1}F\] has cofiber given by a finite coproduct of properly induced objects in $\udl{\D}$.
\end{lem}
\begin{proof}
    Let $\myuline{Q}\coloneqq w_*[1]\backslash 1$ and $\myuline{Q}_i\coloneqq (w_*[1])_{\leq i}\backslash 1$ to ease notation and we write $\myuline{Q}_{i-1}\xhookrightarrow{a} \myuline{Q}_i\xhookrightarrow{b}\myuline{Q}$ and $\udl{Z}_i\coloneqq \myuline{Q}_{i}\backslash \myuline{Q}_{i-1}\xhookrightarrow{z} \myuline{Q}_i$ for the inclusions. Observe that $\udl{Z}_i\in\finite_G$ because all objects in $\udl{Z}_i$ have the same cardinality, and so there are no nonidentity morphisms in the $G$-poset $\myuline{Q}_i$. In fact, $\udl{Z}_i$ can only have proper orbits since the only objects in $w_*[1]$ on which $G$-acts trivially are the empty subset and the whole subset, neither of which are possible elements of $\udl{Z}_i$.
    
    Write $M\coloneqq b^*F\in \udl{\func}(\myuline{Q}_i,\udl{\D})$. Using these notations, the statement may then be rewritten to assert that the cofiber  of the counit map
    \[\myuline{Q}_{(i-1)!}a^*M \longrightarrow \myuline{Q}_{i!}M\] is a finite coproduct of properly induced objects in $\udl{\D}$.    Now, by definition $a: \myuline{Q}_{i-1}\subseteq \myuline{Q}_i$ is a full subcategory inclusion, and so the counit map $a_!a^*M\to M$ is an equivalence on $\myuline{Q}_{i-1}$, more precisely upon applying $a^*$. Setting $K\coloneqq\cofib(a_!a^*M\rightarrow M)$, we thus find that $a^*K=0$. Since $\udl{Z}_i\subset \myuline{Q}_i$ is a cosieve, it follows that the map $z_!z^*K\rightarrow K$ is an equivalence, and so $\myuline{Q}_{i!}K\simeq \udl{Z}_{i!}z^*K$ is a coproduct of properly induced objects of $\udl{\D}$ by the last sentence of the first paragraph. Since $\myuline{Q}_{i!}K$ is the cofiber  $\myuline{Q}_{(i-1)!}a^*M\to \myuline{Q}_{i!}M$, this proves the claim.
\end{proof}

\subsubsection*{Dwyer-Kan localizations}

We now explain some equivariant aspects of Dwyer-Kan localizations. As usual, this is the procedure by which we formally invert arrows belonging to a wide subcategory, to be thought of as the ``weak equivalences'' in the classical setting of model categories.  It turns out that it is given in a fiberwise fashion on a $V$-category for any $V\in\finite_G$, cf. for example the proof of \cite[Prop. 2.3.4]{kaifNoncommMotives}.

\begin{recollect}[Symmetric monoidality]\label{recollect:monoidal_DK_localization}
    Let $T\in\finite_G$, $\udl{\sC}$ be a $T$-symmetric monoidal category, and $\udl{S}\subseteq \udl{\sC}$ a wide $T$-subcategory. Suppose that $\udl{S}$ is closed under the multiplicative norms, i.e. for all morphisms $w\colon W \rightarrow V$ in $(\finite_{G})_{/V}$ and all morphisms $f\colon x\rightarrow y$ in $S_W$, the morphism $w_{\otimes}f\colon w_{\otimes}x\rightarrow w_{\otimes}y$ is in $S_V$. Then by \cite[Prop. 2.3.4]{kaifNoncommMotives}, we see that the Dwyer-Kan localization $\udl{\sC}\rightarrow \udl{\sC}[\udl{S}^{-1}]$ naturally refines to a $V$-symmetric monoidal functor.
\end{recollect}

Next, we show that Dwyer-Kan localizations interact nicely with semiadditivity.

\begin{lem}\label{lm:Gaddloc}
     Let $\udl{\A}$ be a $G$-category, and $\udl{S}\subset\udl{\A}$ a wide $G$-subcategory.      If $\udl{\A}$ admits finite indexed coproducts and $\udl{S}$ is stable under those, then $\udl{\A}[\udl{S}^{-1}]$ admits finite indexed coproducts and the localization functor preserves them. Furthermore, if $\udl{\A}$ is $G$-semiadditive, then so is $\udl{\A}[\udl{S}^{-1}]$. Finally, if $F\colon \udl{\A}\rightarrow \udl{\sC}$ preserves finite indexed coproducts and sends $\udl{W}$ to equivalences, then the unique factorization $\overline{F}\colon \udl{\A}[\udl{S}^{-1}]\rightarrow\udl{\sC}$ also preserves finite indexed coproducts.
 \end{lem}
 \begin{proof}
     Let $w\colon W \rightarrow V$ be a map in $\finite_G$. By basechanging first to $V$-categories if necessary, we may without loss of generality assume that $V=G/G$. Recall from \cref{recoll:adjunction_colimits} that $\udl{\A}$ admits $W$-indexed coproducts is equivalent to the diagonal functor $\udl{\A}\to W_*W^* \udl{\A}$ admitting a left adjoint.      If this left adjoint preserves $w_*\udl{S}$, then it descends to the localization and the following diagram is horizontally left adjointable, see \cite[Prop. 7.1.14]{cisinski}: 
\begin{equation}\label{eqn:left_adjointable}
\begin{tikzcd}
	{\udl{\A}} & {w_*\udl{\A}} \\
	{\udl{\A}[\udl{S}^{-1}]} & {(w_*\udl{\A})[w_*\udl{S}^{-1}]}
	\arrow[from=1-1, to=1-2]
	\arrow[from=1-1, to=2-1]
	\arrow[from=1-2, to=2-2]
	\arrow[from=2-1, to=2-2]
\end{tikzcd}\end{equation}
Using the equivalence $(w_*\udl{\A})[w_*\udl{S}^{-1}]\simeq w_*(\udl{\A}[\udl{S}^{-1}])$ (which can be verified fiberwise), we conclude that $\udl{\A}$ has $w$-indexed coproducts and the localization preserves them. 

Furthermore, the dual statement shows that if $\udl{\A}$ admits $w$-indexed products that preserve $\udl{S}$, so does $\udl{\A}[\udl{S}^{-1}]$ and the localization preserves them. In fact, if $\udl{\A}$ is $G$-semiadditive, that $w$-indexed coproducts preserve $\udl{S}$ implies that $w$-indexed products also do. In total, if $\udl{\A}$ is $G$-semiadditive and $\udl{S}$ is closed under $w$-indexed coproducts, then $\udl{\A}[\udl{S}^{-1}]$ admits $w$-indexed products and coproducts and the localization functor preserves both of them. By naturality of the norm map $w_!\to w_*$ along functors preserving $w$-(co)products, and by essential surjectivity of $\udl{\A}\to \udl{\A}[\udl{S}^{-1}]$, we deduce that $\udl{\A}[\udl{S}^{-1}]$ is also $G$-semiadditive. 

Lastly, the statement that $\overline{F}$ preserves finite indexed coproducts is a straightforward consequence of the fact that the square \cref{eqn:left_adjointable} is horizontally left adjointable and that the vertical arrow ${w_*\udl{\A}}\rightarrow {(w_*\udl{\A})[w_*\udl{S}^{-1}]}$ is essentially surjective.
 \end{proof}

\subsection{Perfect \texorpdfstring{$G$}--stable categories and Karoubi sequences}\label{subsec:catperf}

We now recall the substrate on which our notion of equivariant motives will be built, namely, an equivariant generalization of $\cat\perfect$. The basic reference for these materials is \cite[\textsection 2.5]{kaifNoncommMotives}.

\begin{recollect}
    We let $\udl{\cat}\perfect_G\subset \udl{\cat}_G$ be the nonfull $G$-subcategory of small $G$-stable categories which are idempotent-complete and $G$-exact functors.  As in the nonequivariant setting, there is an equivalence 
    \begin{equation}\label{eqn:ind_omega_equivalence}
        \ind\colon \udl{\cat}\perfect_G\simeq \udl{\presentable}^L_{G,\mathrm{st},\omega}\cocolon (-)^{\omega}
    \end{equation} of $G$-categories, cf. \cite[Nota. 2.1.38]{kaifNoncommMotives}, where the right hand side is the nonfull $G$-subcategory of the $G$-category $\udl{\presentable}^L_{G,\mathrm{st}}$ of $G$-presentable-stable categories spanned by the compactly-generated ones and compact-preserving functors. In fact, there are $G$-symmetric monoidal structures on both sides and this equivalence can be made compatible with these structures by \cite[Thm. E]{brankoKaifNatalie}.    Moreover, by \cite[Thm. 2.5.11]{kaifNoncommMotives}, the inclusion $\udl{\cat}\perfect_G\subset \udl{\cat}_G$ induces a nonfull inclusion
    \begin{equation}\label{eqn:inclusion_catperf_into_mackey}
        \udl{\cat}\perfect_G\subset \udl{\mackey}_G({\cat}\perfect)
    \end{equation}
    which inherits all $G$-(co)limits.
\end{recollect}

\begin{example}[Modules]\label{example:algebra_objects_catperfGI}
For any normed $G$-ring spectrum $\udl{R}\in\calg_G(\myuline{\spectra})$, the $G$-category of  $\udl{R}$-modules $\udl{\module}_{\udl{R}}(\myuline{\spectra})$ canonically admits the structure of a $G$-commutative algebra in $\udl{\presentable}^L_G$, see \cite[\textsection 5]{phil}. In fact, more generally by \cite[Thm. A]{QuinnZhu}, for any $G$-operad $\udl{\orbit}$ and $R\in \algebraCategory_{\udl{\orbit}\otimes  \everythingAlgebra_1}(\myuline{\spectra})$, the $G$-category $\udl{\lmodule}_R(\myuline{\spectra})$ lifts to an object in $\algebraCategory_{\udl{\orbit}}(\udl{\presentable}^L_G)$. By the explicit formula for the norms on module categories from the cited references given by norming up using the structure on $\myuline{\spectra}$ and then basechanging the ring, and the fact that norming up in spectra preserves finite equivariant colimits by distributivity and the equivariant finiteness of equivariant cubes by \cref{prop:finite_posets}, we learn that  $\udl{\lmodule}_R(\myuline{\spectra})$ even lifts to an object in $\algebraCategory_{\udl{\orbit}}(\udl{\presentable}^L_{G,\omega})$. Thus, by the $G$-symmetric monoidal equivalence \cref{eqn:ind_omega_equivalence}, we get an object
\[\udl{\perfectCat}(\udl{R})\coloneqq \udl{\lmodule}_{\udl{R}}(\myuline{\spectra})^{\omega}\in \algebraCategory_{\udl{\orbit}}(\udl{\cat}\perfect)\] for every $R\in \algebraCategory_{\udl{\orbit}\otimes \everythingAlgebra_1}(\myuline{\spectra})$. We refer to \cite{QuinnZhu} for many examples of such ring $G$-spectra of current interest coming from the realm of  equivariant chromatic stable homotopy theory.
\end{example}

\begin{example}[Categories with actions]\label{example:algebra_objects_catperfGII}
    Let $\sC\in \calg(\cat\perfect)^{BG}$ be an ordinary stably symmetric monoidal category with a $G$-action. Then $\sC$ induces a Borel $G$-stable $G$-category $\udl{\borel}(\sC)=\{G/H\mapsto \sC^{hH}\}_{H\leq G}$, and it can be canonically promoted to a $G$-commutative algebra in $\udl{\cat}\perfect_G$.  This is because by \cite[Thm. 2.4.10 (2)]{kaifNoncommMotives}, the inclusion $\udl{\borel}(\cat\perfect)\hookrightarrow \udl{\cat}\perfect$ refines to a $G$-lax symmetric monoidal functor. Then, by applying $\calg_G(-)$ to this map and invoking  \cite[Thm. 2.4.10 (3)]{kaifNoncommMotives}, we obtain an inclusion $\calg(\cat\perfect)^{BG}\hookrightarrow \calg_G(\udl{\cat}\perfect)$.
\end{example}

\begin{defn}
    A Karoubi sequence is a sequence $\udl{\A}\rightarrow \udl{\B}\rightarrow\udl{\sC}$ in $\udl{\cat}\perfect$ whose composite is null and is both a fiber and a cofiber sequence in $\udl{\cat}\perfect$.
\end{defn}

\begin{rmk}
    Since the inclusion \cref{eqn:inclusion_catperf_into_mackey} inherits all $G$-(co)limits, in particular, it inherits the (co)fiber sequences in Mackey functors which in turn are computed pointwise by \cite[Cor. 6.7.1]{barwick1}. This allows us to port many of the standard results about Karoubi sequences in the nonequivariant context, for instance, as recorded in \cite[App. A]{Hermitian2}. Furthermore, note that nullity of a map in $\udl{\cat}\perfect$ is a property, and so the definition above makes sense without  having to specify a commuting square involving the zero functor.
\end{rmk}

The $G$-category $\udl{\cat}\perfect_G$ is amenable to isotropy separation methods which allows us to perform proofs by inducting on the size of the finite group $G$. We record these points now.

\begin{cons}\label{cons:categorical_isotropy_separation}
    Let $\family$ be a family of subgroups in $G$. We now explain how to perform a categorical isotropy separation argument on $\udl{\cat}\perfect\simeq \udl{\presentable}^L_{\mathrm{st},\omega}$. By \cite{PD1}, for any $\udl{\sC}\in\cat\perfect_G$ and writing $\widetilde{\udl{\sC}}\coloneqq \ind\udl{\sC}$, we have a stable recollement 
    \begin{center}
        \begin{tikzcd}
             \widetilde{\udl{E\family}}\otimes \widetilde{\udl{\sC}}=\udl{\func}_*(\udl{\widetilde{E\family}},\udl{\widetilde{\sC}})\ar[rr, "p^*", hook] && \udl{\widetilde{\sC}} \ar[rr, "j^*", two heads]\ar[ll, "p_*"', two heads, bend left = 40]\ar[ll, "p_!"', two heads, bend right = 40]&& \udl{\func}(\udl{E\family},\udl{\widetilde{\sC}}) = {\udl{E\family}}_+\otimes \widetilde{\udl{\sC}}  \ar[ll, "j_!"', hook, bend right = 40]\ar[ll, "j_*"', hook, bend left = 40]
        \end{tikzcd}
    \end{center}
    associated to tensoring $\widetilde{\udl{\sC}}$ with the cofiber \begin{equation}\label{eqn:family_cofiber_sequence}
        E\family_+ \longrightarrow S^0\longrightarrow \widetilde{E\family}
    \end{equation}
    sequence in $\presentable^L_{G,\mathrm{st}}$. For any $y\in\widetilde{\udl{\sC}}$, the bifiber sequence $j_!j^*y\rightarrow y\rightarrow p^*p_!y$ may thus be identified with the bifiber sequence in $\widetilde{\udl{\sC}}$
    \[{E\family}_+\otimes y\longrightarrow y \longrightarrow \widetilde{E\family}\otimes y\] obtained by tensoring $y$ with the cofiber sequence.
    By the existence of the adjoints $p_*$ and $j_*$, the top  sequence of functors preserve compact objects and so  restricts  to a commuting diagram 
    \[
    \begin{tikzcd}
        \widetilde{\udl{E\family}}\otimes \widetilde{\udl{\sC}} & \widetilde{\udl{\sC}} \lar["p_!"',two heads]& {\udl{E\family}}_+\otimes \widetilde{\udl{\sC}} \lar["j_!"',hook]\\
        (\widetilde{\udl{E\family}}\otimes \widetilde{\udl{\sC}})^{\omega} \uar[hook, "\yoneda"]& \widetilde{\udl{\sC}}^{\omega} \lar["p_!"']\uar[hook, "\yoneda"]& ({\udl{E\family}}_+\otimes \widetilde{\udl{\sC}})^{\omega} \lar["j_!"']\uar[hook, "\yoneda"]
    \end{tikzcd}
    \] with the bottom sequence happening in $\cat\perfect_G$.

    It will be convenient to use the suggestive notation $\geomFix{\family}\colon \widetilde{\udl{\sC}}\twoheadrightarrow \geomFix{\family}\widetilde{\udl{\sC}}$ and $\geomFix{\family}\colon {\udl{\sC}}\rightarrow \geomFix{\family}{\udl{\sC}}$ to denote $p_!\colon \widetilde{\udl{\sC}}\twoheadrightarrow \widetilde{\udl{E\family}}\otimes \widetilde{\udl{\sC}}$ and $p_!\colon {\udl{\sC}}\twoheadrightarrow \widetilde{\udl{E\family}}\otimes {\udl{\sC}}$ respectively.
\end{cons}

\begin{obs}\label{obs:proper_isotropy_separation_generalities}
    Let $\proper$ be the family of proper subgroups. We may use this family to perform isotropy separations in an arbitrary $G$-stable presentable category as follows: let $\udl{\E} \in \presentable^L_{G,\mathrm{st}}$ and $x\in\udl{\E}$. Then $x=0$ if and only if for all $H\lneq G$, $\res^G_Hx=0\in\udl{\E}(G/H)$ and $\Phi^{\proper}x=0\in \Phi^{\proper}\udl{\E}\in\presentable^L_{\mathrm{st}}$. The forwards implication is clear and to see the converse, we consider the cofiber sequence 
    $E\proper_+\otimes x\longrightarrow x\longrightarrow \widetilde{E\proper}\otimes x$ in $\udl{\E}$ where $\widetilde{E\proper}\otimes x$ is the image of $\Phi^{\proper}x$ under the inclusion $\Phi^{\proper}\udl{\E}\subseteq \udl{\E}$. Since $\Phi^{\proper}x=0$, the map $E\proper_+\otimes x\rightarrow x$ is an equivalence. On the other hand, $E\proper_+\otimes  x$ is in the image of the inclusion $\udl{\func}(\udl{E\proper},\udl{E})\subseteq \udl{\E}$. Since $\coprod_{H\lneq G} G/H\rightarrow E\proper$ is an effective epimorphism, checking that $E\proper_+\otimes x\in \udl{\func}(\udl{E\proper},\udl{\E})$ is zero is equivalent to checking that $x$ goes to zero under the restriction ${\func}_G(\udl{E\proper},\udl{\E})\rightarrow \func_G(G/H,\udl{\E})\simeq \udl{\E}(G/H)$. But this is precisely supplied by the hypothesis that $\res^G_Hx=0$.
\end{obs}

\subsection{Isotropy separation for fully faithfulness}
\label{subsec:fully_faithful}

A the key point in our work is an analysis to prove fully faithfulness of certain functors. In this subsection, we provide a tool to detect fully faithfulness of $G$-exact functors between $G$-stable $G$-categories.  We do this by isotropy separation  in the $G$-category of $G$-stable $G$-categories. Since this $G$-category is itself not $G$-stable, isotropy separation methods do not come ``for free''.

\begin{lem}[Fully faithfulness isotropy detection]\label{lem:fully_faithful_isotropy_separation}
    Let $F\colon \udl{\sC}\rightarrow\udl{\D}$ be a morphism in $\cat\perfect_G$. Let $\family$ be a family, and suppose $\res^G_H F\colon \res^G_H\udl{\sC}\rightarrow \res^G_H\udl{\D}$ are fully faithful for all $H\in\family$. Then the functor $F$ is fully faithful if and only if  $\Phi^{\family} F\colon \Phi^{\family}\udl{\sC}\rightarrow\Phi^{\family}\udl{\D}$ is. 
\end{lem}
\begin{proof}
     We write $\widetilde{\udl{\sC}}\coloneqq \udl{\ind}(\udl{\sC})$ and $\widetilde{\udl{\D}}\coloneqq \udl{\ind}(\udl{\D})$. Let $x,y\in\udl{\sC}$.  By the commuting square
   \[
   \begin{tikzcd}
       \udl{\sC}\rar["F"]\dar["\yoneda"', hook] & \udl{\D}\dar["\yoneda", hook]\\
       \widetilde{\udl{\sC}} \rar["F\coloneqq F_!"] & \widetilde{\udl{\D}},
   \end{tikzcd}
   \] we may equivalently analyse the map  $\myuline{\mapsp}_{\widetilde{\udl{\sC}}}(x,y)\longrightarrow \myuline{\mapsp}_{\widetilde{\udl{\D}}}(Fx,Fy)$.  Using the isotropy separation sequence in the $y$--variable and that $x\in \widetilde{\udl{\sC}}$, $Fx\in\widetilde{\udl{\D}}$ are compact objects, we obtain the map of fiber sequences in $\myuline{\spectra}$
   \begin{equation}\label{eqn:isotropy_separation_mapping_spectra}
   \begin{tikzcd}
       E\family_+\otimes \myuline{\mapsp}_{\widetilde{\udl{\sC}}}(x,y)\rar\dar["F", "\simeq"'] & \myuline{\mapsp}_{\widetilde{\udl{\sC}}}(x,y)\rar\dar["F"] & \myuline{\mapsp}_{\widetilde{\udl{\sC}}}(x,\widetilde{E\family}\otimes y)\dar["F"]\\
       E\family_+\otimes \myuline{\mapsp}_{\widetilde{\udl{\D}}}(Fx,Fy)\rar & \myuline{\mapsp}_{\widetilde{\udl{\D}}}(Fx,Fy)\rar & \myuline{\mapsp}_{\widetilde{\udl{\D}}}(Fx,\widetilde{E\family}\otimes Fy)
   \end{tikzcd}
   \end{equation}
   where we have also used that ${E\family}_+\otimes Fy\simeq F({E\family}_+\otimes y)$ and $\widetilde{E\family}\otimes Fy\simeq F(\widetilde{E\family}\otimes y)$ since $F\colon \widetilde{\udl{\sC}}\rightarrow\widetilde{\udl{\D}}$ preserves parametrized colimits. By hypothesis, the left vertical map is an equivalence. 
   
   Recall by definition that $\Phi^{\family}\udl{\sC}\coloneqq (\Phi^{\family} \widetilde{\udl{\sC}})^{\omega}$ and similarly for $\udl{\D}$. Hence, since the functor $\Phi^{\family}(-)=\widetilde{E\family}\otimes-\colon \widetilde{\udl{\sC}}\rightarrow\widetilde{\udl{\D}}$ preserves compact objects because its right adjoint has a further right adjoint, we obtain identifications
   \[\myuline{\mapsp}_{\widetilde{\udl{\sC}}}(x,\widetilde{E\family}\otimes y)\simeq \myuline{\mapsp}_{\Phi^{\family}\widetilde{\udl{\sC}}}(\widetilde{E\family}\otimes x,\widetilde{E\family}\otimes y) \simeq \myuline{\mapsp}_{\Phi^{\family}{\udl{\sC}}}(\Phi^{\family} x,\Phi^{\family} y)\] and similarly, $\myuline{\mapsp}_{\widetilde{\udl{\D}}}(Fx,\widetilde{E\family}\otimes Fy)\simeq\myuline{\mapsp}_{\Phi^{\family}{\udl{\D}}}(F\Phi^{\family} x,F\Phi^{\family} y)$. 
   Hence, $\Phi^{\family} F$ is an equivalence if and only if the right vertical map in \cref{eqn:isotropy_separation_mapping_spectra} is an equivalence if and only if the middle vertical map of \cref{eqn:isotropy_separation_mapping_spectra} is an equivalence. This completes the proof.
\end{proof}

We will also need the following fact for our analysis.

\begin{lem}\label{lem:norms_preserve_fully_faithfulness}
    Let $F\colon\udl{\sC}\rightarrow \udl{\D}$ be a fully faithful morphism in $\cat\perfect_W$. Then $w_{\otimes}F\colon w_{\otimes}\udl{\sC}\rightarrow w_{\otimes}\udl{\D}$ is a fully faithful morphism in $\cat\perfect_V$.
\end{lem}
\begin{proof}
    By the $G$-symmetric monoidal equivalence $\ind \colon \udl{\cat}\perfect_G\rightleftharpoons \udl{\presentable}^L_{\mathrm{st},\omega}\cocolon (-)^{\omega}$, we see that $\ind(w_{\otimes}F)\colon \ind(w_{\otimes}\udl{\sC})\rightarrow \ind(w_{\otimes}\udl{\D})$ identifies with $w_{\otimes}\ind F\colon w_{\otimes}\ind\udl{\sC}\rightarrow w_{\otimes}\ind\udl{\D}$. Hence, it suffices to prove that $w_{\otimes}\ind F$ is fully faithful.
    
    Now, $\ind F$ is fully faithful by hypothesis, and it admits a right adjoint $R$ which is also in $\presentable^L_{W,\mathrm{st}}$. In particular, the unit map $\id\to R\circ \ind F$ is an equivalence, and induces, by $2$-functoriality of $w_\otimes$ from \Cref{prop:mackey2}, an equivalence $\id \to w_\otimes R \circ w_\otimes \ind F$. Again by $2$-functoriality of $w_\otimes$, this equivalence is the unit of an adjunction $w_\otimes\ind F\dashv w_\otimes R$, and so we conclude that $w_\otimes \ind F$ is fully faithful, as was to be shown.  
\end{proof}

% In fact, one can prove the following finer statement and give a formula for the mapping spectra in categories which are normed up which might be of independent interest. This will not play a role in this paper, however.

% \begin{cor}\label{mappingSpectraOfNorms}
% For $x,c \in\udl{\sC}$, there is an equivalence
% \[\myuline{\mapsp}_{w_{\otimes}\underline{\sC}}(w_{\otimes} x, w_{\otimes} c) \simeq w_{\otimes}\myuline{\mapsp}_{\underline{\sC}}(x,c)\] which is natural in both the $x$ and $c$ variables. 
% \end{cor}
% \begin{proof}\todo[color=yellow!50]{K: make this into a nonmandatory remark for its own independent interest and cite the thing with branko for all the requisite ind stuff}
%     Since $i\colon \udl{\sC}\hookrightarrow \udl{\ind}(\udl{\sC})$ is given by sending $c$ to $\myuline{\mapsp}_{\udl{\sC}}(-,c)$, the commuting diagram from \cref{IndSymmetricMonoidality} gives an equivalence of functors
%     \[\myuline{\mapsp}_{w_{\otimes}\underline{\sC}}(-, w_{\otimes} c) \simeq w_{\otimes}\myuline{\mapsp}_{\underline{\sC}}(-,c) \colon (w_{\otimes}\udl{\sC})\vop\longrightarrow\myuline{\spectra}.\] Thus, precomposing with $w_*\udl{\sC}\vop\xrightarrow{\tau}(w_{\otimes}\udl{\sC})\vop$ and tracing where $x\in w_*\udl{\sC}$ gets sent, we obtain the claimed natural equivalence.
% \end{proof}

\section{Equivariant localizing motives}\label{section:GMot}

\subsection{Definitions and constructions}\label{subsec:motloc}
The goal of this section is to  define $G$-localizing motives, following the  perspective from \cite{MotLoc}, and prove \cref{thm:GMotGPres}.

\begin{defn}
A \textit{$G$-localizing invariant} is a $G$-semiadditive functor $E\colon\udl{\cat}\perfect_G\to \udl{\E}$, where $\udl{\E}$ is some $G$-stable $G$-category, such that  $E_H\colon {\cat}\perfect_H\to {\E}_H$  sends Karoubi sequences  to bifiber sequences for every subgroup $H\leq G$. More generally, for $V$ a finite $G$-set, a \textit{$V$-localizing invariant} is a $V$-semiadditive functor $E\colon\udl{\cat}\perfect_V\to \udl{\E}$ where $\udl{\E}$ is some $V$-stable $V$-category, and which sends fiberwise Karoubi sequences to levelwise bifiber sequences. 

A $G$- or more generally $V$-localizing invariant is said to be \textit{finitary} if it also preserves fiberwise filtered colimits. We denote by $\udl{\func}^{\mathrm{fin. loc.}}(\udl{\cat}\perfect_G,\udl{\E})$ the full $G$-subcategory of $\udl{\func}(\udl{\cat}\perfect_G,\udl{\E})$ which at level $H\leq G$ consists of the finitary $H$-localizing invariants.
\end{defn}

The following is the  ur-example that the above notions are supposed to axiomatize: 
\begin{example}\label{ex:KG}
Defined similarly to \cite[Cons. 4.2.12]{kaifNoncommMotives}, the equivariant (non-connective, algebraic) $K$-theory functor $\udl{K}_G:\udl{\cat}\perfect_G\to \myuline{\spectra}_G$ defined by $\udl{K}_G(\udl{\D})^H = K(\D^H)$ is a finitary localizing invariant. We will give a precise definition in \Cref{defn:officialdefnK}.
\end{example}
There are many examples one can construct from equivariant $K$-theory, or similarly from classical (finitary) localizing invariants by applying them pointwise to $\mackey_G(\cat\perfect)$ and restricting to the nonfull $G$-subcategory $\udl{\cat}\perfect_G\subset \udl{\mackey}_G(\cat\perfect)$. We will later construct an example that is not of this form\footnote{See \Cref{rmk:THHgenuine}.}, namely equivariant $\mathrm{THH}$:

\begin{example}\label{ex:TrG}
In \Cref{subsec:THH}, we will explain how to construct ``equivariant $\THH$'' as a $G$-symmetric monoidal finitary $G$-localizing invariant.  
\end{example}
And finally, we can introduce motivic equivalences as in \cite{MotLoc}. 
\begin{defn}\label{defn:motivic_equivalence}
    A map $\udl{\A}\to \udl{\B}$ in $\cat\perfect_G$ is called a \textit{motivic equivalence} if it is sent to an equivalence by every finitary $G$-localizing invariant. We also have an analogous notion of motivic equivalences in $\cat\perfect_V$ for any finite $G$-set $V$.    We let $\udl{W}_\mot$ denote the $G$-subcategory of $G$-motivic equivalences, whose $H$-fixed points is the category of $H$-motivic equivalences. 
\end{defn}

To see that $\udl{W}_\mot$ is well-defined, we need the following verification:
\begin{lem}\label{lm:weqsubGcat}
    $\udl{W}_\mot$ does define a $G$-subcategory of $\udl{\cat}\perfect_G$, i.e. restrictions of $G$-motivic equivalences are $H$-motivic equivalences, for every subgroup $H\leq G$. More generally, for any map $w:W\to V$ of finite $G$-sets, $w^*$ sends motivic equivalences of $V$-stable $V$-categories to motivic equivalences of $W$-stable $W$-categories. 
\end{lem}
\begin{proof}
      Given any $H$-localizing invariant $E: \udl{\cat}\perfect_H\to \udl{\E}$, by adjunction we get a $G$-functor $\coind_H^G E\colon  \udl{\cat}\perfect_G\to \coind_H^G\udl{\E}$. Note that if $E$ is finitary, then so is $\coind_H^GE$ by \cref{rmk:sifted_constant_diagrams}.  We claim now that $\coind_H^G E$ is a $G$-localizing invariant. 

      First, we note that it is clearly $G$-semiadditive if $E$ was $H$-semiadditive since $\coind^G_HE$ may be written as the composite
      \[\udl{\cat}\perfect_G\xlongrightarrow{\eta} \coind^G_H\res^G_H\udl{\cat}\perfect_G = \coind^G_H\udl{\cat}\perfect_H\xlongrightarrow{\coind^G_HE}\coind^G_H\udl{\E}\] and both maps preserve finite indexed products.

      For $f:\udl{\A}\to\udl{\B}$ a $G$-exact functor of $G$-stable $G$-categories, and $K\leq G$ a subgroup, $\coind_H^GE(f)^K\in (\coind_H^G \udl{\E})^K\simeq \prod_{r\in H\backslash G /K} \udl{\E}^{H\cap K^r}$ is given by the family of the $E^{H\cap K^r}(\res_{H\cap K^r}^Gf)$. From this the claim follows directly because restrictions of fiberwise Karoubi sequences are (by definition!) fiberwise Karoubi sequences.     Now if $f$ is a $G$-motivic equivalence, the above family is an equivalence and hence plugging in subgroups $K$ of $H$ we find that $E^K(\res_K^H\res_H^G f)$ is an equivalence, and so $f$ is an $H$-motivic equivalence.
 \end{proof}

 With this, we can now use Dwyer--Kan localizations to define equivariant localizing motives.

 \begin{defn}
    We let $\udl{\Mot}\loc_G := \udl{\cat}\perfect_G[\udl{W}^{-1}_\mot]$  denote the $G$-category of \textit{$G$-localizing motives}, and we let $\udl{\U}\loc_G: \udl{\cat}\perfect_G\to \udl{\Mot}\loc_G$ denote the induced localization functor.    More generally, for a finite $G$-set $V$, we have similarly $\udl{\Mot}\loc_V$ and $\udl{\U}\loc_V$. 
\end{defn}
\begin{rmk}
    Let us briefly discuss the role of the adjective ``finitary'' in the definion above. One could perfectly have defined all relevant notions here without this adjective, and in fact our proof ultimately works also in the context of non-finitary localizing motives. 

    The first role this adjective plays is an a priori size control: it is easier to show that the localization at ``finitary motivic equivalences'' is actually a locally small category than to show the same statement for the localization at ``absolute motivic equivalences'', though the latter is also true. This would maybe justify considering $\kappa$-finitary localizing invariants for arbitrary regular cardinals $\kappa$, rather than specifically $\kappa = \aleph_0$ -- again, our arguments carry over to this setting with no change.
    
    However, the point of restricting to $\kappa=\aleph_0$ is that for $\kappa >\aleph_1$, the relevant category of motives seems much harder to access, and in particular it seems difficult to say anything nontrivial about it, while for $\kappa=\aleph_1$, it turns out (miraculously) to agree with the case of $\kappa=\aleph_0$, see \cite[Thm.  1.8]{MotLoc}. Since virtually every localizing invariant one encounters is $\aleph_1$-finitary, the restriction to this version of localizing motives seems completely mild, and we work in this setting for convenience and in order to avoid overloading our statements with set-theoretic considerations. 
\end{rmk}

Before we go on, we make a small observation that will be useful throughout. Namely, we defined motivic equivalences in terms of $G$-localizing invariants, which are certain $G$-functors, and staring at the bare category ${\cat}\perfect_G$, it may not be clear what these correspond to without remembering all the $G$-structure. We now work towards proving \cref{cor:Mackeyweq=weq} that we can also do away with some of the ``$G$'' in this definition. This will be convenient for some proofs.

\begin{lem}\label{lm:Mackeylocimpliesloc}
    Let $E:{\cat}\perfect_G\to \E$ be an additive functor to a stable category sending Karoubi sequences to bifiber sequences. The $G$-semiadditive $G$-functor $\widetilde{E}: \udl{\cat}\perfect_G\to \udl{\mackey}_G(\E)$ corresponding to $E$ under the adjunction from \Cref{lem:fixmackadj} is a $G$-localizing invariant. If $E$ preserves filtered colimits, then $\tilde{E}$ preserves fiberwise filtered colimits. More generally, the statement holds when we replace $G$ by a finite $G$-set  $V$. 
\end{lem}
\begin{proof}
    For fixed subgroups $K\leq H \leq G$, using the description of the unit and the counit from \Cref{lem:fixmackadj}, we find that the composite $\widetilde{E}^H :\cat\exact_H \to \mackey_H(\E)\xrightarrow{(-)^K} \E$ is given by $E(\induced_K^G\res_K^H -)$, which clearly sends Karoubi sequences to bifiber sequences if $E$ does.     The filtered colimits claim also follows from that analysis. 
\end{proof}

\begin{defn}
    A \textit{Mackey $G$-localizing invariant} is an additive functor $E:{\cat}\perfect_G\to \E$ to some stable category $\E$ sending Karoubi sequences to bifiber sequences.     It is finitary if it sends filtered colimits to filtered colimits. A \textit{Mackey motivic equivalence} is a $G$-exact functor $\udl{\A}\to\udl{\B}$ between small $G$-stable $G$-categories that is sent to an equivalence by every finitary Mackey $G$-localizing invariant. We have analogous definitions for finite $G$-sets $V$. 
\end{defn}
\begin{defn}\label{defn:officialdefnK}
    Ordinary (non-connective, algebraic) $K$-theory of $G$-fixed points evidently gives a Mackey $G$-localizing invariant $K((-)^G): \cat\perfect_G\to \spectra$. We define equivariant $K$-theory $\udl{K}_G$ to be the $G$-localizing invariant associated to it as in \Cref{lm:Mackeylocimpliesloc}.
\end{defn}
\begin{obs}
    The combination of \Cref{lem:fixmackadj} and \Cref{lm:Mackeylocimpliesloc} essentially states that (finitary) Mackey $G$-localizing invariants are a special case of $G$-localizing invariants, namely the special case where the target $\udl{\E}$ is\footnote{We do not believe that this is a property, rather it seems to be extra structure on $\udl{\E}$ to witness it as such.} of the form $\udl{\mackey}_G(\E_0)$ for some $\E_0$. 
\end{obs}
\begin{cor}\label{cor:Mackeyweq=weq}
    Motivic equivalences and Mackey motivic equivalences agree. 
\end{cor}
\begin{proof}
    By \Cref{lm:Mackeylocimpliesloc}, we find that every finitary Mackey $G$-localizing invariant $E:{\cat}\perfect_G\to \E$ gives rise to a finitary $G$-localizing invariant $\widetilde{\E}: \udl{\cat}\perfect_G\to \udl{\mackey}_G(\E)$, and $E$ can be recovered as the composite ${\cat}\perfect_G\xrightarrow{\widetilde{E}^G}\mackey_G(\E)\xrightarrow{(-)^G}\E$
    This directly implies that motivic equivalences are Mackey motivic equivalences.     Conversely, a  finitary $G$-localizing invariant $\udl{E}: \udl{\cat}\perfect_G\to \udl{\E}$ gives rise to a Mackey localizing invariant $E_G:{\cat}\perfect_G\to \E_G$ upon taking $G$-fixed points, and hence Mackey motivic equivalences are also motivic equivalences. 
\end{proof}

 Now as in \cite{MotLoc}, and based on \Cref{thm:equivmod}, we find: 
 
 \begin{thm}\label{thm:GMotGPres}
 The fiberwise localization $\udl{\cat}\perfect_G\to \udl{\Mot}\loc_G$ is the universal finitary $G$-localizing invariant. More precisely, $\udl{\Mot}\loc_G$ is $G$-presentable and $G$-stable, and for any $G$-presentable $G$-stable $G$-category $\udl{\E}$, restriction along this localization  $$(\udl{\sU}\loc_G)^*\colon \udl{\func}^L(\udl{\Mot}\loc_G, \udl{\E})\xlongrightarrow{\simeq}\udl{\func}^{\mathrm{fin. loc.}}(\udl{\cat}\perfect_G,\udl{\E})$$ is an equivalence. More generally, for any finite $G$-set $V$, the  localization $\udl{\sU}\loc_V\colon \udl{\cat}\perfect_V\to \udl{\Mot}\loc_V$ witnesses the target as the universal finitary $V$-localizing invariant.  
 \end{thm}
Thus, $\udl{\sU}\loc_G\colon \udl{\cat}\perfect_G\rightarrow \udl{\Mot}\loc_G$ admits another universal property apart from the tautological one as  the initial functor which inverts motivic equivalences. This is analogous to the main result of \cite{MotLoc} and shows that the localization of $\udl{\cat}\perfect_G$ at motivic equivalences satisfies a Blumberg--Gepner--Tabuada-style universal property \cite{BGT}, that is, a universal property among presentable $G$-categories equipped with a finitary $G$-localizing invariant.

 \begin{proof}[Proof of \Cref{thm:GMotGPres}]
     We first show that $\udl{\Mot}\loc_G$ is $G$-semiadditive, and that the localization functor $\udl{\sU}\loc_G$ is $G$-semiadditive as well. By \Cref{lm:Gaddloc}, it suffices to show that induction $\ind_H^K: \cat\perfect_H\to\cat\perfect_K$ preserves motivic equivalences for all subgroups $H\leq K \leq G$. By \Cref{cor:Mackeyweq=weq}, it suffices to show that it preserves Mackey motivic equivalences, which is evident since it is finitary and preserves Karoubi sequences. 
     
     Next, we argue that $\udl{\Mot}\loc_G$ is $G$-presentable and $G$-stable, and that $\udl{\sU}\loc_G$ is a finitary $G$-localizing invariant. For $G$-presentable-stability, since we already know that $\udl{\Mot}\loc_G$ is $G$-semiadditive (which in particular means that the restriction functorialities preserve ordinary colimits), it thus suffices by \cite[Thm. 6.1.2 (7)]{kaifPresentable} to show that $\udl{\Mot}\loc_G$ is fiberwise stable and presentable. By \Cref{thm:equivmod} and \Cref{cor:Gmotivic}, the localization in question is fiberwise 
     equivalent to the localization of $\module_{\spectra_H^\omega}(\cat\perfect)$ at $\spectra_H^\omega$-motivic equivalences in the sense of \cite[App. B]{MotLoc}. Therefore, by \cite[Thm. B.6 and Rmk. B.11]{MotLoc}, it is fiberwise presentable and stable, and the localization functor $\cat\perfect_H\to \cat\perfect_H[W_{\mot,H}^{-1}]$ is a finitary Mackey $H$-localizing invariant. Hence, together the localization functors assemble into a finitary $G$-localizing invariant $\udl{\sU}\loc_G$.

     Finally, we show the universal property.  By construction, we have a fully faithful functor $(\udl{\U}\loc_G)^*\colon \udl{\func}^L(\udl{\Mot}\loc_G,\udl{\E}) \hookrightarrow \udl{\func}(\udl{\cat}\perfect_G,\udl{\E})$. This in fact lands in the full subcategory of finitary $G$-localizing invariants  since $\udl{\cat}\perfect_G\rightarrow \udl{\Mot}\loc_G$ was itself a $G$-localizing invariant by the previous paragraph. To see that this is essentially surjective on finitary $G$-localizing invariants $E\colon \udl{\cat}\perfect_G\rightarrow\udl{\E}$, first note that any such functor will in particular invert the motivic equivalences by definition, and therefore it factors uniquely through a functor $\overline{E}\colon \udl{\Mot}\loc_G\rightarrow \udl{\E}$. It remains to show that $\overline{E}$ preserves $G$-colimits. By \cref{lm:Gaddloc}, we know at least that it preserves finite indexed coproducts. On the other hand, by the previous paragraph and again \cite[Thm. B.6]{MotLoc}, $\overline{E}$ also preserves fiberwise colimits. Hence, $\overline{E}$ preserves arbitrary $G$-colimits, as required. 
 \end{proof}

We end this subsection by making contact with recent breakthroughs in K-theory. A remarkable result of Efimov \cite{efimovIII} states that $\Mot\loc$ is rigid in the sense of Gaitsgory--Rozenblyum \cite[Def. 9.1.2]{gaitsgoryrozenblyum}, which is an unexpected strong finiteness property. In fact, this was proved to such a level of generality that it applies directly to our setting to yield:

\begin{cor}\label{rmk:rigidityeh}\label{cor:rigidityeh}
The category $\Mot\loc_G$ is a rigid commutative algebra in $\presentable^L_{\mathrm{st}}$, as well as as a commutative algebra over $\spectra_G$, equivalently,\footnote{See \cite{leor} for a definition of rigidity in arbitrary symmetric monoidal $2$-categories.}  in $\presentable^L_{G,\mathrm{st}}$.
\end{cor}
\begin{proof}
    The proof of \cref{thm:GMotGPres}  shows that $\Mot\loc_G$ is equivalent to $\Mot_{\spectra_G}$ with the notation from \cite{efimovIII}.      In particular, Theorem 0.3 in \textit{loc. cit.} applies, and we find that $\Mot\loc_G$ is indeed rigid.       Furthermore, as the $G$-fixed points of a $G$-presentable $G$-stable symmetric monoidal $G$-category, \Cref{thm:equivmod} allows us to view $\Mot\loc_G$ as a commutative $\spectra_G$-algebra. By \cite[Cor. 4.46]{ramzirigid}, $\Mot\loc_G$ is also rigid as a commutative $\spectra_G$-algebra.  By \Cref{thm:equivmod}, the claim about rigidity in $\presentable^L_{G,\mathrm{st}}$ follows at once.
\end{proof}

\subsection{\texorpdfstring{$G$}--symmetric monoidality of localizing motives}\label{subsec:GSymMonofMotLoc}
We now come to the heart of the matter, and prove in this subsection the following, which is the main  result of this paper.

\begin{thm}\label{cor:Nmweq}
    Let $w:W\to V$ be a map of finite $G$-sets. The functor $f_\otimes: \cat\perfect_W\to \cat\perfect_V$ sends motivic equivalences to motivic equivalences (cf. \cref{defn:motivic_equivalence}). 
\end{thm}

Given this, it is then rather immediate to obtain the following:

\begin{thm}\label{thm:motnorm}
    There exists a unique $G$-symmetric monoidal structure on $\udl{\Mot}\loc_G$ compatible with the norm $G$-symmetric monoidal structure on $\udl{\cat}\perfect_G$ under the localization $G$-functor $\udl{\U}\loc_G$.    Furthermore, for any $G$-symmetric monoidal $G$-presentable $G$-stable $G$-category $\udl{\E}$, restriction along $\U\loc_G$ induces a equivalences \small $$\udl{\func}^{L,\otimes}(\udl{\Mot}\loc_G, \udl{\E})\xrightarrow{\simeq}\udl{\func}^{\mathrm{fin. loc.}, \otimes}(\udl{\cat}\perfect_G,\udl{\E}),\quad  \udl{\func}^{L,\otimes\mathrm{-lax}}(\udl{\Mot}\loc_G, \udl{\E})\xrightarrow{\simeq}\udl{\func}^{\mathrm{fin. loc.}, \otimes\mathrm{-lax}}(\udl{\cat}\perfect_G,\udl{\E})$$ \normalsize between the $G$-category of $G$-symmetric monoidal (resp. $G$-lax symmetric monoidal), $G$-colimit-preserving functors  and the $G$-category of $G$-symmetric monoidal (resp. $G$-lax symmetric monoidal) finitary $G$-localizing invariants. 
\end{thm}
\begin{proof}
    The first statement follows immediately from \Cref{cor:Nmweq} and \cref{recollect:monoidal_DK_localization}. The second statement for $G$-symmetric monoidal (resp. $G$-lax symmetric monoidal) functors  is then a combination of   \cite[Lem. 5.28 (3)]{quigleyShahParametrisedTate} with \cref{thm:GMotGPres}.
\end{proof}

Before embarking on the proof proper, let us first give a sketch of the strategy thereof. Suppose we are given a motivic equivalence $\udl{\sC}\rightarrow \udl{\D}$ in $\cat\perfect_H$. Our task is to prove that $\norm^G_H\udl{\sC}\rightarrow\norm^G_H\udl{\D}$ is a motivic equivalence in $\cat\perfect_G$, i.e. we need to show that $\udl{\sU}\loc_G\norm^G_H\udl{\sC}\rightarrow \udl{\sU}\loc_G\norm^G_H\udl{\D}$ is an equivalence in $\Mot\loc_G$. By the isotropy separation from \cref{obs:proper_isotropy_separation_generalities} using the family $\proper$ of proper subgroups of $G$ and inducting on the size of the group, we are essentially reduced to proving that $\Phi^{\proper}\udl{\sU}\loc_G\norm^G_H\udl{\sC}\rightarrow \Phi^{\proper}\udl{\sU}\loc_G\norm^G_H\udl{\D}$ is an equivalence in $\Phi^{\proper}\udl{\Mot}\loc_G$  for the inductive step. We will in fact show that the functor $\Phi^{\proper}\udl{\sU}\loc_G\norm^G_H\colon \cat\perfect_H\rightarrow \Phi^{\proper}\udl{\Mot}\loc_G$ of nonequivariant categories is a Mackey $H$-localizing invariant. 

In turn, by virtue of the behaviour of distributive functors with respect to cofiber sequences as in \cref{prop:cubical_cofibs}, most of the work for this will go into proving (i) that, upon applying $\norm^G_H$ to a $H$-Karoubi sequence, the first map from the cube in the lemma is fully faithful to verify that it is a $G$-Karoubi sequence; and (ii) that the functor $\Phi^{\proper}\udl{\sU}\loc_G\norm^G_H$ erases the undesirable cubical colimits. We formulate these steps precisely in the following key proposition, for which we recall  \cref{nota:filtration_of_cubes} for the filtration of equivariant cubes.

\begin{prop}\label{prop:cubical_inclusions} Consider the map $f:G/H\to G/G$. 
    Suppose we have a Karoubi sequence $\udl{\A}\rightarrow\udl{\B}\rightarrow\udl{\sC}$ in $\cat\perfect_H$. Then:
    \begin{enumerate}[label=(\arabic*)]
        \item the induced functor
    \[\colim_{f_*[1]\setminus  1}f_{\otimes}(\udl{\A}\rightarrow\udl{\B})\longrightarrow f_{\otimes}\udl{\B}\] is fully faithful, with cofiber $f_{\otimes}\udl{\sC}$;
        \item for every $1\leq i\leq |G/H|$, the map 
        \[\colim_{(f_{*}[1])_{\leq i-1}\setminus 1}f_{\otimes}(\udl{\A}\rightarrow\udl{\B}) \longrightarrow \colim_{(f_{*}[1])_{\leq i}\setminus 1}f_{\otimes}(\udl{\A}\rightarrow\udl{\B})\] is fully faithful with cofiber given by a finite sum of properly induced categories.
    \end{enumerate}
\end{prop}
Together, these claims imply that \textit{modulo motives of properly induced categories}, $$f_\otimes \udl{\A}\to f_\otimes \udl{\B} \to f_\otimes \udl{\sC}$$
is a motivic cofiber sequence, thus proving that $\Phi^\mathcal P\udl{\U}\loc\norm^G_H=\Phi^\mathcal P\udl{\U}\loc f_\otimes$ is a localizing invariant (its finitariness being automatic). In more precise terms, we obtain:
\begin{cor}\label{cor:Geometriclocalizing}
    Let $f\colon W \rightarrow V$ be a map between orbits in $\finite_{G}$ and $\udl{\A}\rightarrow\udl{\B}\rightarrow\udl{\sC}$ a Karoubi sequence in $\cat\perfect_W$. Then $f_{\otimes}\udl{\A}\rightarrow f_{\otimes}\udl{\B}\rightarrow f_{\otimes}\udl{\sC}$ is sent to a cofiber sequence in $\Phi^\mathcal P\underline{\Mot}\loc_V$.     In particular, $\Phi^\mathcal P\udl{\U}\loc_Vf_\otimes(-): \cat\perfect_W\to \Phi^\mathcal P\udl{\Mot}\loc_V$ is a finitary Mackey $W$-localizing invariant. 
\end{cor}
\begin{proof}
 By \cref{prop:cubical_inclusions} (1), we get a Karoubi sequence
    $$\colim_{f_*[1]\setminus 1}f_{\otimes}(\udl{\A}\rightarrow\udl{\B})\longrightarrow f_{\otimes}\udl{\B} \longrightarrow f_{\otimes}\udl{\sC}$$ and hence, an equivalence in $\Mot\loc_G$ 
    \[\frac{\udl{\U}\loc_V(f_{\otimes}\udl{\B})}{\udl{\U}\loc_V(\colim_{f_*[1]\setminus  1}f_{\otimes}(\udl{\A}\rightarrow\udl{\B}))} \xlongrightarrow{\simeq} \udl{\U}\loc_V(f_{\otimes}\udl{\sC}).\]  
    
    We claim now that the map \begin{equation}\label{eqn:inductive_equivalence_cube_filter}
    \Phi^\mathcal P\udl{\U}\loc_V(\colim_{(f_{\otimes}[1])_{\leq i-1}\setminus 1}f_{\otimes}(\udl{\A}\rightarrow\udl{\B})) \longrightarrow \Phi^\mathcal P\udl{\U}\loc_V(\colim_{(f_{\otimes}[1])_{\leq i}\setminus 1}f_{\otimes}(\udl{\A}\rightarrow\udl{\B}))
    \end{equation} is an equivalence for all $1\leq i\leq |G/H|$.   Given this, we thus see that $\Phi^\mathcal P\udl{\U}\loc_V(f_{\otimes}\udl{\A})\rightarrow \Phi^\mathcal P\udl{\U}\loc_V(\colim_{f_*[1]\setminus  1}f_{\otimes}(\udl{\A}\rightarrow\udl{\B}))$ is an equivalence, so that $\Phi^\mathcal P\udl{\U}\loc_V(f_{\otimes}\udl{\B})/\Phi^\mathcal P\udl{\U}\loc_V(f_{\otimes}\udl{\A})\rightarrow \Phi^\mathcal P\udl{\U}\loc_V(f_{\otimes}\udl{\sC})$ is indeed an equivalence in $\Phi^\mathcal P\udl{\Mot}\loc_V$, hence proving that $\Phi^\mathcal P\udl{\U}\loc_V f_\otimes$ is a Mackey $G$-localizing invariant -- that it is finitary is immediate from $f_\otimes$ preserving filtered colimits by \cref{rmk:distributivity_over_sifted}, and that the functors $\Phi^{\proper}$ and $\udl{\sU}\loc_V$ do too.

    To prove the claim we note that since $\udl{\U}\loc_V$ is a $V$-localizing invariant, \cref{prop:cubical_inclusions} (2) implies that the cofiber of $\Phi^\mathcal P$ applied to \cref{eqn:inductive_equivalence_cube_filter} is $\Phi^\mathcal P\udl{\U}\loc_V$ applied to a properly induced object, and therefore is $0$. This completes the proof.
\end{proof}

Let us now work towards proving \cref{prop:cubical_inclusions}. The main tool for this will be the isotropy separation criterion \cref{lem:fully_faithful_isotropy_separation} to check fully faithfulness of functors. As such, the proof will be inductive and we first establish it in the nonequivariant base case when $G=e$. This will use the following abstract categorical lemma.

\begin{lem}\label{lem:adjointability_implies_fully_faithful}
Let $I$ be a poset with binary joins, and let $F:I^\triangleright\to \presentable^L$ be a cocone such that for each $i\in I, F(i)\to F(\infty)$ is fully faithful and has a colimit-preserving right adjoint. If for each $i,j\in I$ and $k\in I^{\triangleright}, k\geq i,j$, the following square is vertically right adjointable: 
\[\begin{tikzcd}
	{F(i\wedge j)} & {F(j)} \\
	{F(i)} & {F(k)} 
	\arrow[from=1-1, to=1-2]
	\arrow[from=1-1, to=2-1]
	\arrow[from=1-2, to=2-2]
	\arrow[from=2-1, to=2-2]
\end{tikzcd}\]
then the induced map $\colim_I F(i)\to F(\infty)$ is fully faithful. 
\end{lem}
Note that the adjointability condition for a general $k\in I^\triangleright$ follows from the case of $k=\infty$. Thus, the assumptions are slightly redundant. 
\begin{proof}
We will show that the following square is also vertically right adjointable: 
\[\begin{tikzcd}
	{F(i\wedge j)} & {F(j)} \\
	{F(i)} & {\colim_I F} 
	\arrow[from=1-1, to=1-2]
	\arrow[from=1-1, to=2-1]
	\arrow[from=1-2, to=2-2]
	\arrow[from=2-1, to=2-2]
\end{tikzcd}\]
Given this, we can argue as follows: let $L: \colim_I F\to F(\infty)$ denote the induced functor, and let $\iota_i: F(i)\to\colim_I F$ denote the  inclusions and $f_i: F(i)\to F(\infty)$ the maps induced by $F$. We  prove that the unit $\id \to L^RL$ is an equivalence. The images of $\iota_i:F(i)\to \colim_IF , i\in I$ jointly generate $\colim_I F$ under colimits, so it suffices to prove that $\iota_j^R\iota_i\to \iota_j^R L^R L\iota_i$ is an equivalence for all $i,j\in I$. But $L\iota_i = f_i$ by design, so this is the canonical map $\iota_j^R \iota_i\to f_j^R f_i$, which is an equivalence by adjointability of the two squares, since they are both equivalent to $F(i\wedge j \to j)\circ F(i\wedge j\to i)^R$ (and we can check compatibility of the maps).

In turn, the adjointability of the square with the colimit follows from our adjointability  assumption and the description of colimits in $\presentable^L$ as limits along the right adjoints: consider the transformation $F(i\wedge -)\to F$ of functors $I\rightarrow \presentable^L$ and the induced map on colimits $\colim_{j\in I} F(i\wedge j)\to \colim_{j\in I} F(j)$. By adjointability, we may view this map between colimits as a map $\lim_{j\in I\op} F(i\wedge j)\to \lim_{j\in I\op} F(j)$ in $\presentable^L$ which  is again given levelwise by $F(i\wedge j)\to F(j)$, natural in $j\in I\op$. In particular, this naturality gives us the outer commuting diagrams
\[
\begin{tikzcd}
    F(i) \ar["\simeq", d] \ar[dr]\\
    \lim_{j\in I\op} F(i\wedge j) \rar\dar& \lim_{j\in I\op} F(j)\dar\\
    F(i\wedge j) \rar& F(j)
\end{tikzcd}
\] which were the squares sought after.
\end{proof}

Using this, we can now prove the nonequivariant base case, which is a stable analogue of \cite[Cor. 5.9]{maximePushout}:
\begin{prop}\label{prop:nonequivariant_cube_fully_faithful}
    Let $n\geq 1$ be an integer, and consider a family $f_i:\A_i\to \B_i$ of fully faithful exact functors between stable categories. Consider the induced cube $$C_f: [1]^n \xrightarrow{(f_i)_i} (\cat\perfect)^{\times n} \xrightarrow{\bigotimes_n}\cat\perfect,$$ and for $i\leq n$, let $[1]^n_{\leq i}\subset [1]^n$ denote the full subposet spanned by sequences with a number of $1$'s smaller or equal to $i$. 

For all $i\leq n-1$, the map $\colim_{[1]^n_{\leq i}} C_f\to \colim_{[1]^n_{\leq i+1}} C_f$ is fully faithful. In particular, for $i=n-1$, the map $\colim_{[1]^n\setminus \{1\}} C_f\to \bigotimes_j \B_j$ is fully faithful. 
\end{prop}
\begin{proof}
    Since fully faithfulness may also be checked after taking Ind-completions, we consider the composite $C_f\colon [1]^n\rightarrow \cat\perfect \xrightarrow{\ind} \presentable^L_{\mathrm{st}}$ landing in presentable categories, where the functor $\ind$ is symmetric monoidal. Note that the value of $C_f$ at a morphism $S\rightarrow T$ in $[1]^n$, i.e. an inclusion of subsets of $\udl{n}=\{1,\ldots,n\}$, is given by $\bigotimes_{\overline{s}\in \udl{n}\backslash S}\ind\A_{\overline{s}}\otimes \bigotimes_{s\in S}\ind\B_s\rightarrow \bigotimes_{\overline{t}\in \udl{n}\backslash T}\ind\A_{\overline{t}}\otimes \bigotimes_{t\in T}\ind\B_t$. Since $\ind$ preserves fully faithful functors and sends exact functors between stable categories to morphisms in $\presentable^L$ with colimit-preserving right adjoints, we see that $C_f\colon [1]^n \rightarrow \presentable^L_{\mathrm{st}}$ lands in fully faithful functors which admit colimit-preserving right adjoints. Furthermore, for $S,T\subseteq \udl{n}$, we claim that the square
    \[
    \begin{tikzcd}
        \bigotimes_{\overline{j}\in \udl{n}\backslash (S\cap T)}\ind\A_{\overline{j}}\otimes \bigotimes_{j\in S\cap T}\ind\B_j\rar[hook] \dar[hook]& \bigotimes_{\overline{t}\in \udl{n}\backslash T}\ind\A_{\overline{t}}\otimes \bigotimes_{t\in T}\ind\B_t \dar[hook]\\
        \bigotimes_{\overline{s}\in \udl{n}\backslash S}\ind\A_{\overline{s}}\otimes \bigotimes_{s\in S}\ind\B_s \rar[hook]& \bigotimes_{i\in\udl{n}}\ind\B_i
    \end{tikzcd}
    \]
    is vertically right adjointable. To see this, observe that this square is obtained by tensoring the two morphisms \scriptsize
    \[
    \begin{tikzcd}
        \bigotimes_{S\backslash T}\ind\A\dar[hook]   & \bigotimes_{(S\cup T)^c}\ind\A \otimes \bigotimes_{T\backslash S}\ind\A \otimes \bigotimes_{S\cap T}\ind\B \rar[hook]  & \bigotimes_{(S\cup T)^c}\ind\A \otimes \bigotimes_{T\backslash S}\ind\B \otimes \bigotimes_{S\cap T}\ind\B\\
        \bigotimes_{S\backslash T}\ind\B 
    \end{tikzcd}
    \]\normalsize
    where both (and in particular, the vertical) morphisms have colimit-preserving right adjoints. Since the tensor product in $\presentable^L_{\mathrm{st}}$ preserves internal adjunctions, the vertically right adjointed square is simply obtained by tensoring the horizontal morphism above with the right adjoint of the vertical morphism, and so the adjointed square commutes.
    
    Therefore, by setting $I=[1]^n_{\leq i}$ in \cref{lem:adjointability_implies_fully_faithful}  for $i\leq n-1$, we see that $\colim_{[1]^n_{\leq i}}C_f\rightarrow \bigotimes_{j}\ind \B_j$ is fully faithful for all $i\leq n-1$. Consequently also, by considering the composites $\colim_{[1]^n_{\leq i}}C_f\rightarrow \colim_{[1]^n_{\leq i+1}}C_f\rightarrow  \bigotimes_{j}\ind \B_j$, we see also that the maps $\colim_{[1]^n_{\leq i}}C_f\rightarrow \colim_{[1]^n_{\leq i+1}}C_f$ are fully faithful for $i\leq n-1$. This completes the proof. 
\end{proof}

We are now ready for the proof of the proposition.

\begin{proof}[Proof of \Cref{prop:cubical_inclusions}]
The fully faithfulness in each statement will be done by induction on the order of $G$ and we prove each in turn. To see (1), the statement about the cofiber is immediate from \cref{prop:cubical_cofibs}. We will prove fully faithfulness by a stratification induction via \cref{lem:fully_faithful_isotropy_separation}. The base case of $G=e$ is known by \cref{prop:nonequivariant_cube_fully_faithful}. We use the family of proper subgroups to apply \cref{lem:fully_faithful_isotropy_separation}. By induction, we know the statement to be true upon restriction to all proper subgroups of $G$. 

To wit, let $w\colon W\rightarrow V$ and $u\colon U \rightarrow V$ be  morphisms in $\finite_{G}$, where $u$ is not an equivalence. And let $\udl{\E}$ be an arbitrary $G$-cocomplete $G$-symmetric monoidal category (which we will later specialise to the case of $\udl{\cat}\perfect_G$). Let $F\colon x\rightarrow y$ be a morphism in $\udl{\E}(W)$ and consider the pullback $Z=W\times_VU$ in $\finite_G$ as in \cref{setting:pullback_notations}  with $\overline{w}\colon Z\to U$, $ \overline{u}\colon Z\to W$. Then $u^*$ applied to  $\colim_{w_*[1]-1}w_{\otimes}(x\rightarrow y)\rightarrow w_{\otimes}y$ in $\udl{\sC}(V)$ yields the analogous morphism
    \[\colim_{\overline{w}_{*}[1]-1}\overline{w}_{\otimes}(\overline{u}^*x\rightarrow\overline{u}^*y)\longrightarrow \overline{w}_{\otimes}\overline{u}^*y.\] Specialising to the case $\udl{\E}=\udl{\cat}\perfect_G$, $x= \udl{\A}$, and $y=\udl{\B}$,  this shows the statement upon restriction to proper subgroups by the inductive hypothesis.
    
    Next, to see the statement on geometric fixed points, note by \cite[Obs. 4.2.2]{PD1} that applying $\Phi^{\proper}$ to the map of interest yields the map
    \[\Phi^{\proper}w_{\otimes}\udl{\A}=\colim_{(w_*[1]-1)^G}\Phi^{\proper}w_{\otimes}(\udl{\A}\rightarrow \udl{\B})\longrightarrow \Phi^{\proper}w_{\otimes}\udl{\B}.\] By \cref{lem:norms_preserve_fully_faithfulness}, $w_{\otimes}\udl{\A}\rightarrow w_{\otimes}\udl{\B}$ is fully faithful. By \cref{lem:fully_faithful_isotropy_separation}, we know that $\Phi^{\proper}$ preserves fully faithfulness, whence the desired statement. This completes the proof of (1).

    Point (2) may be proved by a similar isotropy separation inductive proof using instead as an input \cref{lem:inducedness_of_cube_cofibers}, where this time applying $\Phi^{\proper}$ yields the map $\Phi^{\proper}w_{\otimes}\udl{\A}\longrightarrow \Phi^{\proper}w_{\otimes}\udl{\A}$ which is of course an equivalence.
\end{proof}

Armed with  the requisite preparations, we  finally come to the proof of the theorem.

\begin{proof}[Proof of \cref{cor:Nmweq}.]
Since motivic equivalences in $\cat\perfect_V$ are exactly the preimage under $\U\loc_V$ of equivalences, we may instead prove that $\U\loc_V(w_\otimes-)$ inverts motivic equivalences (be warned that it is however \textit{not} a localizing invariant!). 

\textbf{Reduction to orbits:} We observe as a general claim that if $w_i\colon  W_i\to V$ satisfy this conclusion for $i$ in some finite set $I$, then so does $\coprod_i w_i\colon  \coprod_i W_i\to V$. Indeed, in this case $(\coprod_i w_i)_\otimes \simeq \bigotimes_i (w_i)_\otimes$, so it suffices to show that ordinary tensor products preserve motivic equivalences. The proof of this fact is the same as in \cite[Lem. 5.5, Thm. 5.8]{BGTMult} (see also \cite[Prop. 4.2]{MotLoc}) and simply follows from the fact that for any $V$-stable $V$-category $\sC$, $\sC\otimes -\colon  \udl{\cat}\perfect_V\to \udl{\cat}\perfect_V$ preserves Karoubi sequences. Thus, without loss of generality, we may assume $V$ is an orbit, and we work by descending induction on the size of $V$ (i.e. by induction on the size of stabilizers in $V$). 

\textbf{Induction hypothesis:} We assume the result holds for every map $w\colon  W\to V'$ where $V'$ is a strictly larger orbit than $V$ (i.e. strictly smaller isotropy subgroup), and we prove it for our given map $w\colon W\to V$.  We may then write $W$ as a coproduct of orbits, and by reducing to orbits as above, we may also assume that $W$ is an orbit.  We now argue by isotropy separation:
 
 \textbf{Geometric fixed points: } By \Cref{cor:Geometriclocalizing},  $\Phi^\mathcal P\U\loc_V(w_\otimes -)$ is a finitary Mackey $W$-localizing invariant, and so it sends Mackey motivic equivalences to equivalences. By \Cref{cor:Mackeyweq=weq}, it therefore sends motivic equivalences to motivic equivalences.   Thus it suffices to prove that for every proper orbit map $u\colon U\to V$, $u^*\U\loc_V(f_\otimes -)$ inverts motivic equivalences. 
    
\textbf{Proper subgroups:}
    Letting $Z$ denote the pullback $W\times_V U$, with $\overline{w}\colon Z\to U$ and $ \overline{u}\colon Z\to W$ as in \cref{setting:pullback_notations}, we find $u^*\U\loc_V(f_\otimes -)\simeq \U\loc_U(u^*f_\otimes -)\simeq \U\loc_U(\overline{w}_\otimes \overline{u}^*-)$.    Now $U\to V$ is proper and hence $U$ is a strictly larger orbit, so by induction, $\U\loc_U (\overline{w}_\otimes -)$ inverts motivic equivalences and since $\overline{u}^*$ preserves motivic equivalences, we are done. 
\end{proof}

\begin{rmk}
    The presentably symmetric monoidal category $\Phi^{\proper}\udl{\Mot}\loc_G$ has played only an auxiliary role in the story so far, but we think that it is a category that deserves further attention in matters pertaining to isotropy separation methods for equivariant algebraic K-theory. For example, by the universal property of the map $\udl{\Mot}\loc_G\rightarrow \Phi^{\proper}\udl{\Mot}\loc_G$ from \cref{cons:categorical_isotropy_separation}, we immediately get from \cref{thm:GMotGPres} that
    \[(\Phi^{\proper}\udl{\sU}\loc_G)^*\colon \udl{\func}^L(\Phi^{\proper}\udl{\Mot}\loc_G, \udl{\E})\xlongrightarrow{\simeq}\udl{\func}^{\mathrm{fin. loc.}}(\udl{\cat}\perfect_G,\udl{\E})\] for all $\udl{\E}\in \presentable^L_{G,\mathrm{st}}$ in the image of the inclusion $\presentable^L_{\mathrm{st}}\subseteq \presentable^L_{G,\mathrm{st}}$, i.e. $\Phi^{\proper}\udl{\Mot}\loc_G$ is the receptacle of the universal finitary $G$-localizing invariant which vanishes on properly induced categories.
    
    While it might appear totally uncontrollable, observe nonetheless, for example, that \cref{thm:motnorm} implies that it has \textit{another} structure of a $\Mot\loc$--commutative algebra in $\presentable^L_{\mathrm{st}}$ apart from the one coming from the universal property of $\Mot\loc$ as in \cref{lm:adjmotGmot}.    To wit, the multiplicative norm yields the symmetric monoidal functor 
    \begin{equation}\label{eqn:motloc_to_geom_fix_point}
        \Mot\loc\xlongrightarrow{\norm^G_e} \Mot\loc_G \xlongrightarrow{\Phi^{\proper}} \Phi^{\proper}\udl{\Mot}\loc_G.
    \end{equation} This preserves all colimits by applying \cref{cor:Geometriclocalizing} in the case of $f\colon G/e\rightarrow G/G$, the universal property of $\Mot\loc$, and that $\udl{\sU}\loc_G$ commutes with the norms by \cref{thm:motnorm}.
    
    By \cref{cor:corep_verdier_quot} below, the tensor unit of $\Phi^{\proper}\udl{\Mot}\loc_G$ corepresents $\Phi^GK_G(-)$, and so the map \cref{eqn:motloc_to_geom_fix_point} may be viewed as a categorification of the map $K((-)^e)\rightarrow \Phi^GK_G(-)$ of functors $\calg_G(\udl{\cat}\perfect_G)\rightarrow \calg(\spectra)$ coming from applying $\Phi^G$ to $\norm^G_e\res^G_eK_G(-)\rightarrow K_G(-)$ and recalling that $\Phi^G\norm^G_e\simeq \id$.
\end{rmk}

\subsection{Corepresentedness of localizing equivariant K--theory}\label{subsec:corepresented_K-theory}

We start by recording the equivariant analogue of \cite[Prop. 9.26]{BGT}, namely corepresentability of equivariant $K$-theory in $G$-motives, specifically by the symmetric monoidal unit. Since \Cref{thm:motnorm} provides $\udl{\Mot}\loc_G$ with a $G$-symmetric monoidal structure, this will canonically endow $K$-theory with a $G$-lax symmetric monoidal structure. 

\begin{prop}\label{prop:corep}
    The functor $\udl{K}_G\colon  \udl{\cat}\perfect_G \to \myuline{\spectra}_G$ given in \Cref{ex:KG} is equivalent to the mapping $G$-spectrum functor $$\myuline{\mapsp}_{\udl{\Mot}\loc_G}(\udl{\U}\loc_G(\myuline{\spectra}^\omega_G), \udl{\U}\loc_G(-))\simeq \myuline{\mapsp}_{\udl{\Mot}\loc_G}(\one_{\udl{\Mot}\loc_G},\udl{\U}\loc_G(-))$$
    where the equivalence is since $\myuline{\spectra}^\omega_G$ is the unit in ${\cat}\perfect_G$. In particular,  for any $\udl{\A},\udl{\B}\in\udl{\cat}\perfect_G$ with $\udl{\A}$ being dualizable, we have an identification 
    \[\myuline{\mapsp}_{\udl{\Mot}\loc_G}(\udl{\U}\loc_G(\udl{\A}), \udl{\U}\loc_G(\udl{\B}))\simeq \udl{K}_G(\udl{\func}\exact(\udl{\A},\udl{\B})).\]
\end{prop}
For the proof, it is convenient to construct the following general piece of structure: 
\begin{lem}\label{lm:adjmotGmot}
    The symmetric monoidal adjunction  $\myuline{\spectra}_G^\omega \otimes - \colon \cat\perfect \rightleftarrows {\cat}\perfect_G \cocolon (-)^G$ afforded by \cref{thm:equivmod} induces an adjunction between colimit-preserving functors $$\U\loc_G(\myuline{\spectra}^\omega_G)\otimes - \colon \Mot\loc \rightleftarrows \Mot\loc_G \cocolon \U\loc((-)^G)$$ which participates in a commuting square of symmetric monoidal categories
    \[
    \begin{tikzcd}
        \cat\perfect \ar[rr,"\myuline{\spectra}_G^\omega \otimes - "]\dar["\sU\loc"'] & &\cat\perfect_G\dar["\sU\loc_G"] \\
        \Mot\loc \ar[rr,"\U\loc_G(\myuline{\spectra}^\omega_G)\otimes -"] && \Mot\loc_G
    \end{tikzcd}
    \]
    whose underlying square of categories is horizontally right adjointable.
\end{lem}
\begin{proof}
    We only have to show that the functors $\myuline{\spectra}^{\omega}_G\otimes-$ and $(-)^G$ preserve motivic equivalences. The functors in the induced adjunction preserve colimits by the universal properties of $\Mot\loc$ and $\Mot\loc_G$ among presentable categories, and that we obtain a square of symmetric monoidal functors is since the composite $\sU\loc_G\circ \myuline{\spectra}^{\omega}_G\otimes(-)$ has a symmetric monoidal structure. 
    
    If $E$ is a (finitary) Mackey $G$-localizing invariant, then $E(\myuline{\spectra}_G^\omega\otimes -)$ is a (finitary) localizing invariant, so the left adjoint sends motivic equivalences of stable categories to Mackey motivic equivalences of $G$-stable $G$-categories, and hence, by \Cref{cor:Mackeyweq=weq}, to motivic equivalences of $G$-stable $G$-categories. Conversely, if $E$ is a (finitary) localizing invariant, $E((-)^G)$ is a Mackey $G$-localizing invariant, so that $(-)^G$ sends Mackey motivic equivalences of $G$-stable $G$-categories to motivic equivalences of stable categories, and hence by \Cref{cor:Mackeyweq=weq}, sends motivic equivalences of $G$-stable $G$-categories to motivic equivalences of stable categories. 
\end{proof}

\begin{proof}[Proof of \Cref{prop:corep}]
By \Cref{lem:fixmackadj}, and since $\myuline{\spectra}_G= \udl{\mackey}_G(\spectra)$, it suffices to prove that the mapping spectrum functor $\mapsp_{\Mot\loc_G}(\U\loc_G(\myuline{\spectra}_G^\omega), \U\loc_G(-))\colon \cat\perfect_G\to \spectra $ is equivalent to $K((-)^G)$. But then we have  equivalences
\[\mapsp_{\Mot\loc_G}(\U\loc_G(\myuline{\spectra}_G^\omega), \U\loc_G(-))\simeq \mapsp_{\Mot\loc}(\U\loc({\spectra}^\omega), \U\loc((-)^G)) \simeq K((-)^G)\]
where the first step is by  the adjunction and symmetric monoidality of $\U\loc_G(\myuline{\spectra}_G^\omega)\otimes-$ from \cref{lm:adjmotGmot}, and the second  is by \cite[Prop. 9.26]{BGT}. This completes the first part.

Lastly, since $\udl{\sU}\loc_G$ is symmetric monoidal,  we get equivalences $\udl{\sU}\loc_G(\udl{\func}\exact(\udl{\A},\udl{\B}))\simeq \udl{\sU}\loc_G(\udl{\A}^{\vee}\otimes\udl{\B})\simeq \udl{\sU}\loc_G(\udl{\A})^{\vee}\otimes\udl{\sU}\loc_G(\udl{\B})$. Hence, we get $\myuline{\mapsp}_{\udl{\Mot}\loc_G}(\udl{\U}\loc_G(\udl{\A}), \udl{\U}\loc_G(\udl{\B}))\simeq \myuline{\mapsp}_{\udl{\Mot}\loc_G}(\udl{\U}\loc_G(\myuline{\spectra}^{\omega}), \udl{\sU}\loc_G(\udl{\A})^{\vee}\otimes\udl{\U}\loc_G(\udl{\B}))\simeq \udl{K}_G(\udl{\func}\exact(\udl{\A},\udl{\B}))$, which completes the proof.
\end{proof}
We thus obtain the remaining part in \cref{alphThm:K_is_G-laxsym} announced in the introduction:
\begin{cor}\label{cor:consKlaxG}
    The (nonconnective) equivariant algebraic $K$-theory functor $\udl{K}_G: \udl{\cat}\perfect_G\to \myuline{\spectra}_G$ has a unique  $G$-lax symmetric monoidal refinement. Furthermore, the canonical map $(-)^\simeq \to \Omega^\infty \udl{K}_G$ has a unique lax $G$-symmetric monoidal refinement. 
\end{cor}
\begin{proof}
    By \cref{thm:motnorm}, lax $G$-symmetric monoidal structures on $\udl{K}_G$ as a functor out of $\udl{\cat}\perfect_G$ are the same as those as a functor out of $\udl{\Mot}\loc_G$. By \cref{prop:corep}, the associated functor out of motives is $\myuline{\mapsp}_{\udl{\Mot}\loc_G}(\one_{\udl{\Mot}\loc_G},-)$, which is the  initial $G$–commutative algebra in $\udl{\func}\exact(\udl{\Mot}\loc,\myuline{\spectra})$ under Day convolution, whence the uniqueness statement.
    
    The last claim follows as this map is of the form $\myuline{\map}_{\udl{\A}}(\one,-)\to \myuline{\map}_{\udl{\B}}(\one,-) $ for some $G$-symmetric monoidal functor between $G$-symmetric monoidal $G$-categories $\udl{\A}\to\udl{\B}$, namely the $G$-symmetric monoidal functor $\udl{\sU}\loc_G\colon \udl{\cat}\perfect_G\rightarrow \udl{\Mot}\loc_G$.
\end{proof}

\begin{example}
    By  \cref{example:algebra_objects_catperfGI}, for any $G$-operad $\udl{\orbit}$ and any $R\in \algebraCategory_{\udl{\orbit}\otimes \everythingAlgebra_1}(\myuline{\spectra})$, the $G$-spectrum $\{K(\perfectCat_{\res^G_HR}(\spectra_H))\}_{H\leq G}$ refines to a $\udl{\orbit}$-ring $G$-spectrum.
    
    By  \cref{example:algebra_objects_catperfGII}, for any $\sC\in\calg(\cat\perfect)$ equipped with a symmetric monoidal $G$-action, the $G$-spectrum $\udl{K}_G(\udl{\borel}(\sC))=\{G/H\mapsto K(\sC^{hH})\}_{H\leq G}$ refines to a normed $G$-ring spectrum, i.e. an object in $\calg_G(\myuline{\spectra})$.
\end{example}

Another useful consequence of \cref{prop:corep} is the following:

\begin{cor}\label{cor:compactness_of_unit_in_motloc}
    The unit $\unit_{\udl{\Mot}\loc_G}=\udl{\sU}\loc_G(\myuline{\spectra}^{\omega})\in\udl{\Mot}\loc_G$ is compact.
\end{cor}
% \begin{proof}
%     This is a fiberwise statement, and so without loss of generality, we just need to show that $\mapsp_{\Mot\loc_G}(\udl{\sU}\loc_G(\myuline{\spectra}^{\omega}),-)\colon \Mot\loc_G\rightarrow \spectra$ preserves filtered colimits. By the adjunction from \Cref{lm:adjmotGmot}, this is equivalent to $\mapsp_{\Mot\loc}({\sU}\loc({\spectra}^{\omega}),\sU\loc_G((-)^G))$ where $\sU\loc_G((-)^G)\colon \Mot\loc_G\rightarrow\Mot\loc$ preserves colimits. By \cite[Theorem 9.36]{BGT}, $\mapsp_{\Mot\loc}({\sU}\loc({\spectra}^{\omega}),-)$ preserves filtered colimits, and so we are done. 
% \end{proof}
\begin{proof}
    By definition, we have a fully faithful embedding $$(\udl{\sU}\loc_G)^*\colon \udl{\func}(\udl{\Mot}\loc_G,\myuline{\spectra}_G)\hookrightarrow\udl{\func}(\udl{\cat}\perfect_G,\myuline{\spectra}_G)$$ which, by  \cref{thm:GMotGPres}, restricts to an equivalence 
    $$(\udl{\sU}\loc_G)^*\colon \udl{\func}^L(\udl{\Mot}\loc_G,\myuline{\spectra}_G)\xrightarrow{\simeq }\udl{\func}^{\mathrm{fin.loc.}}(\udl{\cat}\perfect_G,\myuline{\spectra}_G).$$ Since $\udl{K}_G $ is a finitary $G$-localizing invariant and since $\udl{K}_G(-)\simeq~(\udl{\U}\loc_G)^*\myuline{\mapsp}_{\udl{\Mot}\loc_G}(\udl{\U}\loc_G(\myuline{\spectra}^\omega_G), -)$ by \cref{prop:corep}, we thus see that $\myuline{\mapsp}_{\udl{\Mot}\loc_G}(\udl{\U}\loc_G(\myuline{\spectra}^\omega_G), -)$ preserves $G$-colimits, and in particular, $\udl{\U}\loc_G(\myuline{\spectra}^\omega_G)$ is compact.
\end{proof}

\begin{cor}\label{cor:corep_verdier_quot}
    The functor $\Phi^GK_G\colon \cat\perfect_G\rightarrow\spectra_G\rightarrow\spectra$ is corepresented by the tensor unit of $\Phi^{\proper}\udl{\Mot}\loc_G$. In particular,  for any $\udl{\A},\udl{\B}\in\udl{\cat}\perfect_G$ with $\udl{\A}$ being dualizable, we have an identification $\myuline{\mapsp}_{\Phi^{\proper}\udl{\Mot}\loc_G}(\Phi^{\proper}\udl{\U}\loc_G(\udl{\A}), \Phi^{\proper}\udl{\U}\loc_G(\udl{\B}))\simeq \Phi^G\udl{K}_G(\udl{\func}\exact(\udl{\A},\udl{\B}))$.
\end{cor}
\begin{proof}
    We simply compute that
    \begin{equation*}
        \begin{split}
            \Phi^GK_G(-) &= \big[\widetilde{E\proper}\otimes \myuline{\mapsp}_{\udl{\Mot}\loc_G}(\unit_{\udl{\Mot}\loc_G},\udl{\sU}\loc_G(-))\big]^G\\
            &\xrightarrow{\simeq} \myuline{\mapsp}_{\udl{\Mot}\loc_G}(\unit_{\udl{\Mot}\loc_G},\widetilde{E\proper}\otimes\udl{\sU}\loc_G(-))^G\\
            &{\simeq}\myuline{\mapsp}_{\Phi^{\proper}\udl{\Mot}\loc_G}(\unit_{\Phi^{\proper}\udl{\Mot}\loc_G},\widetilde{E\proper}\otimes\udl{\sU}\loc_G(-))^G 
        \end{split}
    \end{equation*}
    where the first equivalence is by \cref{prop:corep} and the second by \cref{cor:compactness_of_unit_in_motloc} which implies that the mapping $G$-spectrum out of $\unit_{\udl{\Mot}\loc_G}$ commutes with all equivariant colimits, and the last equivalence is since $\Phi^{\proper}\colon \udl{\Mot}\loc_G\rightarrow \Phi^{\proper}\udl{\Mot}\loc_G$ is a Bousfield localization.
\end{proof}

\section{Applications}\label{section:applications}

\subsection{Borel localizing motives and tensor powers}\label{subsec:borel_motives}

\begin{defn}
    A functor $\func(BG,\cat\perfect)\rightarrow \E$ to a stable category $\E$ is a \textit{Borel $G$-localizing invariant} if it is additive and it sends Karoubi sequences to bifiber sequences. It is furthermore said to be \textit{finitary} if it preserves filtered colimits. A morphism in $\func(BG,\cat\perfect)$ is said to be a \textit{Borel motivic equivalence} if it is sent to an equivalence by all finitary Borel $G$-localizing invariants out of $\func(BG,\cat\perfect)$. 
\end{defn}

\begin{obs}\label{obs:borelification_preserves_moteq}
    The Borelification functor $b^* \colon \cat\perfect_G\rightarrow \func(BG,\cat\perfect)$ sends Mackey motivic equivalences to Borel motivic equivalences since it clearly sends Karoubi sequences to Karoubi sequences and it preserves filtered colimits.
\end{obs}

\begin{obs}\label{obs:orbits_preserves_moteq}
    The functor $(-)_{hG}\colon \func(BG,\cat\perfect)\rightarrow \cat\perfect$ sends Karoubi sequences to Karoubi sequences and it preserves filtered colimits. To see this, since $(-)_{hG}$ is a left adjoint, it of course preserves cofiber sequences and filtered colimits. To see that it additionally takes Verdier  sequences to fiber sequences, let $\A \rightarrow \B\rightarrow \sC$ be a Karoubi sequence. Since $\A_{hG} \rightarrow \B_{hG}\rightarrow \sC_{hG}$ is still a cofiber sequence, we only have to show that the functor $\A_{hG} \rightarrow \B_{hG}$ is fully faithful. To see this, simply note that under the equivalence $\ind\colon \cat\perfect\rightleftharpoons \presentable^L_{\mathrm{st},\omega}\cocolon (-)^{\omega}$, the fact that the inclusion $\presentable^L_{\mathrm{st},\omega}\subset \presentable^L_{\mathrm{st}}$ preserves colimits, and that $\presentable^L_{\mathrm{st}}$ is $\infty$-semiadditive, that the map $\A_{hG} \rightarrow \B_{hG}$ is identified with $((\ind\A)^{hG})^{\omega} \rightarrow ((\ind\B)^{hG})^{\omega}$. And this functor is certainly fully faithful since $(\ind\A)^{hG}\rightarrow (\ind\B)^{hG}$ is so.
\end{obs}

\begin{prop}\label{cor:tensor_power_preserve_moteq}
    Let $X\in\finite_G$. The functor $(-)^{\otimes X}\colon \cat\perfect\rightarrow \func(BG,\cat\perfect)$ sends motivic equivalences to Borel motivic equivalences. Consequently, the endofunctor $((-)^{\otimes X})_{hG}\colon \cat\perfect\rightarrow \cat\perfect$ preserves motivic equivalences.
\end{prop}
\begin{proof}
    Since $X$ is a finite coproduct $\coprod_iG/H_i$ of orbits, we know that $(-)^{\otimes X}\simeq \bigotimes_i(-)^{\otimes G/H_i}$. Hence, it suffices to deal with the case when $X= G/H$ for some subgroup $H\leq G$. In this case, first note that there is a factorization 
    \[(-)^{\otimes G/H}\colon \cat\perfect \xrightarrow{\const_H} (\cat\perfect)^{BH} \xhookrightarrow{b_!} \cat\perfect_H \xrightarrow{\norm^G_H} \cat\perfect_G \xrightarrow{b^*}(\cat\perfect)^{BG}.\] This is because \cite[Thm. 2.4.10 (2)]{kaifNoncommMotives} ensures that the composite $\cat\perfect_H \xrightarrow{\norm^G_H} \cat\perfect_G \xrightarrow{b^*}(\cat\perfect)^{BG}$ is equivalent to $\cat\perfect_H \xrightarrow{b^*} (\cat\perfect)^{BH} \xrightarrow{\norm^G_H}(\cat\perfect)^{BG}$, and $b^*b_!\simeq \id$.
    
    Next, note that the constant functor $\const_H\colon \cat\perfect \rightarrow \func(BH,\cat\perfect)$ sends motivic equivalences to Borel motivic equivalences since it sends Karoubi sequences to Karoubi sequences and it preserves filtered colimits. Similarly, since $b_!\colon (\cat\perfect)^{BH}\hookrightarrow \cat\perfect_H\subset \mackey_H(\cat\perfect)$ is levelwise given by taking the appropriate homotopy orbits,  $b_!$ sends Borel motivic equivalences to Mackey motivic equivalences by \cref{obs:orbits_preserves_moteq}. Moreover, by \cref{obs:borelification_preserves_moteq}, $b^*\colon \cat\perfect_G\rightarrow (\cat\perfect)^{BG}$ sends Mackey motivic equivalences to Borel motivic equivalences. Hence,  together with \cref{cor:Nmweq}, we learn that $(-)^{\otimes G/H}\colon \cat\perfect\rightarrow (\cat\perfect)^{BG}$ sends motivic equivalences to Borel motivic equivalences, as required.

    The last statement about $((-)^{\otimes X})_{hG}$ preserving motivic equivalences is now an immediate combination of the first statement together with the fact that $(-)_{hG}\colon \func(BG,\cat\perfect)\rightarrow \cat\perfect$ sends Borel motivic equivalences to motivic equivalences since $(-)_{hG}$ preserves Karoubi sequences by \cref{obs:orbits_preserves_moteq}.
\end{proof}

\begin{warning}
    While it was true that $(-)^{\otimes X}$ and $((-)^{\otimes X})_{hG}$ preserved the appropriate notions of motivic equivalences, it is \textit{not} true that $((-)^{\otimes X})^{hG}$ does since $(-)^{hG}$ does not preserve Karoubi sequences, or filtered colimits. On the other hand, Efimov shows in \cite[Cor. 3.30]{efimovII} that this can be remedied by considering instead \textit{dualizable} fixed points. Since $BG$ is countable, we find that in conjunction with \cite[Thm. 4.26]{efimovIII}, $((-)^X)^{hG,\mathrm{dual}}$ does preserve motivic equivalences.   
\end{warning}

While this will not strictly be needed in the paper, we nevertheless clarify the relationship between the genuine equivariant localizing motives $\Mot\loc_G$ that we have been considering so far with the notion of \textit{Borel} localizing motives, i.e. the category $\Mot^{\mathrm{loc},\mathrm{bor}}_G$ defined as the Dwyer-Kan localization of $\func(BG,\cat\perfect)$ against the Borel motivic equivalences.

\begin{prop}\label{prop:borelification_motives_square}
    The category $\Mot^{\mathrm{loc},\mathrm{bor}}_G$ is a Dwyer--Kan localization of $\Mot\loc_G$, and this localization fits into a commuting square of Dwyer--Kan localizations
    \[
    \begin{tikzcd}
        \cat\perfect_G \rar\dar& \Mot\loc_G\dar\\
        \func(BG,\cat\perfect)\rar & \Mot^{\mathrm{loc},\mathrm{bor}}_G.
    \end{tikzcd}
    \]
\end{prop}
\begin{proof}
    The existence of the square is clear by the universal property of $\Mot\loc_G$ and the fact that $\cat\perfect_G\rightarrow \func(BG,\cat\perfect)$ preserves Karoubi sequences and filtered colimits. The fact that the morphism $\Mot^{\mathrm{loc}}_G\rightarrow \Mot^{\mathrm{loc},\mathrm{bor}}_G$ is a Dwyer--Kan localization is now also clear, since the functor $\cat\perfect_G\to \func(BG,\cat\perfect)$ is.
\end{proof}
\begin{rmk}
Another perspective here, coming from \Cref{thm:equivmod}, is that while $\cat\perfect_G$ can be described as the category of compactly generated modules over $\spectra_G$, $\func(BG,\cat\perfect)$ can be described as the category of compactly generated $\spectra^{BG}$-modules. The localization functor $\Mot\loc_G\to \Mot^{\mathrm{loc},\mathrm{bor}}_G$ can therefore be viewed as given by basechange along the Borelification functor $\spectra_G\to \spectra^{BG}$. 
\end{rmk}

\begin{warning}
    Be warned however that $\Mot^{\mathrm{loc},\mathrm{bor}}_G$ is \textit{not} the Borelification of $\Mot\loc_G$. By the square from \cref{prop:borelification_motives_square} and the fact that $\func(BG,\cat\perfect)\rightarrow\func(BG,\Mot\loc)$ is a Borel $G$-localizing invariant, the Bousfield localization $\Mot\loc_G\rightarrow \func(BG,\Mot\loc)$ \textit{does} however factor uniquely through $\Mot^{\mathrm{loc},\mathrm{bor}}_G$. 
\end{warning}

\begin{rmk}
    In fact, the square \cref{prop:borelification_motives_square} can be assembled to a square of objects in $\calg_G(\udl{\presentable}^L_{G,\mathrm{st}})$. Implicit in this claim is that $\udl{\Mot}^{\mathrm{loc},\mathrm{bor}}$ admits the structure of multiplicative norms (which is a straightforward consequence of \cref{cor:tensor_power_preserve_moteq}). The remarks in the preceding warning  then amounts to the statement that the map $\udl{\Mot}\loc\rightarrow \udl{\Mot}^{\mathrm{loc},\mathrm{bor}}$ induces an equivalence $\udl{\borel}(\udl{\Mot}\loc)\rightarrow \udl{\borel}(\udl{\Mot}^{\mathrm{loc},\mathrm{bor}})$. We have  chosen not to present this subsection as such since we would like to emphasize  that our main theorem has consequences also in the unparametrized and nonequivariant setup.
\end{rmk}

We end this subsection by giving a new proof of the key nonformal step in the construction of the polynomial functoriality on algebraic K-theory by Barwick--Glasman--Mathew--Nikolaus, thus completing the proof of \cref{alphCor:symmetric_powers}. 

\begin{nota}
    Let $\sC,\D$ be stable categories and $n\geq 1$. We write $\func\nrex(\sC,\D)$ for the full subcategory of reduced $n$-excisive functors, i.e. the reduced functors which send strongly cocartesian  $(n+1)$-cubes to cartesian  $(n+1)$-cubes. Recall from Goodwillie calculus that when $\sC$ is small and $\D$ is presentable, the inclusion $\func\nrex(\sC,\D)\subseteq \func(\sC,\D)$ admits a left adjoint $P_n$ and a right adjoint $P^n$.
\end{nota}

\begin{lem}\label{lem:describing_Gamma_functor}
    The composite $\sC\xhookrightarrow{\yoneda}\func(\sC\op,\spc)\xrightarrow{\susps_+}\func(\sC\op,\spectra)\xrightarrow{P_n}\func\nrex(\sC\op,\spectra)$ lands in compact objects and the induced functor $\sC\rightarrow \func\nrex(\sC\op,\spectra)^{\omega}$ identifies with the initial reduced $n$-excisive functor $\sC\rightarrow \Gamma_n\sC$ from \cite[Cons. 4.16]{BGMN}.
\end{lem}
\begin{proof}
    The fact that the functor lands in compact objects is clear since $\susps_+$ and $L_n$ preserve compact objects: the former because its right adjoint $\loops$ preserves filtered colimits, and the latter because the right adjoint of $P_n$ admits a further right adjoint.

    To see that $\func\nrex(\sC\op,\spectra)^{\omega}$ has the required universal property, let $S$ be the set of morphisms  in $\func(\sC\op,\spectra)$ given precisely by the morphism comparing the zero object in $\sC$ and that in $\func(\sC\op,\spectra)$ as well as those described as follows: for a strongly cocartesian $(n+1)$-cube $D\colon [1]^{n+1}\rightarrow\sC$, the morphism $\colim_{[1]^{n+1}\backslash 1}\susps_+\yoneda D \rightarrow \susps_+\yoneda D(1)$ is in $S$. By abuse of notation, since $\susps_+$ preserves colimits, we will also let $S$ denote the analogous set of morphisms in $\func(\sC\op,\spc)$. Then note, by the Yoneda lemma, that $\func\nrex(\sC\op,\spectra)$ is precisely the $S$-local Bousfield subcategory of $\func(\sC\op,\spectra)$. Hence, for any $\D\in\cat\perfect$, we may now compute:
    \small
    \begin{equation*}
        \begin{split}
            \func\exact(\func\nrex(\sC\op,\spectra)^{\omega},\D) &\simeq \func^{\mathrm{ex},\omega}(\func\nrex(\sC\op,\spectra)^{\omega},\ind\D)\simeq \func^{L,\omega}(\func\nrex(\sC\op,\spectra),\ind\D)\\
            &\xrightarrow[\simeq]{P_n^*} \func^{L,\omega}_{S^{-1}}(\func(\sC\op,\spectra),\ind\D)\\
            &\xrightarrow[\simeq]{(\susps_+)^*} \func^{L,\omega}_{S^{-1}}(\func(\sC\op,\spc),\ind\D)\\
            &\xrightarrow[\simeq]{\yoneda^*} \func^{n\mathrm{-ex},\omega}(\sC,\ind\D) \simeq \func\nrex(\sC,\D).
        \end{split}
    \end{equation*}\normalsize
    Here, $\func^{\mathrm{ex},\omega}$ denotes exact functors which land in compact objects, and analogously for $\func^{n\mathrm{-ex},\omega}$; $\func^{L,\omega}$ denotes colimit-preserving functors which preserve compact objects. All the equivalences are clear by the respective definitions and universal properties. Hence, we see that $\func\nrex(\sC\op,\spectra)^{\omega}$ has the same universal property as $\Gamma_n\sC$, as wanted.
\end{proof}

\begin{lem}\label{lem:karoubi_seq_Gamma_functors}
    Let $\sC\in\cat\perfect$ and $n\geq 1$. There is a Karoubi sequence $\sC^{\otimes (n+1)}_{h\Sigma_{n+1}}\rightarrow \Gamma_{n+1}\sC\rightarrow \Gamma_n\sC$. Consequently, $\Gamma_n(-)\colon \cat\perfect\rightarrow \cat\perfect$ preserves motivic equivalences.
\end{lem}
\begin{proof}
    By Goodwillie's classification of homogeneous functors, we obtain the stable recollement
    \[
    \begin{tikzcd}
        \func\nrex(\sC\op,\spectra) \rar[hook]& \func^{(n+1)\mathrm{-lex}}(\sC\op,\spectra) \rar[two heads] \lar[two heads, bend right = 20]\lar[two heads, bend right = -20]& \func\exact\big((\sC^{\otimes(n+1)}_{h\Sigma_{n+1}})\op,\spectra\big). \lar[hook, bend right = 20]\lar[hook, bend right = -20]
    \end{tikzcd}
    \] In particular, the maps in the top row preserve compact objects. Thus, applying $(-)^{\omega}$  to it and invoking \cref{lem:describing_Gamma_functor} gives a Karoubi sequence 
    \[\func\exact\big((\sC^{\otimes(n+1)}_{h\Sigma_{n+1}})\op,\spectra\big)^{\omega} \longrightarrow \Gamma_{n+1}\sC\longrightarrow \Gamma_n\sC.\] But now recall that Ind-completions of stable categories may also be computed as $\ind(-)\simeq \func\exact((-)\op,\spectra)$, and hence the first term is $\big(\ind(\sC^{\otimes(n+1)}_{h\Sigma_{n+1}})\big)^{\omega}\simeq \sC^{\otimes(n+1)}_{h\Sigma_{n+1}}$.

    The second statement now is a straightforward induction on $n$, using \cref{cor:tensor_power_preserve_moteq} as the main ingredient.
\end{proof}

The following is the key result of Barwick--Glasman--Mathew--Nikolaus, cf. \cite[Prop. 4.18]{BGMN}, from which their main Theorem 4.9 follows rather formally as explained in \textit{loc. cit}.

\begin{prop}\label{prop:BGMN_key}
    Let $n\geq 1$. The functor $\Gamma_n\colon \cat\perfect\rightarrow\cat\perfect$ preserves universal $K$-equivalences.
\end{prop}
\begin{proof}
    Combining the  transformations $(-)\rightarrow \Gamma_n(-)$  and $(-)^{\simeq}\rightarrow\loops K$, we have the square  
    \[
    \begin{tikzcd}
        (-)^{\simeq} \rar\dar & \big(\Gamma_n(-)\big)^{\simeq}\dar\\
        \loops K(-) \rar[dashed] & \loops K\big(\Gamma_n(-)\big)
    \end{tikzcd}
    \]
    of functors $\cat\perfect \rightarrow\spc$. Since $\Gamma_n$ preserves motivic equivalences by \cref{lem:karoubi_seq_Gamma_functors}, we see that $\loops K\big(\Gamma_n(-)\big) $ is a motivic invariant. By the motivic universal property of $\loops K$ from \cite{MotLoc} (but not as an additive invariant!), we obtain a unique dashed extension. On the other hand, for $\sC,\D\in\cat\perfect$, since $\D\rightarrow \Gamma_n\D$ is reduced $n$-excisive, we have the functor  $\func\exact(\sC,\D)\rightarrow \func\nrex(\sC,\Gamma_n\D)\simeq \func\exact(\Gamma_n\sC,\Gamma_n\D)$. Since this functor is itself reduced $n$-excisive, it induces an exact functor $\Gamma_n\func\exact(\sC,\D)\rightarrow  \func\exact(\Gamma_n\sC,\Gamma_n\D)$ natural in $\sC$ and $\D$.

    Therefore, combining all these maps, we obtain the diagram
    \[
    \begin{tikzcd}
        \func\exact(\sC,\D)^{\simeq} \rar\dar & \big(\Gamma_n\func\exact(\sC,\D)\big)^{\simeq}\rar\dar & \func\exact(\Gamma_n\sC,\Gamma_n\D)^{\simeq}\dar \\
        \loops K\big(\func\exact(\sC,\D)\big)\rar & \loops K\big(\Gamma_n\func\exact(\sC,\D)\big)\rar & \loops K\big(\func\exact(\Gamma_n\sC,\Gamma_n\D)\big).
    \end{tikzcd}
    \] In particular, when $\D=\sC$, tracing where $\id_{\sC}\in\func\exact(\sC,\sC)^{\simeq}$ goes, we see that $\pi_0 K\big(\func\exact(\sC,\sC)\big)\rightarrow \pi_0K\big(\func\exact(\Gamma_n\sC,\Gamma_n\sC)\big)$  sends $[\id_{\sC}]$ to $[\id_{\Gamma_n\sC}]$. Hence, we see that if $f\colon \sC\rightarrow \D$ and $g\colon \D\rightarrow \sC$ are inverse universal $K$-equivalences, so are $\Gamma_nf\colon \Gamma_n\sC\rightarrow \Gamma_n\D$ and $\Gamma_ng\colon \Gamma_n\D\rightarrow \Gamma_n\sC$, as required.
\end{proof}

\subsection{\texorpdfstring{$G$}--symmetric monoidality of splitting motives}\label{subsec:splitting_motives}
In this section, we explain how to \textit{deduce} from our results the existence of a $G$-symmetric monoidal structure on the equivariant version of splitting motives. 
\begin{rmk}\label{rmk:counterintuitive_splitting}
It seems somewhat counterintuitive to prove a result about localizing motives and deduce one about splitting motives, since the latter are usually thought of as easier. However, the key technical input in our proof, namely \Cref{prop:cubical_inclusions}, seems harder to establish in the split case: one would have to prove that the relevant maps are not only fully faithful, but also have adjoints (upon assuming, of course, that the map $\udl{\A}\to\udl{\B}$ has the relevant adjoint). This is more tedious than establishing fully faithfulness. Therefore, while it is possible to give a direct proof, we opt for the path of least resistance and simply deduce the splitting results from the localizing results. 
\end{rmk}
The basic idea for the deduction is that splitting motives (up to finitary-ness questions) are obtained from stable categories by inverting so-called ``universal $K$-theory equivalences'', cf. \cite[Def. 4.4]{BGMN}. Our construction of norms provides, inter alia, a map $$N_H^GK_H(\udl{\func}_H\exact(\udl{\A},\udl{\B})) \to K_G(\udl{\func}_G\exact(N_H^G\udl{\A},N_H^G\udl{B}))$$ compatible with compositions and natural in $\udl{\A},\udl{\B}$. This will directly imply that $N_H^G$ preserves universal $K$-theory equivalences, so that the general strategy from the end of the previous section will simply work here. So, all in all, the desired construction will be a direct corollary of the fact that the connective cover of nonconnective $K$-theory is connective $K$-theory. 

We now give more details, starting with variants of splitting motives. 

\begin{defn}\label{defn:splitting_invariants}
    A \textit{$G$-splitting invariant} is a $G$-semiadditive $G$-functor $\udl{E}: \udl{\cat}\perfect_G\to \udl{\E}$ to some $G$-additive $G$-category  $\udl{\E}$ sending split Karoubi sequences to products - more precisely, so that for every split Karoubi sequence $\udl{\A}\rightleftharpoons \udl{\B}\rightleftharpoons\udl{\D}$, the induced map $\udl{E}(\udl{\B})\to \udl{E}(\udl{\A})\times \udl{E}(\udl{\D})$ is an equivalence. 

    A \textit{Mackey splitting invariant} is a functor $E: {\cat}\perfect_G\to \E$ to some additive category sending split Karoubi sequences to products. 

    A \textit{universal $K_G$-theory} equivalence is a map in ${\cat}\perfect_G$ sent to an equivalence by any $G$-splitting invariant.

    We have the obvious analogous definitions for finite $G$-sets $V$. 
\end{defn}

\begin{rmk}\label{rmk:nomackeymotivicsplit}
 Note that the proof of \Cref{cor:Mackeyweq=weq} applies here too and so universal $K_G$-theory are equivalently those maps sent to an equivalence by any Mackey splitting invariant. The proof of \Cref{lm:weqsubGcat} also applies and so universal $K_G$-theory equivalences are preserved by restriction functors.
\end{rmk}

We now define two $G$-categories of splitting motives, one of which is some kind of absolute variant with no finitary-ness or cocompleteness condition of Blumberg--Gepner--Tabuada's original definition of ``additive motives'', and the other one is a  variant which recovers an additive (as opposed to stable) variant of Blumberg--Gepner--Tabuada's original definition. There are many other variants, including stable variants (we only impose additivity), but the proof we give generalizes to these other variants quite easily.  
\begin{defn}
    The $G$-category of \textit{absolute splitting $G$-motives} $\udl{\Mot}_G^{\spl,\abs}$ is the localization of the $G$-category $\udl{\cat}\perfect_G$ at the class of universal $K_G$-equivalences. The localization functor is denoted $\udl{\U}_G^{\spl,\abs}$.     The $G$-category of \textit{finitary splitting $G$-motives} $\udl{\Mot}_G^{\spl,\finiteSmall}$ is defined as the sifted cocompletion $\udl{\presheaf}_{\Sigma}(\udl{{\Mot}}_G^{\spl,\finiteSmall,\omega})$ where $\udl{{\Mot}}_G^{\spl,\finiteSmall,\omega}$ is the localization of the $G$-category $(\udl{\cat}\perfect_G)^{\omega}$ at the class of universal $K_G$-equivalences. The  functor $\udl{\cat}\perfect_G\simeq \ind\big((\udl{\cat}\perfect_G)^{\omega}\big)\rightarrow \udl{\Mot}^{\spl,\finiteSmall}_G$ induced by the composition $(\udl{\cat}\perfect_G)^{\omega}\xrightarrow{\mathrm{loc.}}\udl{{\Mot}}_G^{\spl,\finiteSmall,\omega}\xhookrightarrow{\yoneda}\udl{{\Mot}}_G^{\spl,\finiteSmall}$ is denoted $\udl{\U}_G^{\spl,\finiteSmall}$.     There are obvious analogous definitions of absolute/finitary splitting $V$-motives for $V$ a finite $G$-set. 
\end{defn}

\begin{rmk}
    By construction and \Cref{lm:Gaddloc}, these $G$-categories satisfy the expected universal properties which we collect at the end of the subsection. 
\end{rmk}

As before, to construct a $G$-symmetric monoidal category on absolute splitting $G$-motives, it suffices to prove that universal $K$-theory equivalences are stable under norm functors. We do this below after observing the compatibility of our definition of universal $K_G$-equivalence with the classical one from \cite{BGMN} (thus also justifying the name):
\begin{lem}\label{lm:univKeqdesc}
Let $V$ be a finite $G$-set. 
    A map ${f}\colon \udl{\A}\to \udl{\B}$ of $V$-stable $V$-categories is a universal $K_V$-equivalence if and only if there exists a ${g}\colon \udl{\B}\to \udl{\A}$ such that \[[{g}\circ {f}] = [\id_{\udl{\A}}] \in \pi_0(K_V(\udl{\func}_V\exact(\udl{\A},\udl{\A})))\quad\text{and}\quad[{f}\circ {g}] = [\id_{\udl{\B}}] \in \pi_0(K_V(\udl{\func}_V\exact(\udl{\B},\udl{\B}))).\]    
\end{lem}
\begin{proof}
By \cite[Lem. 4.1.12]{kaifNoncommMotives}, if $E:\cat\perfect_V\to \E$ is any $V$-Mackey splitting invariant, the induced map $\pi_0\map_{\cat\perfect_V}(\udl{\A},\udl{\A})\to \pi_0\map_{\E}(E(\udl{\A}),E(\udl{\A}))$ factors through $\pi_0\map_{\cat\perfect_V}(\udl{\A},\udl{\A})\to \pi_0(K_V(\udl{\func}\exact_V(\udl{\A},\udl{\A})))\cong K_0(\func\exact_V(\udl{\A},\udl{\A}))$ (and similarly for $\udl{\B}$). Therefore the existence of ${g}$ as indicated guarantees that $E({f})$ is an equivalence with inverse $E({g})$. This proves that ${f}$ is a universal $K_V$-equivalence, cf. \Cref{rmk:nomackeymotivicsplit}. 

Conversely, let ${f}$ be a Mackey motivic equivalence. Now, $\pi_0(K_V(\udl{\func}_V\exact(\udl{\B},-)))\colon \cat\perfect_V \to\mathrm{Ab}$ is a $V$-Mackey splitting invariant, and hence by assumption ${f}$ induces an equivalence \[\pi_0(K_V(\udl{\func}_V\exact(\udl{\B},\udl{\A})))\to \pi_0(K_V(\udl{\func}_V\exact(\udl{\B},\udl{\B}))).\]
Picking a preimage of $[\id_{\udl{\B}}]$ and a lift of that preimage to $\func\exact_V(\udl{\B},\udl{\A})^\simeq$ provides a ${g}$ with $[{f}\circ {g}] = [\id_{\udl{\B}}]$. Now, $\pi_0(K_V(\udl{\func}_V\exact(\udl{\A},-)))\colon \cat\perfect_V \to\mathrm{Ab}$ is also a $V$-Mackey splitting invariant, so that \[\pi_0(K_V(\udl{\func}_V\exact(\udl{\A},\udl{\A})))\to \pi_0(K_V(\udl{\func}_V\exact(\udl{\A},\udl{\B})))\]
is also an isomorphism. In particular, since $[{g}\circ {f}]$ and $[\id_{\udl{\A}}]$ have the same image under this map, we find that they must be equal. This proves the claim overall.
\end{proof}

\begin{cor}\label{cor:norm_of_K-equiv}
    Let $w: W\to V$ be a map of finite $G$-sets. The functor $w_\otimes: \cat\perfect_W\to \cat\perfect_V$ sends universal $K_W$-equivalences to universal $K_V$-equivalences. 
\end{cor}
\begin{proof}
    Consider the following composite in $\spectra_V$, natural in $\udl{\A},\udl{\B}\in \cat\perfect_W$: 
    \begin{equation}\label{eqn:map_on_K-theory_of_functors}
        w_\otimes\udl{K}_W(\udl{\func}\exact_W(\udl{\A},\udl{\B})) \to \udl{K}_V(w_\otimes\udl{\func}\exact_W(\udl{\A},\udl{\B}))\to \udl{K}_V(\udl{\func}\exact_V(w_\otimes\udl{\A}, w_\otimes\udl{\B})).
    \end{equation}

    Here, the first map comes from the lax $G$-symmetric monoidal structure on $\udl{K}$, constructed in \Cref{cor:consKlaxG}, while the second map comes from adjointing over the normed evaluation map $w_\otimes\udl{\A}\otimes w_\otimes\udl{\func}\exact_W(\udl{\A},\udl{\B})\simeq w_\otimes(\udl{\A}\otimes \udl{\func}\exact_W(\udl{\A},\udl{\B}))\rightarrow w_\otimes\udl{\B}$. 

    The  map \cref{eqn:map_on_K-theory_of_functors} is compatible with compositions: since we have a map $(-)^{\simeq}\rightarrow \udl{\mathcal{K}}\coloneqq\loops\udl{K} $ of $G$-lax symmetric monoidal functors $\udl{\cat}\perfect\rightarrow \udl{\spc}$ by \cref{cor:consKlaxG}, we get a  diagram
    \[
    \begin{tikzcd}
        w_*\udl{\func}\exact_W(\udl{\A},\udl{\B})^{\simeq}\times w_*\udl{\func}\exact_W(\udl{\B},\udl{\sC})^{\simeq} \rar\dar["\circ"] & \loops\Big(w_\otimes\udl{K}_W(\udl{\func}\exact_W(\udl{\A},\udl{\B}))\otimes w_\otimes\udl{K}_W(\udl{\func}\exact_W(\udl{\B},\udl{\sC}))\Big)\dar["\mathrm{lax}"]\\
        w_*\udl{\func}\exact_W(\udl{\A},\udl{\sC})^{\simeq} \dar["\mathrm{lax}"]& \loops\udl{K}_V(w_\otimes\udl{\func}\exact_W(\udl{\A},\udl{\B})\otimes w_\otimes\udl{\func}\exact_W(\udl{\A},\udl{\sC}))\dar["\circ"]\\
        \big(w_{\otimes}\udl{\func}\exact_W(\udl{\A},\udl{\sC})\big)^{\simeq}\rar \dar& \loops\udl{K}_V(w_\otimes\udl{\func}\exact_W(\udl{\A},\udl{\sC}))\dar\\
        \udl{\func}\exact_V(w_{\otimes}\udl{\A},w_{\otimes}\udl{\sC})^{\simeq}\rar & \loops\udl{K}_V(\udl{\func}\exact_V(w_{\otimes}\udl{\A},w_{\otimes}\udl{\sC}))
    \end{tikzcd}
    \]
    whence the compatibility with compositions. Furthermore, for $\udl{\A}= \udl{\B}$, the map \cref{eqn:map_on_K-theory_of_functors} sends the identity to the identity  because the composite
    \small\[\map_W(\myuline{\spectra}^{\omega},\udl{\func}\exact_W(\udl{\A},\udl{\A}))\xrightarrow{w_{\otimes}}\map_V(\myuline{\spectra}^{\omega},w_{\otimes}\udl{\func}\exact_W(\udl{\A},\udl{\A}))\rightarrow \map_V(\myuline{\spectra}^{\omega},\udl{\func}\exact_V(w_{\otimes}\udl{\A},w_{\otimes}\udl{\A}))\]\normalsize may easily be seen to identify with the map $w_{\otimes}\colon \map_W(\udl{\A},\udl{\A})\rightarrow \map_V(w_{\otimes}\udl{\A},w_{\otimes}\udl{\A})$. 
    
    It thus follows directly that if ${f},{g}$ are maps of $W$-stable $W$-categories as in \Cref{lm:univKeqdesc}, then  $w_\otimes {f}, w_\otimes {g}$ are maps of $V$-stable $V$-categories , so that, by \Cref{lm:univKeqdesc}, if $f$ is a universal $K_W$-equivalence, then $w_\otimes f$ is a universal $K_V$-equivalence. 
\end{proof}
To state the main result of this subsection, it will be convenient to use the following notations: let $\udl{\A}$ be a $G$-additive $G$-category which admits filtered colimits; we write $\udl{\func}^{\spl}(\udl{\cat}\perfect_G,\udl{\A}), \udl{\func}^{\spl, \finiteSmall}(\udl{\cat}\perfect_G,\udl{\A})\subseteq \udl{\func}(\udl{\cat}\perfect_G,\udl{\A})$ for the full subcategories of splitting invariants and finitary splitting invariants (i.e. those that additionally preserve filtered colimits), respectively. We also write $\udl{\func}^{\times}\subseteq \udl{\func}$ for the full subcategory of functors preserving finite indexed products.

\begin{cor}\label{cor:norms_on_splitting_motives}
     The $G$-category $\udl{\Mot}^{\spl,\abs}_G$ (resp. $\udl{\Mot}^{\spl,\finiteSmall}_G$) is the universal target for absolute (resp. finitary) $G$-splitting invariants. That is, it is $G$-additive (resp. and also $G$-presentable), and for any $G$-additive categories $\udl{\A}, \udl{\B}$ where $\udl{\B}$ is also $G$-cocomplete, the maps 
    \[(\udl{\U}_G^{\spl,\abs})^*\colon \udl{\func}^{\times}(\udl{\Mot}^{\spl,\abs}_G,\udl{\A}) \xlongrightarrow{\simeq } \udl{\func}^{\spl}(\udl{\cat}\perfect_G,\udl{\A})\]
    \[(\udl{\U}_G^{\spl,\finiteSmall})^*\colon \udl{\func}^{L}(\udl{\Mot}^{\spl,\finiteSmall}_G,\udl{\B}) \xlongrightarrow{\simeq } \udl{\func}^{\spl, \finiteSmall}(\udl{\cat}\perfect_G,\udl{\B})\] are equivalences.  

     Furthermore, there exists a unique $G$-symmetric monoidal structure on $\udl{\Mot}^{\spl,\abs}_G$ (resp. $\udl{\Mot}^{\spl,\finiteSmall}_G$) compatible with the norm $G$-symmetric monoidal structure on $\udl{\cat}\perfect_G$ (resp. $(\udl{\cat}\perfect_G)^{\omega}$) under the localization $G$-functor $\udl{\U}^{\spl,\abs}_G$ (resp. $\udl{\U}^{\spl,\finiteSmall}_G$).
\end{cor}
\begin{proof}
    From \cref{cor:norm_of_K-equiv} and the fact that the inclusion $(\udl{\cat}\perfect_G)^{\omega}\subseteq \udl{\cat}\perfect_G$ inherits the $G$-symmetric monoidal structure by \cref{cor:closure_under_norms_of_compacts} in the finitary case, we immediately deduce the symmetric monoidality statement. By the same argument as in the proof of \cref{thm:GMotGPres}, using \cref{rmk:nomackeymotivicsplit} and that $\ind^K_H$ preserves filtered colimits and split Karoubi sequences, we see that both the motivic $G$-categories are $G$-semiadditive. That they are furthermore additive may be checked fiberwise and proceeds exactly as in \cite[Lem. 6.6]{MotLoc}. Furthermore, the $G$-presentability of $\udl{\Mot}^{\spl,\finiteSmall}_G$ is also immediate by its construction as a sifted cocompletion.
    
    We are left with showing the first statement about universal properties.  We only deal with the finitary case since the other case is similar but easier. Now, we have a fully faithful functor $(\udl{\U}_G^{\spl,\finiteSmall,\omega})^*\colon \udl{\func}^{\times}(\udl{\Mot}^{\spl,\finiteSmall,\omega}_G,\udl{\A}) \hookrightarrow \udl{\func}((\udl{\cat}\perfect_G)^{\omega},\udl{\A})\simeq \udl{\func}^{\finiteSmall}(\udl{\cat}\perfect_G,\udl{\A})$ which lands in the full subcategory of splitting invariants since $(\udl{\cat}\perfect_G)^{\omega}\rightarrow \udl{\Mot}^{\spl,\finiteSmall,\omega}_G$ was itself a splitting invariant, and\footnote{This is the crucial difference with the localizing case.} any split Karoubi sequence is a filtered colimit of split Karoubi sequences between compact $G$-stable $G$-categories by \cite[Cor. 4.1.9]{kaifNoncommMotives}. To see that this is essentially surjective on splitting invariants $(\udl{\cat}\perfect_G)^{\omega}\rightarrow\udl{\A}$, simply note that any such splitting invariant will in particular invert the universal $K_G$-theory equivalences by definition, and therefore it factors through a functor $\udl{\Mot}^{\spl,\finiteSmall,\omega}_G\rightarrow \udl{\A}$ which preserves finite indexed (co)products by  \cref{lm:Gaddloc}. Together with the universal property $\udl{\func}^{L}(\udl{\Mot}^{\spl,\finiteSmall}_G,\udl{\A})\simeq \udl{\func}^{\sqcup}(\udl{\Mot}^{\spl,\finiteSmall,\omega}_G,\udl{\A})\simeq  \udl{\func}^{\times}(\udl{\Mot}^{\spl,\finiteSmall,\omega}_G,\udl{\A})$  of sifted cocompletions by \cite[Thm. E]{shahPaperII}, this completes the proof.
\end{proof}

\subsection{Equivariant THH and the Dennis trace}\label{subsec:THH}
In this section, we explain how to define equivariant $\THH$, and how to equip it with a $G$-symmetric monoidal structure. This will, in particular, automatically render the Dennis trace $G$-symmetric monoidal as well. 

\begin{lem}\label{lm:Gdbl}
    Let $\udl{\A}$ be a $G$-stable $G$-category. The $G$-presentable $G$-stable $G$-category $\udl{\ind}(\udl{\A})$ is dualizable in $\presentable^L_{G,\mathrm{st}}$ for the tensor product from \cite[Prop. 2.2.19]{kaifNoncommMotives}.

    Furthermore, for a $G$-exact functor $\A\to \B$, the induced functor $\udl{\ind}(\A)\to\udl{\ind}(\B)$ is an internal left adjoint, i.e. its right adjoint is $G$-colimit preserving.
\end{lem}
\begin{proof}
    Indeed, by \Cref{thm:equivmod} and since $\ind$ is fiberwise, it suffices to prove that $\ind(\A_G)$ is dualizable over $\spectra_G$. Since the latter is rigid over $\spectra$ in the sense of \cite[Def. 4.34]{ramzirigid} (see Example 4.40 in \textit{loc. cit.}), we see by \cite[Prop. 4.17]{ramzirigid} that it suffices to prove that $\ind(\A_G)$ is dualizable in $\presentable^L_{\mathrm{st}}$. In turn, this is a classical fact recalled, e.g. in \cite[Prop. D.7.2.3]{lurieSAG}. 

The second part is easy to verify: the left adjoint preserves fiberwise compacts so the right adjoint preserves fiberwise filtered colimits, and it is already  $G$-exact. 
\end{proof}
 It follows that the formalism of \cite[\textsection 2]{HSS} provides the following: 
 \begin{cons}
   We obtain a symmetric monoidal functor $\Tr_G: {\cat}\perfect_G \to \spectra_G$, see \cite[Def.'s 2.9 \& 2.11]{HSS} (in fact, with a refinement to $\spectra_G^{BS^1}$ see Theorem 2.14 in \textit{loc. cit.}) which, on objects, is given by the symmetric monoidal dimension.

Since the construction from \cite[\textsection 2]{HSS} is natural in symmetric monoidal $2$-functors, and the restriction maps $\presentable^L_{G,\mathrm{st}} \to \presentable^L_{H,\mathrm{st}}$ are such functors, we obtain a symmetric monoidal $G$-functor $$\udl{\Tr}_G: \udl{\cat}\perfect_G\to \myuline{\spectra}_G$$ 
which we also denote by $\udl{\mathrm{THH}}_G$ and call ``equivariant topological Hochschild homology'', or equivariant $\mathrm{THH}$ for short. 

 \end{cons}
\begin{prop}
    $\udl{\THH}_G: \udl{\cat}\perfect_G\to \myuline{\spectra}_G$ is a finitary $G$-localizing invariant. 
\end{prop}
\begin{proof}
    That it sends Karoubi sequences to bifiber sequences follows from \cite[Thm. 3.4]{HSS}, and that it is finitary is standard, see also \cite[Cor. 4.25]{HSS}.

    We now prove $G$-semiadditivity  by induction via  isotropy separation. So suppose the restriction of $\udl{\THH}_G$ to any proper subgroup $K\lneq G$ is $K$-semiadditive, and consider, for a proper subgroup $H$ and an $H$-stable $H$-category $\udl{\A}$, the canonical map $$\induced_H^G\udl{\THH}_H(\udl{\A})\to \udl{\THH}_G(\induced_H^G\udl{\A})$$
(note that it is sufficient to prove that this map is an equivalence, since $\udl{\THH}_G$ is clearly fiberwise semiadditive)

Since restriction to proper subgroups is a strong symmetric monoidal functors, it therefore commutes with symmetric monoidal dimensions, so that for any proper subgroup $K<G$, $\res_K^G$ applied to the above map is a special case of a $K$-semiadditivity comparison map for $\udl{\THH}_K$, which is therefore an equivalence by assumption. Thus, to conclude, we may take geometric fixed points of either side. 

For the left hand side, the result is $0$ by definition of geometric fixed points and since $H<G$ is a proper subgroup, while for the right hand side, it is also $0$: this is because geometric fixed points is also a strong symmetric monoidal functor and hence also commutes with symmetric monoidal dimensions, so that we are left with calculating $\Phi^{\proper}\ind_H^G\udl{\A}=  0$, which is again true by construction since $H< G$ was proper (cf. \cref{cons:categorical_isotropy_separation,obs:proper_isotropy_separation_generalities}). 
\end{proof} 
\begin{rmk}\label{rmk:THHgenuine}
    It is not hard to see that $(\THH_G)^G: \cat\perfect_G\to \spectra$ is not of the form $E((-)^G)$ for any ordinary localizing invariant $E: \cat\perfect\to \spectra$. Indeed, $E$ can be recovered from $E((-)^G)$ by plugging in induced $G$-stable $G$-categories, so by $G$-semiadditivity of $\udl{\THH}_G$ we would have $E(\sC) \simeq E((\induced_e^G \sC)^G) \simeq (\THH_G(\induced_e^G\sC))^G \simeq (\induced_e^G \THH(\sC))^G\simeq \THH(\sC)$, i.e. $E\simeq\THH$. 

    However, $\THH_G(\spectra_G)\simeq \sphere_G$ and the $G$-fixed points of the latter satisfies a tom Dieck splitting with terms of the form $\Sigma^\infty_+ BW_G(H)$. On the other hand, one can show that $\THH(\spectra_G)$ satisfies a tom Dieck splitting with terms of the form $\Sigma^\infty_+LBW_G(H)$. We explain this for $G=C_p$, and refer to the proof of \cite[Thm. 1.6]{KSpGMaxime} for the general case. 

    For $G=C_p$, we use the isotropy separation localization sequence $\perfectCat(\sphere[C_p])\to\spectra^\omega_{C_p}\to \spectra^\omega$ to deduce a bifiber sequence of spectra $$\THH(\sphere[C_p])\to \THH(\spectra^\omega_{C_p})\to \THH(\spectra^\omega).$$ This bifiber sequence splits via the unit map $\spectra^\omega\to\spectra^\omega_{C_p}$, and combining it with the usual formula $\THH(\sphere[G]) \simeq \Sigma^\infty_+ LBG$, we obtain the desired formula: $$\THH(\spectra^\omega_{C_p})\simeq \sphere\oplus\Sigma^\infty_+LBC_p.$$
\end{rmk}
\begin{rmk}
    Our functor $\udl{\THH}_G$ is a reasonable object: when $R$ is a connective ring $G$-spectrum, $\pi_0^G\THH_G(R)$ is given by the quotient of $\udl{\pi}_0(R)$ by the commutator \textit{Mackey subfunctor} (as opposed to the levelwise commutator sub-group). 

    Indeed, combining \cite[Prop. 4.47]{ccryTraces} with \Cref{thm:equivmod} we find that $\THH_G(R)\simeq R\otimes_{R\otimes R\op}R$ in $\spectra_G$. Since $R$ is connective, we deduce that $\udl{\pi}_0\THH_G(R)\cong \udl{\pi}_0(R)\otimes_{\udl{\pi_0}(R)\otimes\udl{\pi}_0(R)\op}\udl{\pi}_0(R)$ in Mackey functors. The latter is a coequalizer $$\udl{\pi}_0(R)\otimes\udl{\pi}_0(R)\otimes\udl{\pi}_0(R)\op\otimes\udl{\pi}_0(R)\rightrightarrows \udl{\pi}_0(R)\otimes\udl{\pi}_0(R)$$ of the two maps informally described by $x\otimes a\otimes b\otimes y\mapsto bxa \otimes y$, resp. $x\otimes ayb$ by definition. From this, the description as $\udl{\pi}_0(R)$ modulo the commutator Mackey subfunctor is clear. 
\end{rmk}

Next, we upgrade the construction of $\udl{\THH}_G$ to make it $G$-symmetric monoidal. We again take advantage of the naturality of the construction of \cite[\textsection 2]{HSS} to reduce to one simple claim: that $\udl{\presentable}^L_{G,\mathrm{st}}$ can be viewed as a Mackey functor in (very large) $2$-categories, with emphasis on ``$2$'' (indeed, in 
\cite[Thm. B and E]{brankoKaifNatalie}, it is already constructed as a $G$-symmetric monoidal $G$-category, i.e. as a Mackey functor in (very large) categories, cf. \cite[Thm. 2.3.9]{nardinShah}). Since this is a technical point, we have deferred its precise statement and proof to the appendix as \cref{prop:mackey2}.

 Informally, this construction makes the norm functors $\norm_H^G: \presentable^L_{H,\mathrm{st}}\to \presentable^L_{G,\mathrm{st}}$ into symmetric monoidal $2$-functors. The advantage of this is that this automatically implies their compatibility with symmetric monoidal traces, functorially. This is ultimately what we leverage below:
\begin{cor}\label{cor:dennis_trace}
    The finitary $G$-localizing invariant $\udl{\THH}_G$ can be upgraded to a $G$-symmetric monoidal $G$-functor $\udl{\cat}_G\perfect\to \myuline{\spectra}^{BS^1}_G$. 

    In particular, it factors uniquely as a $G$-symmetric monoidal $G$-colimit-preserving functor $\udl{\THH}_G: \udl{\Mot}_G\to \myuline{\spectra}^{BS^1}_G$, and therefore receives a unique ($S^1$-equivariant) $G$-symmetric monoidal natural transformation from equivariant $K$-theory as $G$-functors $\udl{\cat}\perfect_G\rightarrow\myuline{\spectra}_G$, which we call the equivariant Dennis trace map $$\udl{K}_G\to \udl{\THH}_G.$$
\end{cor}
\begin{proof}
Since $\udl{\cat}\perfect_G\to \udl{\Mot}_G$ is a symmetric monoidal Dwyer--Kan localization whose unit corepresents equivariant $K$-theory, all the statements follow  from the first one. 

For this statement, we note that \cite[Eq. 2.10, Def. 2.11, Thm. 2.14]{HSS} constructs a finite-product-preserving functor $\cmonoid(\CAT_2)\rightarrow \func([1],\cmonoid(\cat))$ which assigns to $B\in \commmonoid(\CAT_2)$ the transformation $\Tr_B\colon B^\dual\to \Endomorphism(\one_B)^{BS^1}$. On the other hand, by \cref{prop:mackey2}, we have a finite-product-preserving functor $(\udl{\boldPr}^L_{G,\mathrm{st}})^\otimes \colon \spancategory(\finite_G)\to \commmonoid(\CAT_2)$. Note here that for each $H\leq G$, the tensor unit in $\presentable^L_{H,\mathrm{st}}$ is $\spectra_H$. Hence, we may compose these two functors to obtain a finite-product-preserving functor $\spancategory(\finite_G)\rightarrow \func([1],\cmonoid(\cat))$, i.e. a $G$-symmetric monoidal functor $\Tr_G\colon (\udl{\presentable}^L_{G,\mathrm{st}})^\dual\rightarrow \myuline{\spectra}_G^{BS^1}$.

Next, recall that by \Cref{lm:Gdbl}, the nonfull $G$-symmetric monoidal functor $\ind\colon \udl{\cat}\perfect_G\to \udl{\presentable}^L_{G,\mathrm{st}}$ lifts to a fully faithful $G$-symmetric monoidal functor $\ind\colon \udl{\cat}\perfect_G\to (\udl{\presentable}^L_{G,\mathrm{st}})^\dual$. Therefore, all in all, we obtain a $G$-symmetric monoidal functor 
\[\udl{\cat}\perfect_G\xhookrightarrow{\ind} (\udl{\presentable}^L_{G,\mathrm{st}})^\dual\xrightarrow{\Tr_G} \myuline{\spectra}_G^{BS^1}\] whose restriction to $\finite_G\op$ is $\THH_G$ as constructed in the beginning of the section. 
\end{proof}

There are a few versions of equivariant THH in the literature which we expect to be distinct from ours, cf. for instance \cite{angeltveitTCNorms,realTHH,chan2025trace}. We close this subsection by highlighting several attractive and naturally desirable formal features of the present construction. For one, it interacts pleasantly with  geometric fixed points, as encapsulated by the following:

\begin{prop}\label{prop:geometric_fixed_THH}
   There is a natural  equivalence of symmetric monoidal functors
   \[\Phi^G\udl{\THH}(-)\simeq \THH(\Phi^{\proper}-)\colon \cat\perfect_G\longrightarrow \spectra^{BS^1}.\] 
   Furthermore, for any operad $\mathcal{O}$ and $R\in \algebraCategory_{\mathcal{O}\otimes\everythingAlgebra_1}(\spectra_G)$, there is an equivalence 
    \[\Phi^G\udl{\THH}(R)\simeq \THH(\Phi^GR)\] of $\mathcal{O}$-algebras in spectra with $S^1$-action.
\end{prop}
 \begin{proof}
 For the first equivalence, we simply note that $\Phi^G:\presentableGstable\to \presentablestable$ is a morphism of commutative $\presentablestable$-algebras, and hence a symmetric monoidal $2$-functor. Thus, by naturality of traces as constructed in \cite[\textsection 2]{HSS}, we have a commutative diagram in $\cmonoid(\cat)$

 \[\begin{tikzcd}
	{(\presentableGstable)^\dual} & {(\presentablestable)^\dual} \\
	{\spectra_G^{BS^1}} & {\spectra^{BS^1}}
	\arrow["{\Phi^\proper}", from=1-1, to=1-2]
	\arrow["{\udl{\THH}}"', from=1-1, to=2-1]
	\arrow["\THH", from=1-2, to=2-2]
	\arrow["{\Phi^G}", from=2-1, to=2-2]
\end{tikzcd}\]
as desired.    For the other equivalence, we observe that the functor $\Phi^G\colon \lmodule_R(\spectra_G)\rightarrow \lmodule_{\Phi^GR}(\spectra)$ identifies $\Phi^{\proper}\lmodule_R(\spectra_G) $ as $\lmodule_{\Phi^GR}(\spectra)$. This is because $\lmodule_R(\spectra_G)$ is compactly generated by the set $\{\ind^G_H\res^G_HR\}_{H\leq G}$, and $\Phi^G\colon \lmodule_R(\spectra_G)\rightarrow \lmodule_{\Phi^GR}(\spectra)$ is the Bousfield localization which precisely kills the orbits $\{\ind^G_H\res^G_HR\}_{H\lneq G}$ and sends the remaining generator $R\in\spectra_G$ to $\Phi^GR\in\spectra$. Hence, $\lmodule_{\Phi^GR}(\spectra)$ satisfies the universal property of $\Phi^{\proper}\lmodule_R(\spectra_G)$, whence the observation. Combining with the commuting square above, we obtain the second equivalence.
\end{proof}

On the geometric side, let us record the following expected result. It further illustrates a certain ``regularity'' of the $\udl{\THH}_G$ construction which is not enjoyed by the pointwise THH functor $\cat\perfect_G\subset \mackey_G(\cat\perfect)\xrightarrow{\THH}\mackey_G(\spectra)=\spectra_G$.

\begin{prop}\label{prop:free_loop_space}
    Let $X\in\spc_G$. Then $\udl{\THH}_G(\udl{\func}(\udl{X},\myuline{\spectra})^{\omega})\simeq \sphere_G[\mathcal{L}X]\in\spectra_G$ where $\mathcal{L}X\in\spc_G$ is the free loop $G$-space $\map_{\spc_G}(S^1,X)\in\spc_G$ on $X$ (with the trivial $G$-action on $S^1$).
\end{prop}
\begin{proof}
    We simply follow \cite[\textsection 4.3]{ccryTraces}: note first that $\udl{\func}(-,\myuline{\spectra})\colon \spc_G \rightarrow \presentable^L_{G,\mathrm{st}}$ with left Kan extension functoriality is a bivariant theory in the sense of Definition 3.2 from \textit{loc. cit.} since the Beck-Chevalley basechange property is satisfied because $\myuline{\spectra}$ is $G$-cocomplete. By \cite[Cor. 2.2.20]{kaifNoncommMotives} for example, this functor also enhances to a symmetric monoidal functor. By \cite[Cor. 3.12]{ccryTraces}, we thus deduce the desired conclusion.
\end{proof}

\begin{rmk}
    From the formula above, we see in particular that 
    \[\Phi^G\udl{\THH}_G(\udl{\func}(\udl{X},\myuline{\spectra})^{\omega})\simeq {\THH}({\func}({X}^G,{\spectra})^{\omega})\] since the left hand side is equivalent to $\Phi^G\sphere_G[\mathcal{L}X]\simeq \sphere[(\mathcal{L}X)^G]\simeq \sphere[\mathcal{L}(X^G)]$ and this is $\THH(\func(X^G,\spectra)^{\omega})$. Alternatively, this can also be deduced directly from \cref{prop:geometric_fixed_THH} using that $\Phi^{\proper}\udl{\func}(\udl{X},\myuline{\spectra})\simeq \udl{\func}(\udl{X},\udl{\func}_*(\udl{E\proper},\myuline{\spectra}))\simeq \udl{\func}(\udl{X},\Phi^{\proper}\myuline{\spectra})$, which at level $G$, evaluates to $\func(X^G,\spectra)$.
\end{rmk}

As in the nonequivariant setting, using the free loop space formula, we may prove an equivariant version of A-theory splitting first studied in \cite{monaCaryCobordism}.\footnote{We do not claim that the version of equivariant A-theory we write here agrees with theirs, although we certainly expect this to be the case.}

\begin{prop}\label{prop:A-theory_splitting}
    Let $X\in\spc_G$. Then $\sphere_G[X]$ is a direct summand of $\udl{K}_G(\udl{\func}(\udl{X},\myuline{\spectra})^{\omega})$, functorial in $X$.
\end{prop}
\begin{proof}
    For this, consider the composition of natural transformations 
    \[\udl{\sphere}_G[-]\longrightarrow \udl{K}_G(\myuline{\spectra}_G)\otimes(-) \xlongrightarrow{\mathrm{assem.}} \udl{K}_G(\udl{\func}(\myuline{(-)},\myuline{\spectra})^{\omega}) \xlongrightarrow{\mathrm{trace}}  \udl{\sphere}_G[\mathcal{L}-] \xlongrightarrow{\eval} \udl{\sphere}_G[-]\]
    of $G$-functors $\udl{\spc}_G\rightarrow\myuline{\spectra}_G$, where the first map is given by the ring unit, the second map is the assembly map, i.e. the colimit interchange map (using that $\udl{\func}(\udl{X},\myuline{\spectra}_G)^{\omega}$ is the colimit in $\udl{\cat}\perfect_G$ of the constant diagram indexed over $\udl{X}$ with value $\myuline{\spectra}_G$), the third map is the Dennis trace together with the identification from \cref{prop:free_loop_space}, and the last map is evaluating the free loop at a point on $S^1$. All in all, this yields a natural transformation $\alpha\colon \udl{\sphere}_G[-]\rightarrow \udl{\sphere}_G[-]$ of $G$-functors $\udl{\spc}_G\rightarrow\myuline{\spectra}_G$ which preserve all $G$-colimits. Evaluating at $\ast\in \spc_H$ clearly gives the identity on $\sphere_H$ for all $H\leq G$, and so since $\udl{\spc}_G$ is generated under $G$-colimits by the point, we get that $\alpha$ is in a natural equivalence, as required.
\end{proof}

As  illustrations of the foregoing developments, let us record some immediate consequences of all the structures developed thus far.

\begin{example}\label{obs:factoring_trace_map}
    We consider the effect of the equivariant Dennis trace map when the input is the $G$-sphere spectrum $\sphere_G$. Since the trace map is a transformation of $G$-lax symmetric monoidal functors, we in particular have a commuting square
    \[
    \begin{tikzcd}
        \norm^G_eK(\sphere) \rar\dar & \norm^G_e\THH(\sphere) \simeq \sphere_G\dar["\simeq"]\\
        \udl{K}(\sphere_G) \rar & \udl{\THH}(\sphere_G)\simeq \sphere_G
    \end{tikzcd}
    \]
    in $\calg(\spectra_G)$, where the identification of the right hand terms with $\sphere_G$ is since $\udl{\THH}$ is a $G$-symmetric monoidal functor. Thus, by applying the symmetric monoidal functor $\Phi^G\colon \spectra_G\rightarrow \spectra$ which satisfies $\Phi^G\norm^G_e\simeq \id$, we obtain a factorization in $\calg(\spectra)$
    \[K(\sphere)\longrightarrow \Phi^G\udl{K}(\sphere_G)\longrightarrow \THH(\sphere)\simeq \sphere\]
    of the nonequivariant Dennis trace $K(\sphere)\rightarrow \THH(\sphere)\simeq \sphere$. 
\end{example}

\begin{example}\label{example:frobenius_lift_THH_borel}
    Let $\sC\in\calg(\cat\perfect)$ equipped with the trivial $G$-action. By considering the effect of the Dennis trace on $\udl{\borel}(\sC)\in\calg_G(\udl{\cat}\perfect_G)$ from \cref{example:algebra_objects_catperfGII} and using as above that $\Phi^G\norm^G_e\simeq \id$, we get the following diagram in $\calg(\spectra)$
    \[
    \begin{tikzcd}
         K(\sC) \rar\dar["\mathrm{tr}"] & \Phi^GK_G(\udl{\borel}(\sC)) \rar \dar["\mathrm{tr}"]& K(\sC)^{t_{\proper}G}\dar["\mathrm{tr}"]\\
         \THH(\sC) \rar & \THH(\stmodSmall_G^{\proper}(\sC)) \rar & \THH(\sC)^{t_{\proper}G}
    \end{tikzcd}
    \] where the horizontal composites are the proper Tate--Frobenius maps (which is the usual Tate--Frobenius map when $G=C_p$ from \cite{nikolausScholze}),  $\stmodSmall_G^{\proper}(\sC)$ is the proper stable module category of $\sC$ (cf. for instance \cite[Cons. 5.1.2]{PD1}. In the case of $G=C_p$, this reduces to the stable module category as defined in \cite[Def. 4.1]{krause2020picard}), and the middle bottom term is by \cref{prop:geometric_fixed_THH}. The middle vertical map should be viewed as a lift of the proper Tate construction of the Dennis trace map, and we speculate that this diagram might be useful in probing the object $\Phi^GK_G(\udl{\borel}(\sC))$, which is an interesting and hard problem related to Swan equivariant K-theory.
\end{example}

\begin{example}\label{example:frobenius_lift_THH_normed_ring}
    Similarly, by \cref{example:algebra_objects_catperfGI}, for $R\in\calg_G(\myuline{\spectra}_G)$ a normed ring $G$-spectrum, the various Dennis trace maps participate in the following diagram in $\calg(\spectra)$
    \[
    \begin{tikzcd}
         K(R^e) \rar\dar["\mathrm{tr}"] & \Phi^GK_G(R) \rar \dar["\mathrm{tr}"]& K(R^e)^{t_{\proper}G}\dar["\mathrm{tr}"]\\
         \THH(R^e) \rar & \THH(\Phi^GR) \rar & \THH(R^e)^{t_{\proper}G}.
    \end{tikzcd}
    \]
\end{example}

\subsection{Every \texorpdfstring{$G$}--spectrum is the K--theory of a  \texorpdfstring{$G$}--stable category}
\label{subsec:every_spectrum}

The following mimics the proof of \cite[Thm. 1.6]{MotLoc} word for word.

\begin{thm}\label{thm:every_spectrum_is_K-theory}
    The functor $\udl{K}\colon \udl{\cat}\perfect\rightarrow\myuline{\spectra}$ is essentially surjective. In particular, every genuine $G$-spectrum is the K--theory of a perfect $G$-stable category.
\end{thm}
\begin{proof}
    Essential surjectivity is a fiberwise statement, and so we may deal with the ordinary categories at level $G$ without loss of generality. Since $\myuline{\mapsp}(\unit,-)\colon {\Mot}\loc_G\rightarrow{\spectra}_G$ is in particular lax symmetric monoidal, it lifts to a functor $\myuline{\mapsp}(\unit,-)\colon {\Mot}\loc_G\rightarrow {\module}_{{\spectra}_G}(\udl{K}(\myuline{\spectra}))$ by the corepresentability of K--theory from \cref{prop:corep}. On the other hand, by the Dennis trace map from \cref{cor:dennis_trace} and the fact that $\udl{\THH}$ is symmetric monoidal from \cref{ex:TrG}, we obtain a map $\udl{K}(\myuline{\spectra})\rightarrow\udl{\THH}(\myuline{\spectra})=\sphere_G$ in $\calg(\spectra_G)$ which witnesses $\sphere_G$ as a commutative algebra retract of $\udl{K}(\myuline{\spectra})$. Hence, restricting along the Dennis trace gives a section to the forgetful functor ${\module}_{{\spectra}_G}(\udl{K}(\myuline{\spectra}))\rightarrow {\module}_{{\spectra}_G}(\sphere_G)={\spectra}_G$. Moreover, since \cref{cor:compactness_of_unit_in_motloc} guarantees that the tensor unit $\sU(\myuline{\spectra})\in{\Mot}\loc_G$ is compact, we see by enriched Schwede--Shipley (cf. \cite[Prop. 4.8.5.8, Thm. 4.8.5.11, Rmk. 4.8.5.12]{lurieHA} setting $\sC=\mathcal{M}={\Mot}\loc_G$ therein)    that there is a fully faithful inclusion ${\module}_{{\spectra}_G}(\udl{K}(\myuline{\spectra}))\subseteq {\Mot}\loc_G$ whose essential image is the $G$-presentable--stable subcategory generated by $\sU(\myuline{\spectra})$. In particular, it provides a section to the functor $\myuline{\mapsp}(\unit,-)\colon {\Mot}\loc_G\rightarrow {\module}_{{\spectra}_G}(\udl{K}(\myuline{\spectra}))$. We may now conclude by combining these with the fact that  $\cat\perfect_G\rightarrow\Mot\loc_G$ is essentially surjective by \cref{thm:GMotGPres}.
\end{proof}

\appendix
\renewcommand{\thesubsection}{\thesection}

\section{The fixed points--cofree adjunction}\label{appendix:fixCofreeAdjunction}

In this technical appendix, we work in the general setting of categories parametrized over an  atomic orbital category $\baseCat$, in the sense of \cite[Def. 4.1]{nardinExposeIV}, with a final object $T$, an example of which is $\orbit_G$. We write $\cat_{\baseCat}^{\times}$ for the nonfull subcategory of $\cat_{\baseCat}$ on those $\baseCat$--categories which admit finite indexed products and morphisms those that preserve these. For $\udl{\sC}\in\cat_{\baseCat}$, we write $\int\underline{\sC}$ for the cocartesian unstraightening of $\underline{\sC}$ under the equivalence $\cat_{\baseCat}\simeq \mathrm{coCart}(\baseCat\op)$.

\begin{lem}\label{lem:indexed_product_adjunction}
    Let $\underline{\sC}\in\cat_{\baseCat}^{\times}$. Then the fully faithful inclusion $i\colon \sC_T\hookrightarrow \int\underline{\sC}$ admits a right adjoint $R$ so that we get a Bousfield colocalization 
\[
\begin{tikzcd}
    i\colon \sC_T \rar[shift left = 1, hook] & \int\udl{\sC}\lar[shift left = 1] \cocolon R.
\end{tikzcd}
\]
\end{lem}
\begin{proof}
Since $T$ is final in $\baseCat$, we have a canonical natural transformation of functors on $\baseCat\op$ of the form $\udl{\constant}_{\baseCat\op}\sC_T\to \udl{\sC}$. By the existence of finite indexed products, this is levelwise a left adjoint, and so by (the dual of) \cite[Prop. 7.3.2.6]{lurieHA}, the induced functor of coCartesian fibrations $\sC_T\times\baseCat\op\to \int\udl{\sC}$ admits a (relative) right adjoint. Using again that $T$ is final in $\baseCat$, the inclusion $\sC_T\to \sC_T\times\baseCat\op$ at $T$ has a right adjoint, and the composition $\sC_T\to \sC_T\times\baseCat\op\to \int\udl{\sC}$ is clearly $i$, so we get the claim by composition of adjoints. 
\end{proof}
\begin{obs}\label{fact:counit_of_cofree_adjunction}
    Let $\D\in\cat$. Recall from \cite[Thm. 2.8]{nardinExposeIV} that there is a 2-adjunction $\int\colon \cat_{\baseCat} \rightleftharpoons \cat \cocolon \udl{\cofree}$ where $\int$ is the cocartesian unstraightening. The adjunction counit $\epsilon\colon \int\udl{\cofree}(\D)\rightarrow \D$  may be described as follows: precomposing $\epsilon$ with the inclusion of the fiber $\func((\baseCat_{/V})\op,\D)$ of $\udl{\cofree}(\D)$ over $V$, the map $\func((\baseCat_{/V})\op,\D)\rightarrow\int\udl{\cofree}(\D)\xrightarrow{\epsilon} \D$   is given by evaluating at the final object $(V=V)\in \baseCat_{/V}$. 

    To see this, without loss of generality by restricting if necessary, let $V\colon\ast\rightarrow\baseCat$ be the final object. Then the datum of an object in the fiber over $V$ of $\int\udl{\cofree}(\D)$ is equivalent to a map of $\baseCat$--functors $x\colon \udl{\ast}\rightarrow \udl{\cofree}(\D)$. We now contemplate the commuting square
    \[
    \begin{tikzcd}
        \func_{\baseCat}(\udl{\cofree}(\D),\udl{\cofree}(\D)) \rar["\simeq"] \dar["x^*"']&  \func(\int\udl{\cofree}(\D),\D) \dar["x^*"]\\
        \func_{\baseCat}(\udl{\ast},\udl{\cofree}(\D)) \rar["\simeq"] & \func(\int\udl{\ast},\D)=\func(\baseCat\op,\D) \rar["V^*"] & \D
    \end{tikzcd}
    \]
    coming from the universal property of the cofree construction. Note here that the bottom equivalence \textit{is} the identification of the $V$-fiber of $\udl{\cofree}(\D)$ as $\func(\baseCat\op,\D)$. Tracing where the identity in the top left corner goes yields the claim.

    Conversely, the adjunction unit $\eta\colon \udl{\sC}\rightarrow\udl{\cofree}(\int\udl{\sC})$ for $\udl{\sC}\in\cat_{\baseCat}$ may be described concretely as follows: without loss of generality by restricting if necessary, for $V\in \baseCat$ the final object, the functor $\eta_V\colon \sC_V\rightarrow\func(\baseCat\op,\int \udl{\sC})$ is given as follows: for $x\in\sC_V$, $\eta_V$ sends it to the $\baseCat\op$ diagram in $\int\udl{\sC}$ given by sending the unique morphism $w\colon W\rightarrow V$ in $\baseCat$ to the cocartesian pushforward $x\rightarrow w^*x$. To see this, consider as above the commuting diagram
    \[
    \begin{tikzcd}
        \func_{\baseCat}(\udl{\sC},\udl{\cofree}(\int\udl{\sC})) \dar["x^*"]\rar["\simeq"]& \func(\int\udl{\sC},\int\udl{\sC})\dar["x^*"]\\
        \func_{\baseCat}(\udl{\ast},\udl{\cofree}(\int\udl{\sC}))\rar["\simeq"] & \func(\baseCat\op,\int\udl{\sC}).
    \end{tikzcd}
    \] Tracing where the identity map in the top right corner goes yields the claim because the right vertical $x^*$ map is precisely implemented by restricting along $x\colon \baseCat\op\rightarrow \int\udl{\sC}$ which is a map of cocartesian fibrations over $\baseCat\op$.
\end{obs}  

\begin{lem}\label{lem:characterisation_of_indexed_product_preserving}
    Let $\udl{\sC}\in\cat^{\times}_{\baseCat}$ and $\D\in\cat^{\times}$. Then:
    \begin{enumerate}
        \item The $\baseCat$--category $\udl{\cofree}(\D)$ has finite indexed products. 

        \item Under the equivalence $\func(\int\udl{\sC},\D)\simeq \func_{\baseCat}(\udl{\sC},\udl{\cofree}(\D))$, a functor $F\colon \udl{\sC}\rightarrow \udl{\cofree}(\D)$ preserves finite indexed products if and only if the associated functor $\overline{F}\colon \int\udl{\sC}\rightarrow \D$ preserves finite products in each fiber and satisfies the property that for every $w\colon W \rightarrow V$ in $\baseCat$ and every $x\in \int\udl{\sC}\times_{\baseCat\op}\{W\}$, the cartesian lift $w_*x\rightarrow x$ in $\int\udl{\sC}$ is sent to an equivalence by $\overline{F}$.
    \end{enumerate}
\end{lem}
\begin{proof}
    To see (1), for $w\colon W \rightarrow V$ in $\baseCat$, orbitality of $\baseCat$ ensures that the right Kan extension $w_* \colon \func((\baseCat_{/W})\op,\D)\rightarrow \func((\baseCat_{/V})\op,\D)$ exists because the appropriate limits that need to be taken are all finite products. Moreover, for $u\colon U \rightarrow V$ and writing $\sqcup_iZ_i$ for the orbital decomposition of $U\times_VW$, the commuting square of restriction functors obtained by applying $\func((-)\op,\D)$ to 
    \[
    \begin{tikzcd}
        \coprod_i\baseCat_{/Z_i}=\baseCat_{/\sqcup_iZ_i} \rar\dar\ar[dr, phantom, "\lrcorner"] & \baseCat_{/U}\dar["u"]\\
        \baseCat_{/W} \rar["w"] & \baseCat_{/V}
    \end{tikzcd}
    \]
    is horizontally right adjointable by smooth/proper basechange because the map $w$ is a cartesian fibration. Therefore, $\udl{\cofree}(\D)$ indeed admits finite indexed products.

    For part (2), note first that a $\baseCat$--category $\udl{\E}$ having finite indexed products is precisely the property that each fiber category has finite products, the cocartesian morphisms of the cocartesian fibration $p\colon \int\udl{\E}\rightarrow \baseCat\op$ preserve these finite products,  that $p$ is in fact also a cartesian fibration, and that these cartesian morphisms satisfy appropriate Beck--Chevalley basechanges. From this perspective, a $\baseCat$--functor $F\colon \udl{\sC}\rightarrow \udl{\cofree}(\D)$ which preserves finite indexed products is thus equivalently  a functor $F\colon \int\udl{\sC}\rightarrow \int\udl{\cofree}(\D)$ over $\baseCat\op$ which is a map of cocartesian and cartesian fibrations and which fiberwise preserves finite products.

    Now,  the functor $\overline{F}$ is given by the composite $\int\udl{\sC}\xrightarrow{\int F} \int\udl{\cofree}(\D) \xrightarrow{\epsilon} \D$. By \cref{fact:counit_of_cofree_adjunction}, for $W\in\baseCat$, the restriction of $\overline{F}$  to the fiber over $W$ is given by $\sC_W \xrightarrow{F_W} \func((\baseCat_{/W})\op,\D)\xrightarrow{\eval_{W}}\D$. In particular, for a morphism $f\colon W \rightarrow V$ in $\baseCat$, the cartesian morphism $w_*x\rightarrow x$ in $\int\udl{\sC}$ gets sent to $\eval_VF_Vw_*x\rightarrow \eval_WF_Wx$. But since for any $d\in \func((\baseCat_{/W})\op,\D)$, $\eval_Vw_*d\simeq \eval_Wd$, we get that $\eval_WF_Wx\simeq \eval_Vw_*F_Wx$. Under this identification, it is straightforward to check that the map $\eval_VF_Vw_*x\rightarrow \eval_WF_Wx$ is obtained by applying $\eval_V$ to the canonical map 
    \begin{equation}\label{eqn:canonical_map_to_cartesian_lift}
        F_Vw_*x\rightarrow w_*F_Wx
    \end{equation} coming from the universal property of the cartesian morphism $w_*F_Wx\rightarrow F_Wx$ in $\int\udl{\cofree}(\D)$. On the other hand, for any other morphism $u\colon U \rightarrow V$ in $\baseCat$, recalling the notation from \cref{setting:pullback_notations}, we have by the Beck--Chevalley basechange and that $F$ was a map of cocartesian fibrations preserving fiberwise products that applying $u^*$ to \cref{eqn:canonical_map_to_cartesian_lift} gives the map
    \begin{equation}\label{eqn:broken_up_canonical_map_to_cartesian_lift}F_U(\prod_i\overline{w}_{i*}\overline{u}_i^*x)\longrightarrow \prod_i\overline{w}_{i*}F_{Z_i}\overline{u}_i^*x.\end{equation} Applying $\eval_{U}$ to this map and using the hypothesis that $\overline{F}\colon \int\udl{\sC}\rightarrow \D$ preserves products in each fiber gives
    \[\prod_i\overline{F}\overline{w}_{i*}\overline{u}_i^*x\simeq \prod_i\eval_U{F}_U\overline{w}_{i*}\overline{u}_i^*x\longrightarrow \prod_i\eval_{Z_i}F_{Z_i}\overline{u}_i^*x\simeq \prod_i\overline{F}\overline{u}_i^*x.\]

    All in all, since the evaluation maps $\func((\baseCat_{/V})\op,\D)\rightarrow \D$ are jointly conservative, we thus see that \cref{eqn:canonical_map_to_cartesian_lift} is an equivalence for all $w\colon W \rightarrow V$ in $\baseCat$ and all $x\in \sC_W$ if and only if $\overline{F}u_*y \rightarrow\overline{F}y$ is an equivalence for all $u\colon U\rightarrow V$ in $\baseCat$ and all $y\in\sC_U$. This completes the proof.
\end{proof}
	
\begin{prop}\label{prop:fix_points_cofree_adjunction}
    Let $\baseCat$ be an atomic orbital category with final object $T$. Then there is an adjunction 
    $\eval_T\colon \cat^{\times}_{\baseCat}\rightleftharpoons \cat^{\times} \cocolon \udl{\cofree}$.

    For $\udl{\sC}\in\cat^{\times}_{\baseCat}$, the unit $\udl{\sC}\rightarrow \udl{\cofree}(\sC_T)$ of the adjunction is given at level $V$ by $\sC_V\xrightarrow{\eta}\func((\baseCat_{/V})\op,\sC_V) \xrightarrow{v_*} \func((\baseCat_{/V})\op,\sC_T)$ where $\eta$ is the adjunction unit for $\int\dashv \udl{\cofree}$ from \cref{fact:counit_of_cofree_adjunction}, and $v\colon V\rightarrow T$ is the unique map. Explicitly,    this takes $x\in\sC_V$ to the diagram $(\baseCat_{/V})\op\rightarrow \sC_T$ given by sending $w\colon W \rightarrow V$ to $v_*w_*w^*x\in\sC_T$.

    For $\D\in\cat^{\times}$, the counit $\func(\baseCat\op,\D)\rightarrow \D$ of the adjunction is given by evaluating at the final object $T$.
\end{prop}   
\begin{proof}
    For $\udl{\sC}\in\cat^{\times}_{\baseCat}$, we write $i\colon \sC_T \hookrightarrow \int\udl{\sC}$ for the inclusion of the fiber over $T$, which is fully faithful since $T$ was the final object in $\baseCat$. Therefore, by the universal property of $\udl{\cofree}$ and the Bousfield colocalization from \cref{lem:indexed_product_adjunction}, we obtain Bousfield colocalization
    \[
    \begin{tikzcd}
        \func(\sC_T,\D) \rar["R^*", hook, shift left = 1] &\func(\int\udl{\sC},\D)\simeq \func_{\baseCat}(\udl{\sC},\udl{\cofree}(\D)). \lar[shift left = 1, "i^*"]
    \end{tikzcd}
    \]
    Now, by the characterisation from \cref{lem:characterisation_of_indexed_product_preserving} (2), $i^*$ restricts to a functor $\func^{\udl{\times}}(\int\udl{\sC},\D)\rightarrow \func^{\times}(\sC_T,\D)$. Conversely, suppose that $\overline{F}_T\colon \sC_T\rightarrow \D$ is a finite product preserving functor. We claim that $R^*\overline{F}_T=\overline{F}_TR\colon \int\udl{\sC}\rightarrow \D$ corresponds to a finite indexed product preserving functor $\udl{\sC}\rightarrow \udl{\cofree}(\D)$. By the characterisation from \cref{lem:characterisation_of_indexed_product_preserving} (2), it thus suffices to show that $\overline{F}_TR$ preserves finite products in each fiber and that for each $w\colon W \rightarrow V$ in $\baseCat$ and each $x\in\sC_W$, the map $\overline{F}_TRw_*x\rightarrow \overline{F}_TRx$ is an equivalence. But then, writing $v\colon V\rightarrow T$ and $f\colon W \rightarrow T$ for the unique maps to the final object $T$, we thus have that $f\simeq vw$, and by construction of $R$ from \cref{lem:indexed_product_adjunction}, $Rw_*x\simeq v_*w_*x\simeq f_*x$ and $Rx\simeq f_*x$. It is then straightforward to check that under these identifications, the map of interest is identified with the equivalence $\overline{F}_Tv_*w_*x\xrightarrow{\simeq} \overline{F}_Tf_*x$. This proves the claim, so that $R^*$ restricts to a functor $\func^{\times}(\sC_T,\D)\hookrightarrow \func^{\udl{\times}}(\int\udl{\sC},\D)$. 

    Therefore, we also have a Bousfield colocalization 
    \begin{equation}\label{eqn:iR*_adjunction}
    \begin{tikzcd}
        \func^{\times}(\sC_T,\D) \rar["R^*", hook, shift left = 1] &\func^{\udl{\times}}(\int\udl{\sC},\D). \lar[shift left = 1, "i^*"]
    \end{tikzcd}
    \end{equation} To show that this is an equivalence, it suffices now to argue that for $F\colon \int\udl{\sC}\rightarrow\D$ a finite indexed product preserving functor, $R^*i^*F\simeq FiR\rightarrow F$ is an equivalence. Again, by the characterisation \cref{lem:characterisation_of_indexed_product_preserving} (2) of finite indexed product preserving functors, this is true by hypothesis. By passing this equivalence to groupoid cores, we have thus shown that there is an equivalence
    $\map_{\times}(\sC_T,\D) \simeq\map_{\udl{\times}}(\udl{\sC},\udl{\cofree}(\D))$ which is clearly natural in $\udl{\sC}$ and $\D$. This witnesses the claimed adjunction in the statement.

    Finally, to see the descriptions for the adjunction (co)units, simply note that the equivalence \cref{eqn:iR*_adjunction} dictates that the adjunction unit and counit are given by 
    \[\udl{\sC}\xlongrightarrow{\eta} \udl{\cofree}(\int\udl{\sC}) \xlongrightarrow{R} \udl{\cofree}(\sC_T)\quad\quad\quad \quad \func(\baseCat\op,\D) \xhookrightarrow{i} \int\udl{\cofree}(\D) \xlongrightarrow{\epsilon} \D \] which, together with \cref{fact:counit_of_cofree_adjunction}, yields the required descriptions.
\end{proof}

We can now provide the proof of \cref{lem:fixmackadj}, as promised. 

\begin{proof}[Proof of \cref{lem:fixmackadj}.]
    The claimed adjunction is obtained by considering 
    \[
    \begin{tikzcd}
        \cat^{\oplus}_G \rar[shift left = 1, hook, "\inclusion"]& \cat^{\times}_G \rar[shift left = 1, "(-)^G"] \lar[shift left = 1, "\udl{\cmonoid}"] & \cat^{\times} \lar[shift left = 1, "\udl{\cofree}"]
    \end{tikzcd}
    \]
    where the first is by \cite[Prop. 5.11]{nardinExposeIV} (see also \cite[Thm. 7.4]{CLLSpans}) and the second by \cref{prop:fix_points_cofree_adjunction}. Now observe that this factors to an adjunction $(-)^G\colon \cat^{\oplus}_G\rightleftharpoons \cat^{\oplus} \cocolon \udl{\mackey}_G(-)$ via the full subcategory $\cat^{\oplus}\subseteq \cat^{\times}$ since for any $G$-semiadditive $\udl{\sC}$, $\udl{\sC}^G=\sC_G$ is semiadditive in the ordinary sense. This also uses the fact that $\udl{\cmonoid}\circ \udl{\cofree}\simeq \udl{\mackey}$ by \cite[6.6]{nardinExposeIV} or \cite[Thm. D]{philMasters}. The descriptions of the unit and counit are now clear by \cref{prop:fix_points_cofree_adjunction}.
\end{proof}

\section{\texorpdfstring{$G$}--presentable--stable categories as modules over \texorpdfstring{$G$}--spectra}
\label{appendix:modules_SpG}

We start with a result of independent interest, which will help us deduce \Cref{thm:GMotGPres} from \cite{MotLoc} instead of reproving it. 
\begin{thm}\label{thm:equivmod}
    The $G$-fixed points functor $\presentable^L_{G,\mathrm{st}}\rightarrow \presentable^L_{\mathrm{st}}$ induces an equivalence of commutative $\presentable^L_{\mathrm{st}}$-algebras \[\presentable^L_{G,\mathrm{st}}\simeq \module_{\spectra_G}(\presentable^L_{\mathrm{st}}). \]

    Furthermore, a $G$-presentable $G$-stable $G$-category is compactly generated if and only if its image in $\presentable^L_{\mathrm{st}}$ is, and the above equivalence restricts to a symmetric monoidal equivalence \[\cat\perfect_G\simeq \module_{\spectra_G^\omega}(\cat\perfect).\]
\end{thm}
We will use the following lemma: 
\begin{lem}\label{lem:Gcons}
    The $G$-fixed points functor $\cat^{\oplus}_G\to \cat^{\oplus}$ from $G$-semiadditive $G$-categories to semiadditive $G$-categories is conservative on idempotent-complete categories. 
\end{lem}
\begin{proof}
Suppose $\udl{\A}\to \udl{\B}$ is a $G$-semiadditive functor of $G$-semiadditive $G$-categories such that $\A_G\to \B_G$ is an equivalence. In this case, for any subgroup $H\leq G$, by $G$-semiadditivity, the following diagram is vertically left adjointable: 
\[\begin{tikzcd}
	{\A_G} & {\B_G} \\
	{\A_H} & {\B_H}.
	\arrow["{f_G}", from=1-1, to=1-2]
	\arrow["{\res_H^G}"', from=1-1, to=2-1]
	\arrow["{\res_H^G}", from=1-2, to=2-2]
	\arrow["{f_H}"', from=2-1, to=2-2]
\end{tikzcd}\]
If $f_G$ is fully faithful, it follows that $f_H$ is fully faithful on the image of $\res_H^G$, and therefore on retracts of objects in the image. But by the double coset formula, $\id_{\A_H}$ is a retract of $\res_H^G\ind_H^G$ so that $f_H$ is fully faithful. 

Second, by the commutativity of the above diagram, if $f_G$ is essentially surjective, then $f_H$ is essentially surjective on the essential image of $\res_H^G$. Thus, if $\A_H$ is idempotent-complete and $f_H$ is fully faithful, the above retract argument also shows that $f_H$ is essentially surjective. This proves that on idempotent-complete $G$-semiadditive $G$-categories, $G$-fixed points is conservative. 
\end{proof}
\begin{proof}[Proof of \Cref{thm:equivmod}]
Consider the basechange functor $\presentable^L_{\mathrm{st}}\to \presentable^L_{G,\mathrm{st}}, \A\mapsto \A\otimes\myuline{\spectra}_G$. This is a strong symmetric monoidal functor, left adjoint to the fixed points functor: combine \cite[Prop. 2.6.3.6, Prop. 2.6.3.7]{MartiniWolf2022Presentable} with the fact that $\presentable^L_{G,\mathrm{st}}\simeq \module_{\myuline{\spectra}_G}(\presentable_G)$ by \cite[Prop. 2.2.19]{kaifNoncommMotives}. By the linear Barr--Beck--Lurie theorem \cite[Prop. 4.8.5.8]{lurieHA}, to prove the equivalence, it therefore suffices to prove that the fixed points functor preserves colimits, tensors with objects of $\presentable^L_{\mathrm{st}}$, and is conservative. The first two properties are in fact much more general, see \cite[Prop. 2.4.4.11, Prop. 2.6.3.13]{MartiniWolf2022Presentable}, and conservativity follows from \Cref{lem:Gcons}.

Now, by definition, an object in $\udl{\A}$ is compact if and only if it is so fiberwise, that is, in $\A_H$ for all $H\leq G$. Since the restriction functors $\res_H^G$ have colimit-preserving right adjoints (the right adjoint is coinduction, which by definition of $G$-semiadditive is equivalent to induction), they preserve compact objects, so an object $x\in \A_G$ is compact in $\udl{\A}$ if and only if it is in $\A_G$.   

This proves that the equivariant notion of compact is the same as underlying notion in this case, so that compact generation and compact preservation also agree in the equivariant and underlying senses. Thus the equivalence restricts to an equivalence between categories of compactly generated categories and compact-preserving functors on either side, which yields the desired result.
\end{proof}

\begin{rmk}
    The proofs above also clearly works word-for-word when we replace $\orbit_G$ with an arbitrary atomic orbital category, in the sense of \cite[Def. 4.1]{nardinExposeIV}, with a final object.
\end{rmk}

\begin{cor}\label{cor:Gmotivic}
    Under the equivalence $\cat\perfect_G\simeq \module_{\spectra_G^\omega}(\cat\perfect)$ from \Cref{thm:equivmod}, the notion of (finitary) Mackey $G$-localizing invariant and the notion of (finitary) $\spectra_G^\omega$-localizing invariant from \cite[App. B]{MotLoc} are equivalent, and thus the notion of motivic equivalences coincide.    In particular, the localization of $\cat\perfect_G$ at motivic equivalences is presentable stable. 
\end{cor}
\begin{proof}
    On either side, Karoubi sequences are defined as fiber/cofiber sequences, i.e. squares \[\begin{tikzcd}
	\A & \B \\
	0 & \sC
	\arrow[from=1-1, to=1-2]
	\arrow[from=1-1, to=2-1]
	\arrow[from=1-2, to=2-2]
	\arrow[from=2-1, to=2-2]
\end{tikzcd}\] that are both cartesian and cocartesian. Since this and filtered colimits are equivalence-invariants, the notions of finitary localizing invariants clearly agree. 
\end{proof}

\begin{rmk}
    The same proof works \textit{verbatim} in the $G$-semiadditive context. In the corollary above, one simply needs to replace $\spectra_G^\omega$ with $\spancategory(\finite_G)$ and all instances of ``$\perfect$'' with $\oplus$'s. But we will not be needing this though.
\end{rmk}

\section{Finite posets with action  and some consequences}\label{appendix:posets}

\begin{prop}\label{prop:finite_posets}
    Let $G$ be a finite group and $\udl{I}\in\func(BG,\poset)$ be a finite poset with $G$-action. Then $\udl{I}\in\cat_G$  is a finite $G$-category.
\end{prop}
\begin{proof}
    We will prove by a double induction on the order of $G$ as well as the number of points in the finite poset. The case of $G=e$ is standard, so suppose for induction that the statement holds for all groups strictly smaller than $|G|$. We write $I$ for the underlying poset, and without loss of generality, it is a connected poset. Suppose we have also shown the statement  for a $G$-invariant sieve ${J}\subseteq {I}$. 
    
    Now let $x\in I\backslash J$ be a minimal element which is not in $J$ and write $H\leq G$ for its stabilizer subgroup. We need to show that the full subposet $\overline{J}\coloneqq J\cup G\times_H\{x\}\in\func(BG,\poset)$ of $I$ also gives a finite $G$-category $\udl{\overline{J}}$. We write  $J_{/x}\coloneqq I_{/x}\cap J$ for the not necessarily $G$-invariant subposet of $J$.
    
    If $H=G$, then  we have  $\udl{\overline{J}}=\udl{I}_{/x}\cup_{\udl{J}_{/x}}\udl{J}$. Now $\udl{I}_{/x}=\udl{J}\cone_{/x}$ is given by  $\udl{J}_{/x}\times \Delta^1/\udl{J}_{/x}\times\{1\}$. This is a finite $G$-category because by induction on the number of points in a finite poset, $\udl{J}_{/x}$ was assumed to be a finite $G$-category already. Hence, $\udl{\overline{J}}$ is a finite $G$-category in this case.

    Next, suppose $H\lneq G$. We may consider the subposet $J_{/x}\subseteq J$ as a poset with a $H$--action so that by induction on the order of the group, this induces a finite $H$--category $\udl{J}_{/x}$ which in turn gives rise to the finite $H$--category $\udl{I}_{/x}=\udl{J}_{/x}\cone$ by the same argument as in the previous paragraph. 
    Then by adjunction, the inclusion $J_{/x} \subseteq {J}$ of elements in $J$ under $x$ induces a map of $G$-categories $G\times_H \udl{J}_{/x}\rightarrow {\udl{{J}}}$. Now we define $\udl{P}$ to be the following pushout in $\cat_G$ and consider the map $\varphi:P \to \udl{\overline{J}}$ induced by the following commutative diagram:
    \[
    \begin{tikzcd}
        G\times_H \udl{J}_{/x}\rar \dar[hook]\ar[dr,phantom, "\ulcorner"]&  {\udl{{J}}}\dar[hook]\ar[ddr, hook, bend left = 30 ]\\
        G\times_H \udl{I}_{/x} 
        \rar\ar[drr, bend right = 20]& \udl{P}\ar[dr, dashed, "\varphi"]\\
        & & \udl{\overline{J}}.
    \end{tikzcd}
    \]
    where the left vertical map is an inclusion of a sieve. In particular,  $\udl{P}$ is a finite $G$-category. 
    
    It thus suffices now to show that the map $\udl{P}\rightarrow
     \udl{\overline{J}}$ is an equivalence.    
    Since $G\times_H \udl{I}_{/x} \sqcup \udl{J}\rightarrow \udl{\overline{J}}$ is essentially surjective, so is $\udl{P}\rightarrow\udl{\overline{J}}$. On the other hand, it is an elementary exercise that there can be no map between elements in the same orbit, and so the composites $G\times_H\{x\}\subseteq G\times_H\udl{I}_{/x}\rightarrow \udl{P}$ and $G\times_H\{x\}\subseteq G\times_H\udl{I}_{/x}\rightarrow \udl{\overline{J}}$ are fully faithful (where the former case uses the mapping space formula in \cite[Cor. 5.2]{maximePushout}). All in all, $\varphi\colon \udl{P}\rightarrow \udl{\overline{J}}$ is fully faithful when restricted to the full subcategories $G\times_H\{x\}$ and $\udl{J}$. Hence, we are left with showing that for $b\in G\times_H\{x\}$ and $c\in \udl{J}$, $\varphi$ induces equivalences
    \[\myuline{\map}_{P}(b,c)\rightarrow \myuline{\map}_{\overline{J}}(\varphi b,\varphi c)=\emptyset\quad\quad \myuline{\map}_{P}(c,b)\rightarrow \myuline{\map}_{\overline{J}}(\varphi c,\varphi b).\] As any map to the empty space, the first map is indeed an equivalence.
    
    Furthermore, unwinding the mapping space formula \cite[Thm. 0.1 (2)]{maximePushout} for the fixed points of our diagram, we find that if $K\leq G$ is a subgroup, and $c\in J^K$, $b=(g,x)\in (G\times_H\{x\})^K$, then the homotopy type of $\myuline{\map}_{P}(c,b)^K$ is the realization of the following poset: $$\{(\sigma,j)\in (G\times_HJ_{/x})^K \mid  \sigma j\geq c,\: \: g^{-1}\sigma \in H,\:\: g^{-1}\sigma j \leq x \} $$
    If $c\not\leq gx$, then this poset is empty, just as $\myuline{\map}_{\overline{J}}(\varphi c,\varphi b)^K$.     On the other hand, if $c\leq gx$, then $(g,g^{-1}c)$ is easily seen to be an initial element of this poset, thus making it contractible, just as $\myuline{\map}_{\overline{J}}(\varphi c,\varphi b)^K$. Thus, the second map is also an equivalence, which proves fully faithfulness and completes the proof. 
\end{proof}

As a consequence of the proof above, we obtain the following:

\begin{lem}\label{lem:omega_compactness_of_generators}
    Let $\udl{P}\in\func(BG,\poset)$ be a finite poset, $\udl{\sC}\in\udl{\cat}\perfect$, and $\udl{\E}\in\udl{\presentable}^L_{\omega,\mathrm{st}}$. Then $\ind\udl{\func}(\udl{P},\udl{\sC})\simeq \udl{\func}(\udl{P},\ind\udl{\sC})$ and $\udl{\func}(\udl{P},\udl{\E})^{\omega}\simeq \udl{\func}(\udl{P},\udl{\E}^{\omega})$. Moreover, if $\udl{\sC}$ is $\omega$--compact, then so is $\udl{\func}(\udl{P},\udl{\sC})$.
\end{lem}
\begin{proof}
    We prove the first statement about commutation of Ind and compact objects by induction on the size of the poset, with the base of the empty poset being trivial. By the proof of the previous result, we may write $\udl{P}$ as a pushout in $\func(BG,\poset)$
    \[
    \begin{tikzcd}
        \udl{M}\rar["f"]\dar[hook, "i"']\ar[dr, phantom, "\lrcorner"] &\udl{J}\dar[hook, "j"]\\
        \udl{I} \rar["h"] & \udl{P}
    \end{tikzcd}
    \]
    where the left vertical map is a sieve inclusion, and by a double induction as in the previous proof, we may assume the desired conclusions on $M, I,$ and $J$. We thus have a pullback square
    \begin{equation}\label{eqn:pullback_of_functor_cats_from_posets}
    \begin{tikzcd}
        \udl{\func}(\udl{P},\udl{\A}) \rar["j^*"]\dar["h^*"']\ar[dr, phantom , "\lrcorner"] & \udl{\func}(\udl{J},\udl{\A})\dar["f^*"]\\
        \udl{\func}(\udl{I},\udl{\A})\rar["i^*"] & \udl{\func}(\udl{M},\udl{\A})
    \end{tikzcd}
    \end{equation} either in $\udl{\cat}\perfect$ or in $\udl{\presentable}^L_{\mathrm{st}}$, depending on whether $\udl{\A}$ is in the former or the latter category. 

    We claim that in the case when $\udl{\A}\in\udl{\presentable}^L_{\omega,\mathrm{st}}$, this is even a pullback  $\udl{\presentable}^L_{\omega,\mathrm{st}}$. In other words, we will show that $\udl{\func}(\udl{X},\udl{\A})$ is $\omega$--compactly generated for $\udl{X}\in\{\udl{P},\udl{I},\udl{J},\udl{M}\}$, that all the maps in \cref{eqn:pullback_of_functor_cats_from_posets} preserve $\omega$--compact objects, and that $\udl{\func}(\udl{P},\udl{\A})$ satisfies the appropriate universal property in $\udl{\presentable}^L_{\omega,\mathrm{st}}$. By induction, $\udl{\func}(\udl{J},\udl{\A})$ and $\udl{\func}(\udl{I},\udl{\A})$ are $\omega$--compactly generated by the sets $\udl{S}$ and $\udl{T}$, say. Then $j_!\udl{S}\cup h_!\udl{T}\subseteq \udl{\func}(\udl{P},\udl{\sC})$ forms a set of $\omega$--compact generators. This is because they form a jointly conservative set since $\myuline{\map}(j_!x,z)\simeq \myuline{\map}(x,j^*z)$ and $\myuline{\map}(h_!y,z)\simeq \myuline{\map}(y,h^*z)$ and $j^*\times h^*\colon \udl{\func}(\udl{P},\udl{\A})\rightarrow \udl{\func}(\udl{J},\udl{\A})\times \udl{\func}(\udl{I},\udl{\A})$ is conservative, and so we may conclude by \cite[Prop. 2.2.14]{kaifNoncommMotives}. Next, to see that all the maps in \cref{eqn:pullback_of_functor_cats_from_posets} preserve $\omega$--compact objects, we just argue for the case $j^*$ since the others are  similar. For this, the right adjoint is computed as the  right Kan extension $j_*$. By the pointwise formula, for $p\in\udl{P}$, the relevant limit to compute is over the diagram $\udl{J}_{/p}\coloneqq \udl{J}\times_{\udl{P}}\udl{P}_{/p}\in\cat_G$.  But this is induced by the object $J\times_{P}P_{/p}\in\func(BG,\poset)$, whose underlying poset is finite. Hence, by \cref{prop:finite_posets}, this is a finite $G$-category,  so the functor $j_*$ preserves $G$-colimits and so $j^*$ preserves $\omega$--compact objects. Finally, since filtered colimits commute with pullbacks, we conclude that an object in $\udl{\func}(\udl{P},\udl{\A})$ is $\omega$--compact if and only if it is so as objects in $\udl{\func}(\udl{J},\udl{\A})$ and $\udl{\func}(\udl{I},\udl{\A})$ under the functors $j^*$ and $h^*$, respectively. From this characterisation, it is clear that the pullback \cref{eqn:pullback_of_functor_cats_from_posets}, which is a priori computed in $\udl{\presentable}^L_{\mathrm{st}}$, is in fact a pullback in $\udl{\presentable}^L_{\omega,\mathrm{st}}$.

    With this claim established, we now deduce the  lemma. By the claim, the equivalence $\ind\colon \udl{\cat}\perfect\simeq \udl{\presentable}^L_{\omega,\mathrm{st}}\cocolon (-)^{\omega}$, and the induction hypotheses, we get pullback squares
    \[
    \begin{tikzcd}
        \ind\udl{\func}(\udl{P},\udl{\sC}) \rar\dar\ar[dr, phantom , "\lrcorner"] & \udl{\func}(\udl{J},\ind\udl{\sC})\dar\\
        \udl{\func}(\udl{I},\ind\udl{\sC})\rar & \udl{\func}(\udl{M},\ind\udl{\sC})
    \end{tikzcd}
    \hspace{6mm}
    \begin{tikzcd}
        \udl{\func}(\udl{P},\udl{\E})^{\omega} \rar\dar\ar[dr, phantom , "\lrcorner"] & \udl{\func}(\udl{J},\udl{\E}^{\omega})\dar\\
        \udl{\func}(\udl{I},\udl{\E}^{\omega})\rar & \udl{\func}(\udl{M},\udl{\E}^{\omega})
    \end{tikzcd}
    \] in $\udl{\presentable}^L_{\omega,\mathrm{st}}$ and $\udl{\cat}\perfect$ respectively. Hence, we get that $\ind\udl{\func}(\udl{P},\udl{\sC})\simeq \udl{\func}(\udl{P},\ind\udl{\sC})$ and $\udl{\func}(\udl{P},\udl{\E})^{\omega}\simeq \udl{\func}(\udl{P},\udl{\E}^{\omega})$, as required.\vspace{1mm}

    Finally, suppose that $\udl{\sC}\in (\udl{\cat}\perfect)^{\omega}$. We will again show that $\udl{\func}(\udl{P},\udl{\sC})$ is also  $\omega$--compact by induction on the size of $P$ again. So suppose we already know the statement for smaller posets, and we write $\udl{P}$ as a pushout as above. By \cite[Prop. 2.5.13]{kaifNoncommMotives}, we get $\udl{\func}(\udl{P},\ind\udl{\sC})\simeq \udl{\func}(\udl{I},\ind\udl{\sC})\cup_{\udl{\func}(\udl{M},\ind\udl{\sC})}\udl{\func}(\udl{J},\ind\udl{\sC})$. Since the right hand side is a pushout of $\omega$--compact objects by the inductive hypothesis, we thus see that $\udl{\func}(\udl{P},\ind\udl{\sC})$ is also $\omega$--compact.
\end{proof}

\begin{cor}
    Let $\lambda \leq \kappa$ be regular cardinals and let $\udl{\E}$ be any $G$-presentably $G$-symmetric monoidal category, i.e. an object in $\calg_G(\udl{\presentable}^L)$, which is $\lambda$--compactly generated by a set $\udl{S}$. Suppose that the set $\udl{S}$ is closed under the multiplicative norms.  Then $\udl{\E}^{\kappa}\subseteq \udl{\E}$ is closed under the multiplicative norms. 
\end{cor}
\begin{proof}
First, we note that without loss of generality, we may assume $\lambda =\kappa$.  We prove the result by induction on the order of the group $G$, the base case of $G=e$ being standard since $S$ is closed under tensor products. So suppose we know the statement to be true for all groups of order strictly smaller than $|G|$. By using again the standard consequence that ordinary tensor products preserve $\kappa$--compactness, we are thus reduced to showing that $f_{\otimes}\colon \E_H^{\kappa}\rightarrow \E_G$ lands in $\E_G^{\kappa}$ for the map $f\colon G/H\rightarrow  G/G$, for any $H\leq G$.
 
Now $\udl{\E}^\kappa$ is generated (up to retracts if $\kappa=\omega$) under $\kappa$-small indexed coproducts and pushouts by $\udl{S}$, and we note that $\kappa$-small coproducts are $\kappa$-small, filtered colimits of finite indexed coproducts. Thus, we are reduced to showing that the collection of objects of $\res_H^G\udl{\E}^\kappa$ whose norm $f_\otimes$ is $\kappa$-compact is closed under pushouts and indexed coproducts. The proof for both is similar, and easier for indexed coproducts (simply recall that distributivity of the norms implies that norms of a finite indexed coproduct is a finite indexed coproduct of smaller norms) so we do the case of pushouts.
    
   To see this case, let $Y\simeq \colim_{\Lambda^2_0}X \simeq X_1\cup_{X_0}X_2$ be a pushout where each $X_i\in \E_H^\kappa$ is such that $f_\otimes X_i \in \E_G^\kappa$.   By $G$-distributivity of the $G$-symmetric monoidal structure on $\udl{\E}$, we thus get 
    \[f_{\otimes}Y \simeq \colim_{f_*\Lambda^2_0}f_{\otimes}X \in \E_G.\] Note that $f_{\otimes}X\colon f_*\Lambda^2_0\rightarrow \udl{\E}$ is a diagram of $\kappa$-compact objects: at level $G$, since $(f_*\Lambda^2_0)^G\simeq \Lambda^2_0$, this diagram is given by $f_{\otimes}X_1\leftarrow f_{\otimes}X_0\rightarrow f_{\otimes}X_2$, which we have assumed to be $\kappa$-compact; at level $K\lneq G$, it is a diagram of terms all of which are indexed coproducts of norms of the $X_i$'s, which are $\kappa$-compact by assumption on the $X_i$'s and the induction hypothesis (some of the terms are tensor products of norms of restrictions of $X_i$'s for smaller groups).  Now, by \cref{prop:finite_posets}, since $f_*\Lambda^2_0$ is a finite poset with $G$-action, the diagram $f_{\otimes}X$ can be re-expressed as an iterated finite colimit of $\kappa$-compact objects, whence $\kappa$-compactness of $f_{\otimes}Y$.

 Finally, since this collection of objects is also clearly closed under $\kappa$-small filtered colimits and contains $\res_H^G\udl{S}$, we see that it is the whole of $\res_H^G\udl{\E}^\kappa$, as was to be shown. 
\end{proof}

\begin{cor}\label{cor:closure_under_norms_of_compacts}
    Let $\kappa $ be a regular cardinal. The full subcategory $(\udl{\cat}\perfect)^{\kappa}\subseteq \udl{\cat}\perfect$ is closed under the multiplicative norms.
\end{cor}
\begin{proof}
    By \cite[Prop. 2.5.7]{kaifNoncommMotives}, $\{\myuline{\spectra}^{\omega},\udl{\func}(\udl{\Delta}^1,\myuline{\spectra}^{\omega})\}$ forms a set of $\omega$--compact generators of $\udl{\cat}\perfect$. Now, for $w\colon W\rightarrow V$ in $\finite_G$, by the equivalences
    \[w_{\otimes}\udl{\func}(\udl{\Delta}^1,\myuline{\spectra}^{\omega}) \simeq (w_{\otimes}\udl{\func}(\udl{\Delta}^1,\myuline{\spectra}))^{\omega}\simeq \udl{\func}(w_{*}\udl{\Delta}^1,\myuline{\spectra})^{\omega}\simeq \udl{\func}(w_*\udl{\Delta}^1,\myuline{\spectra}^{\omega})\]
    by virtue of the $G$-symmetric monoidal equivalence $\ind\colon \udl{\cat}\perfect\simeq \udl{\presentable}^L_{\omega,\mathrm{st}}\cocolon (-)^{\omega}$, the Ind and $(-)^{\omega}$ commutations from \cref{lem:omega_compactness_of_generators}, and the fact that $\udl{\func}(-,\myuline{\spectra})\colon \udl{\cat}\rightarrow \udl{\presentable}^L_{\mathrm{st}}$ refines to a $G$-symmetric monoidal functor by \cite[Cor. 2.2.20]{kaifNoncommMotives}, we thus see, again by \cref{lem:omega_compactness_of_generators}, that the set $\udl{S} \coloneq \{\udl{\func}(w_*\udl{\Delta}^1,\myuline{\spectra}^{\omega})\}_{w\in\udl{\func}(\udl{\Delta}^1,\udl{\finite})}$ forms a set of $\omega$--compact generators for $\udl{\cat}\perfect$ which is closed under multiplicative norms.
    We may thus specialise the previous corollary to the case $\udl{\E}=\udl{\cat}\perfect$,  $(\udl{\cat}\perfect)^{\kappa}\subseteq \udl{\cat}\perfect$ to conclude. 
\end{proof}

Finally, we show how to bootstrap \cref{prop:finite_posets} to  the case of compact Lie groups. While the present article deals only with finite groups, the following result will be needed in \cite{brankoKaifNatalie} when dealing with arbitrary compact Lie groups.

\begin{prop}
    Let $G$ be a compact Lie group and $\udl{I}\in\func(BG,\poset)$ be a finite poset with a continuous $G$-action. Then $\udl{I}\in\cat_G$ is a finite $G$-category.
\end{prop}
\begin{proof}
    Let $G^{\circ}\leq G$ be the normal subgroup of the identity component so that the quotient $Q\coloneqq G/G^{\circ}$ is a finite group. In particular, we have a fibre sequence $BG^{\circ}\rightarrow  BG\xrightarrow{q} BQ$ where $q$ is obtained by applying $(-)_{hQ}$ to the unique $Q$-equivariant map $BG^{\circ}\rightarrow\ast$. Note that the restriction map $q^*\colon \func(BQ,\poset)\rightarrow \func(BG,\poset)$ is an equivalence: this is  since $\Pi_1BG^{\circ}\simeq \ast$, and so we may write $q^*$ as the map
    \[\func(BQ,\poset)\simeq \func(\ast,\poset)^{hQ} \simeq \func(\Pi_1BG^{\circ},\poset)^{hQ}\xlongrightarrow{\simeq}\func(BG^{\circ},\poset)^{hQ}\simeq \func(BG,\poset)\] where the arrowed equivalence is since the $(\infty,1)$-category $\poset$ is in fact a $(1,1)$-category.

    Next, note that we have the pullback of $(\infty,1)$-categories on the left in the diagram
    \[
    \begin{tikzcd}
        BG\dar["q"] \rar[hook, "i"] \ar[dr, phantom , "\lrcorner"]& \orbit_G \dar["p"] \\
        BQ \rar[hook,"j"]& \orbit_Q
    \end{tikzcd}
    \hspace{10mm}
    \begin{tikzcd}
        \func(BG,\cat) \rar[hook, "i_*"]& \func(\orbit_G\op,\cat)  \\
        \func(BQ,\cat) \rar[hook,"j_*"]\uar["q^*",hook]& \func(\orbit_Q\op,\cat) \uar["p^*",hook]
    \end{tikzcd}
    \] where the functor $p\colon \orbit_G\rightarrow \orbit_Q$ is given by $(-)/G^{\circ}$ and this admits a fully faithful right adjoint $r$ given by restriction of actions along $G\twoheadrightarrow Q$. Because $p\colon \orbit_G\rightarrow \orbit_Q$ is a cartesian fibration, we obtain by basechange that the right-hand square commutes. Furthermore, note that $p^*$ sends $Q$-finite categories to $G$-finite categories. This is because it clearly preserves arbitrary colimits, and it also preserves inductions since we have an identification $p^*\simeq r_!\colon \func(\orbit_Q\op,\cat)\rightarrow \func(\orbit_G\op,\cat)$, and left Kan extensions compose.

    Altogether, by the first paragraph, we have a unique $\udl{J}\in \func(BQ,\cat)$ such that $q^*\udl{J}\simeq \udl{I}$, and so $\udl{J}$ also has finite underlying poset. Thus, by the second paragraph, we obtain that
    \[i_*\udl{I}\simeq i_*q^*\udl{J}\simeq p^*j_*\udl{J},\] which is a $G$-finite category since $j_*\udl{J}$ is a $Q$-finite category by \cref{prop:finite_posets}.
\end{proof}

\section{2-categorical enrichment of norms on presentable categories}\label{appendix:2cat}

We give a 2-categorical upgrade of the $G$-symmetric monoidal structure on the 1-category of $G$-presentable-stable $G$-categories, following the strategy of \cite[Prop. 2.3.3]{blansblomChainRule} in the nonequivariant situation.

\begin{nota}\label{nota:notation_for_2-cats}
    We write $\CAT$ and $\CAT_2$ for the (very very large) 1-category of very large 1-categories and very large $2$-categories, respectively. In particular, passing from a 2-category to its underlying 1-category assembles into a product-preserving functor $\forget\colon \CAT_2\rightarrow \CAT$. We also write $\boldTwocat_2$ for the (very very large) 2-category of very large $2$-categories. Furthermore, we write $\widehat{\boldCat}$ and $\boldPr$ to denote the 2-categories of large 1-categories and presentable categories, respectively.
\end{nota}

\begin{fact}\label{fact:2-categorical_facts}
    We collect various facts that we will need to give a 2-categorical enhancement of the $G$-symmetric monoidal structure on presentable-stable categories.
    \begin{enumerate}
        \item By \cite[Prop. A.2.9]{blansblomChainRule}, there is a 2-categorical enhancement of the straightening-unstraightening equivalence $\func(\spancategory(\finite_G),\CAT_2)\simeq \cocartCat_2(\spancategory(\finite_G))$, where here a functor $p\colon \X\rightarrow \sC$ with $\X$ a 2-category and $\sC$ a 1-category is said to be a 2-cocartesian fibration if for all $x,y, z\in \X$, the square 
        \[
        \begin{tikzcd}
            \func_{\X}(y,z) \rar\dar & \func_{\X}(x,z) \dar \\
            \map_{\sC}(py,pz) \rar & \map_{\sC}(px,pz)
        \end{tikzcd}
        \]
        in $\cat$ is a pullback.

        \item Let $w\colon W \rightarrow V$ be a map in $\finite_G$, and $\udl{\sC}\in \presentable^L_{W,\mathrm{st}}$. Then by Nardin's thesis \cite{nardinThesis} (cf.  \cite[Thm. 3.50 and Thm. 3.65]{brankoKaifNatalie} also for a recent treatment which filled in several gaps in the original proof) there exists an object $w_{\otimes}\udl{\sC}\in  \presentable^L_{V,\mathrm{st}}$ and an $w$-distributive $V$-functor $w_*\udl{\sC}\rightarrow w_{\otimes}\udl{\sC}$ which induces an equivalence 
        \[\udl{\func}\distributeAll{w}_V(w_*\udl{\sC},\udl{\D}) \xlongrightarrow{\simeq } \udl{\func}^L_V(w_{\otimes}\udl{\sC},\udl{\D})\] for all $\udl{\D}\in \presentable^L_{V,\mathrm{st}}$. Moreover, these universal distributive functors compose in the following sense: suppose we are given a composition 
        \[
        \begin{tikzcd}
            && P \ar[dd, phantom, "\rotatebox{315}{$\lrcorner$}"]\ar[dl, "\bar{v}"'] \ar[dr, "\bar{w}"] \\
            & W\ar[dl, "f"'] \ar[dr, "w"]&& S \ar[dl, "v"'] \ar[dr, "g"]\\
            U &&V&& T
        \end{tikzcd}
        \]
        of spans in $\finite_G$ and $\udl{\sC}\in\presentable^L_{U,\mathrm{st}}$. Then writing for instance $\mu_{f,w}\colon w_*f^*\udl{\sC}\rightarrow w_{\otimes}f^*\udl{\sC}$ for the universal distributive functor, the composition
        \[g_*v^*w_*f^*\udl{\sC}\xlongrightarrow{g_*v^*(\mu_{f,w})} g_*v^*w_{\otimes}f^*\udl{\sC}\xlongrightarrow{\mu_{v,g}} g_{\otimes}v^*w_{\otimes}f^*\udl{\sC}\] agrees with the universal distributive functor $\mu_{f\bar{v},g\bar{w}}\colon g_*\bar{w}_*\bar{v}^*f^*\udl{\sC}\rightarrow g_{\otimes}\bar{w}_{\otimes}\bar{v}^*f^*\udl{\sC}$. This is a straightforward check of universal properties using the Beck-Chevalley equivalences $v^*w_*\simeq \bar{w}_*\bar{v}^*$ and $v^*w_{\otimes}\simeq \bar{w}_{\otimes}\bar{v}^*$.
    \end{enumerate}
\end{fact}

\begin{lem}\label{lem:cartesian_2-symmon}
    The Mackey functor $\udl{\widehat{\cat}}^\times \colon \spancategory(\finite_G)\to {\CAT}$ classifying the  cartesian $G$-symmetric monoidal structure upgrades to an object in $\commmonoid_G(\CAT_2)$, i.e. there exists a finite-product-preserving dashed lift as in the diagram
    \[
    \begin{tikzcd}
         & \CAT_2 \dar["\forget"]\\
         \spancategory(\finite_G) \rar["\udl{\widehat{\cat}}^{\times}"] \ar[ur, "\udl{\widehat{\boldCat}}^{\times}", dashed] & \CAT
    \end{tikzcd}
    \] where the vertical map is \cref{nota:notation_for_2-cats}.
\end{lem}
\begin{proof}
Note that the restriction functoriality clearly has a lift to a 2-functor $\udl{\widehat{\cat}}_G \colon \finite_G\op\to \boldTwocat_2$. Using the $2$-category structure of the latter together with the universal property of span $2$-categories, we obtain a unique $2$-functor $$\spancategory_{1,5}(\finite_G)^{\mathrm{co}}\to \boldTwocat_2$$  extending the original one, and sending the backwards arrows to the right adjoints of the restriction functors, i.e. to indexed products. Here, $\spancategory_{1,5}$ denotes the usual span $2$-category\footnote{The subscript $1,5$ is meant to evoke that this is a $2$-category, but the $2$-morphisms between spans are themselves only allowed to go in one direction, where there are also span categories where the $2$-morphims are spans between spans.}, also called $2$-category of correspondences in \cite[Thm. 1.1.1]{stefanich}, and we use the universal property established in \textit{loc. cit.}.  We may then restrict this functor along the finite-product-preserving functor $\spancategory(\finite_G)\rightarrow \spancategory_{1,5}(\finite_G)^{\mathrm{co}}$ to obtain a finite-product-preserving 2-functor $\spancategory(\finite_G)\rightarrow  \boldTwocat_2$. But since $\spancategory(\finite_G)$ was a 1-category, this factors as a finite-product-preserving 1-functor $\spancategory(\finite_G)\rightarrow  \CAT_2$ which we denote as $\udl{\widehat{\boldCat}}^{\times}\in \cmonoid_G(\CAT_2)$. Finally, the statement that $\udl{\widehat{\boldCat}}^{\times}$ provides the dashed lift is an immediate consequence of the uniqueness of $G$-cartesian symmetric monoidal structures, cf. \cite[Thm. A']{stewartCommOperads}.
\end{proof}

\begin{prop}\label{prop:mackey2}
    The Mackey functor $\udl{\presentable}_{G,\mathrm{st}}^\otimes \colon \spancategory(\finite_G)\to {\CAT}$ classifying the $G$-symmetric monoidal structure  upgrades to an object in $\commmonoid_G(\CAT_2)$, i.e. there exists a finite-product-preserving dashed lift as in the diagram
    \[
    \begin{tikzcd}
         & \CAT_2 \dar["\forget"]\\
         \spancategory(\finite_G) \rar["\udl{{\presentable}}^{\otimes}_{G,\mathrm{st}}"] \ar[ur, "\udl{{\boldPr}}^{\otimes}_{G,\mathrm{st}}", dashed] & \CAT
    \end{tikzcd}
    \] where the vertical map is as in \cref{nota:notation_for_2-cats}.
\end{prop}
\begin{proof}
    The construction, which ultimately goes back to that of Lurie's \cite[\textsection 4.8.1]{lurieHA} in the nonequivariant situation, proceeds exactly as in \cite[\textsection 3.4]{nardinThesis} (cf. also \cite[\textsection 3.6 and \textsection 3.7]{brankoKaifNatalie} for a different and more highly structure approach) by replacing all mention of 1-cocartesian fibrations to 2-cocartesian fibrations.
    
    In more detail, by \cref{fact:2-categorical_facts} (1) and \cref{lem:cartesian_2-symmon}, we obtain a $2$-cocartesian fibration $p\colon \int\udl{\widehat{\boldCat}}^\times_G\to \spancategory(\finite_G)$. As in \cite[Prop. 2.3.3]{blansblomChainRule}, we can now consider the locally full subcategory (i.e. the 2-subcategory which induces fully faithful inclusions on the mapping 1-categories) $\int \udl{{\boldPr}}^{\otimes}_{G,\mathrm{st}}\subset \int\udl{\widehat{\boldCat}}^\times_G$ given as follows: an object over a finite set $V\in \finite_G$ is given by an object in $\presentable^L_{V,\mathrm{st}}$; a 1-morphism $\udl{\sC}\rightarrow \udl{\D}$ with $\udl{\sC}\in\presentable^L_{U,\mathrm{st}}$ and $\udl{\D}\in\presentable^L_{V,\mathrm{st}}$ lies over a morphism $U\xleftarrow{f} W \xrightarrow{w} V$ in $\spancategory(\finite_G)$ if and only if the associated morphism $w_*f^*\udl{\sC}\rightarrow\udl{\D}$ of $V$-categories is $w$-distributive. By construction, it is clear that the underlying 1-category $\forget(\int \udl{{\boldPr}}^{\otimes}_{G,\mathrm{st}})$ agrees with the 1-category $\int \udl{{\presentable}}^{\otimes}_{G,\mathrm{st}}$

    Now, by \cref{fact:2-categorical_facts} (2), the restricted functor $\int \udl{{\boldPr}}^{\otimes}_{G,\mathrm{st}}\rightarrow\spancategory(\finite_G)$ is a locally 2-cocartesian fibration in the sense of \cite[Def. A.2.11]{blansblomChainRule}. By \cref{fact:2-categorical_facts} (2) again and \cite[Lem. A.2.12]{blansblomChainRule}, it is in fact even a 2-cocartesian fibration (note however that the inclusion $\int \udl{{\boldPr}}^{\otimes}_{G,\mathrm{st}}\subset \int\udl{\widehat{\boldCat}}^\times_G$ does not preserve all cocartesian edges, only the inert ones). 

    Finally, by performing 2-straightening on this via \cref{fact:2-categorical_facts} (1), we obtain a functor $\udl{{\boldPr}}^{\otimes}_{G,\mathrm{st}}\colon \spancategory(\finite_G)\rightarrow \CAT_2$ which agrees with $\udl{{\presentable}}^{\otimes}_{G,\mathrm{st}}$ upon postcomposing with $\forget\colon \CAT_2\rightarrow \CAT$. Since $\forget\colon \CAT_2\rightarrow \CAT$ reflects products and $\udl{{\presentable}}^{\otimes}_{G,\mathrm{st}}\colon \spancategory(\finite_G)\rightarrow \CAT$ is known to preserve finite products already, we also get that $\udl{{\boldPr}}^{\otimes}_{G,\mathrm{st}}\colon \spancategory(\finite_G)\rightarrow \CAT_2$ preserves finite products. This completes the proof.
 \end{proof}

\printbibliography

\end{document}